\documentclass{amsart}
\usepackage[margin=3.9cm]{geometry}
\usepackage{amssymb}
\usepackage{graphicx} 
\usepackage{mathrsfs}
\usepackage{enumerate}
\usepackage{xspace}
\usepackage{color}
\usepackage{enumerate}
\usepackage{enumitem}
\usepackage{amsthm}
\usepackage{dsfont}
\usepackage{bbm}
\usepackage[colorlinks=true, linkcolor = blue, citecolor = blue]{hyperref}
\usepackage[numbers]{natbib}
\usepackage[backgroundcolor=white,bordercolor=red]{todonotes}
\DeclareMathAlphabet{\mathpzc}{OT1}{pzc}{m}{it}
\usepackage[nameinlink]{cleveref}
\numberwithin{equation}{section}
\begin{document}

\allowdisplaybreaks

\theoremstyle{plain}

\newtheorem{theorem}{Theorem}[section]
\newtheorem{lemma}[theorem]{Lemma}
\newtheorem{example}[theorem]{Example}
\newtheorem{proposition}[theorem]{Proposition}
\newtheorem{corollary}[theorem]{Corollary}
\newtheorem{definition}[theorem]{Definition}
\newtheorem{Ass}[theorem]{Assumption}
\newtheorem{condition}[theorem]{Condition}
\theoremstyle{definition}
\newtheorem{remark}[theorem]{Remark}
\newtheorem{SA}[theorem]{Standing Assumption}
\newtheorem*{discussion}{Discussion}

%Stochastic Intervals
\newcommand{\of}{[\hspace{-0.06cm}[}
\newcommand{\gs}{]\hspace{-0.06cm}]}

%Lebesgue Measure
\newcommand\llambda{{\mathchoice
		{\lambda\mkern-4.5mu{\raisebox{.4ex}{\scriptsize$\backslash$}}}
		{\lambda\mkern-4.83mu{\raisebox{.4ex}{\scriptsize$\backslash$}}}
		{\lambda\mkern-4.5mu{\raisebox{.2ex}{\footnotesize$\scriptscriptstyle\backslash$}}}
		{\lambda\mkern-5.0mu{\raisebox{.2ex}{\tiny$\scriptscriptstyle\backslash$}}}}}

%Indikator
\newcommand{\1}{\mathds{1}}

%Filtrations
\newcommand{\F}{\mathbf{F}}
\newcommand{\G}{\mathbf{G}}

\newcommand{\B}{\mathbf{B}}

%Martingale Measures
\newcommand{\M}{\mathcal{M}}

%Scalar Product
\newcommand{\la}{\langle}
\newcommand{\ra}{\rangle}

%Special Type of Quadratic Variation
\newcommand{\lle}{\langle\hspace{-0.085cm}\langle}
\newcommand{\rre}{\rangle\hspace{-0.085cm}\rangle}
\newcommand{\blle}{\Big\langle\hspace{-0.155cm}\Big\langle}
\newcommand{\brre}{\Big\rangle\hspace{-0.155cm}\Big\rangle}

%Coordinate Process
\newcommand{\X}{\mathsf{X}}

%Short Cuts
\newcommand{\tr}{\operatorname{tr}}
\newcommand{\N}{{\mathbb{N}}}
\newcommand{\cadlag}{c\`adl\`ag }
\newcommand{\on}{\operatorname}
\newcommand{\oP}{\overline{P}}
\newcommand{\oQ}{\overline{Q}}
\newcommand{\oO}{\mathcal{O}}
\newcommand{\D}{D(\mathbb{R}_+; \mathbb{R})}
\newcommand{\cT}{\mathcal{T}}
\newcommand{\cE}{\mathcal{E}}
\newcommand{\cC}{\mathcal{C}}
\newcommand{\usc}{\textit{USC}}
\newcommand{\cI}{\mathcal{I}}
\newcommand{\cB}{\mathcal{B}}
\newcommand{\C}{\mathsf{C}}
\newcommand{\cK}{\mathcal{K}}
\newcommand{\oconv}{\overline{\operatorname{co}}\hspace{0.1cm}}
\renewcommand{\B}{\mathbb{B}}
\newcommand{\m}{\mathbb{M}}
\newcommand{\K}{K}

%Typographical
\renewcommand{\epsilon}{\varepsilon}

%Semimartingale laws
\newcommand{\fPs}{\mathfrak{P}_{\textup{sem}}}
\newcommand{\fPss}{\mathfrak{P}_{\textup{sem}}^{\textup{sp}}}
\newcommand{\fSss}{\mathcal{S}^{\textup{sp}}}
\newcommand{\fPas}{\mathfrak{P}^{\textup{ac}}_{\textup{sp\hspace{0.04cm}sem}}}
\newcommand{\fSac}{\mathcal{S}^{\textup{sp}}_{\textup{ac}}}
\newcommand{\rrarrow}{\twoheadrightarrow}
\newcommand{\cA}{\mathcal{C}}
\newcommand{\cR}{\mathcal{R}}
\newcommand{\cQ}{\mathcal{Q}}
\newcommand{\cF}{\mathcal{F}}
\newcommand{\bth}{\overset{\leftarrow}\theta}
\renewcommand{\th}{\theta}
\newcommand{\A}{\mathsf{A}}
\newcommand{\z}{\mathfrak{z}}
\newcommand{\p}{\mathsf{P}}
\newcommand{\q}{\mathsf{Q}}
\renewcommand{\r}{\mathsf{R}}
%\DeclareMathOperator*{\gr}{gr}

%Basic Stuff
\newcommand{\bR}{\mathbb{R}}
\newcommand{\nnabla}{\nabla}
\newcommand{\f}{\mathfrak{f}}
\newcommand{\g}{\mathfrak{g}}
\newcommand{\Y}{\mathfrak{X}}
\newcommand{\bK}{\mathbb{K}}
\newcommand{\bC}{\mathbb{C}}
\newcommand{\tn}{|\!|\!|}
\newcommand{\bx}{\mathsf{x}}
\newcommand{\bX}{\mathsf{X}}
\newcommand{\cU}{\mathcal{U}}

\renewcommand{\emptyset}{\varnothing}

\makeatletter
\@namedef{subjclassname@2020}{%
	\textup{2020} Mathematics Subject Classification}
\makeatother

 \title{Stochastic Processes under Parameter Uncertainty} 
\author[D. Criens]{David Criens}
\address{Albert-Ludwigs University of Freiburg, Ernst-Zermelo-Str. 1, 79104 Freiburg, Germany}
\email{david.criens@stochastik.uni-freiburg.de}

\keywords{%\vspace{1ex} 
	nonlinear stochastic process; martingale problem; stochastic partial differential equations; sublinear expectation; nonlinear expectation; Knightian uncertainty}


\thanks{The author acknowledges financial support from the DFG project SCHM 2160/15-1.}
\date{\today}

\maketitle

\begin{abstract}
In this paper we study a family of nonlinear (conditional) expectations that can be understood as a stochastic process with uncertain parameters.
We develop a general framework which can be seen as a version of the martingale problem method of Stroock and Varadhan with parameter uncertainty.
To illustrate our methodology, we explain how it can be used to model nonlinear L\'evy processes in the sense of Neufeld and Nutz, and we introduce the new class of stochastic partial differential equations under parameter uncertainty. Moreover, we study properties of the nonlinear expectations. We prove the dynamic programming principle, i.e., the tower property, and we establish conditions for the (strong) \(\usc_b\)--Feller property and a strong Markov selection principle. 
\end{abstract}

%\tableofcontents

\section{Introduction}
We study conditional nonlinear expectations on the path space \(\Omega\) of continuous or \cadlag functions from \(\bR_+\) into a Polish space \(F\). More specifically, we are interested in nonlinear expectations of the form
\begin{align} \label{eq: upper exp}
\mathcal{E}_t (\psi) (\omega) = \sup_{P \in \cC(t, \omega)} E^P \big[ \psi \big],
\end{align}
where \(\psi\) is a suitable real-valued function on \(\Omega\) and \(\cC(t, \omega)\) is a set of laws of stochastic processes whose paths coincide with \(\omega \in \Omega\) till time \(t \in \bR_+\).
We think of the nonlinear conditional expectation \(\cE\) as a \emph{nonlinear stochastic process} or a {\em stochastic process under parameter uncertainty}.

The systematic study of nonlinear stochastic processes started with the seminal work of Peng \cite{peng2007g, peng2008multi} on the \(G\)-Brownian motion. 
More recently, larger classes of nonlinear semimartingales have been investigated in \cite{CN22a, CN22b, CN22c, fadina2019affine, hol16, hu2021g, K19, neufeld2017nonlinear}. 
At this point we also highlight the articles \cite{ElKa15, neufeld2014measurability, NVH} where abstract tools for the study were developed.
Next to the measure theoretic approach based on the nonlinear expectation \eqref{eq: upper exp}, nonlinear Markov processes have also been constructed by analytic methods via nonlinear semigroups, see \cite{denk2020semigroup,NMC,NR}. This construction is based on fundamental ideas of Nisio \cite{nisio}. 
In the recent paper \cite{K21}, it was shown that the two approaches are equivalent for the class of nonlinear L\'evy processes from \cite{denk2020semigroup,neufeld2017nonlinear}.
A key property of the construction via \eqref{eq: upper exp} is its relation to a class of stochastic processes given through the set~\(\cC\). This connection opens the door to analyze properties of \(\mathcal{E}\) with powerful tools from probability theory such as stochastic calculus. 

In this paper, we propose a new framework for the set \(\cC\) which extends the nonlinear expectation approach to nonlinear stochastic processes beyond semimartingales. 
The framework is inspired by the martingale problem method of Stroock and Varadhan~\cite{stroock2007multidimensional}. Consider the family 
\begin{align} \label{eq: test processes}
f (X) - \int_0^\cdot g (X_s) ds, \quad (f, g) \in A \subset C_b (F; \bR) \times C_b (F; \bR),
\end{align}
of test processes, where \(X\) is the coordinate process on \(\Omega\). The set \(A\) of test functions is often called a pregenerator. A probability measure \(P\) is said to be a solution to the martingale problem associated to \(A\) if all test processes in \eqref{eq: test processes} are \(P\)-martingales. Our idea is to incorporate uncertainty to the pregenerator \(A\).

Let \(U\) be a countable index set and let \(\{Y^u \colon u \in U\}\) be a family of \cadlag processes on~\(\Omega\), which serve as test processes for the martingale problem under uncertainty.  
If \(P\) is a probability measure such that \(Y^u\) is a \emph{special} \(P\)-semimartingale, then there exists a \(P\)-unique predictable \cadlag process \(\A^P (Y^u)\) of finite variation, starting in zero, such that \(Y^u - Y^u_0 - \A^P (Y^u)\) is a \(P\)-local martingale. In case the process in \eqref{eq: test processes} is a (local) \(P\)-martingale, the process \(f (X)\) is a special semimartingale and \(\A^P (f (X) ) = \int_0^\cdot g (X_s) ds\). Motivated by this observation, we define \(\cC\) to be the set of all probability measures \(P\) under that each test process \(Y^u\) is a special \(P\)-semimartingale with absolutely continuous compensator \(\A^P (Y^u)\) such that \((\llambda \otimes P)\)-a.e.
\[
\big( d \A^P (Y^u) / d \llambda \big)_{u \in U} \in \Theta.
\]
Here, \(\Theta\) is a time and path-dependent set-valued mapping on \(\bR_+ \times \Omega\), which captures the uncertainty. 

Our framework is tailor made for adding uncertainty to any class of stochastic processes which can be characterized by a martingale problem. For instance, this is the case for semimartingales (with absolutely continuous characteristics), solutions to stochastic partial differential equations and many Markov processes. In particular, it is possible to recover nonlinear semimartingale settings which are built from suitably parameterized absolutely continuous semimartingale characteristics.
To illustrate this, we explain how nonlinear L\'evy processes, in the sense of \cite{neufeld2017nonlinear}, can be modeled with our framework. 
Further, we introduce a novel class of nonlinear Markov processes which are not necessarily semimartingales: the class of \emph{nonlinear (in the sense of uncertainty) stochastic partial differential equations (NSPDE)}. 
We think this class is of specific interest and deserve further investigation.

The first main result in this paper is the \emph{dynamic programming principle (DPP)} for the nonlinear expectation \eqref{eq: upper exp}, i.e., the tower property
\[
\mathcal{E}_s (\mathcal{E}_t (\psi)) (\omega) = \mathcal{E}_s (\psi) (\omega), \quad \omega \in \Omega, \ \ s \leq t.
\]
In case the uncertainty set \(\Theta\) has a Markovian structure in the sense that \(\Theta (t, \omega)\) depends on \((t, \omega)\) only through the value \(\omega (t-)\), the DPP induces a \emph{nonlinear Markov property} given by
\begin{align} \label{eq: non MP}
\cE^x ( \cE^{X_t} (\psi (X_s))) = \cE^x (\psi (X_{t + s})), \qquad \cE^x (\psi) \triangleq \cE_0 (\psi) (\omega) \big|_{\omega (0) = x}.
\end{align}
As in the theory of (linear) Markov processes, there is a strong link to semigroups.
Indeed, in this Markovian case, the nonlinear Markov property ensures the semigroup property \( T_t T_s = T_{t + s}\) of the family \((T_t)_{t \geq 0}\) defined by 
\begin{align*} 
T_t (\psi)(x) \triangleq \cE^x (\psi (X_t)),
\end{align*}
where \(\psi\) runs through the class of bounded upper semianalytic functions. It is an important question when the semigroup \((T_t)_{t \geq 0}\) preserves some regularity. 
For a continuous path setting, we provide general conditions for the \emph{\(\usc_b\)--Feller property} of the semigroup \((T_t)_{t \geq 0}\), which means that it is a sublinear Markovian semigroup on the space of bounded upper semicontinuous functions. Further, we establish a \emph{strong Markov selection principle}, i.e., we prove that, for every bounded upper semicontinuous function \(\psi \colon F \to \bR\) and any time \(t > 0\), there exists a (time inhomogeneous) strong Markov family \(\{P_{(s, x)} \colon (s, x) \in \bR_+ \times F\}\) such that \(P_{0, x} \in \cC(0, x)\) and \[T_t (\psi) (x) = E^{P_{(0, x)}} \big[ \psi (X_t) \big].\] We stress that the nonlinear structure of our setting is reflected by the fact that the strong Markov selection \(\{P_{(s, x)} \colon (s, x) \in \bR_+ \times F\}\) depends on the input elements \(\psi\) and \(t\).

Thereafter, we investigate the Markovian class of NSPDEs in more detail.
We establish continuity and linear growth conditions for the \(\usc_b\)--Feller property and the strong Markov selection principle. Moreover, for the subclass of infinite-dimensional nonlinear Cauchy problems with drift uncertainty, we derive conditions for the \emph{strong} \(\usc_b\)--Feller property, which means that \(T_t (\usc_b) \subset C_b\) for all \(t > 0\). The strong \(\usc_b\)--Feller property can be seen as a \emph{smoothing property} of the sublinear semigroup \((T_t)_{t \geq 0}\). To the best of our knowledge, such an effect was first discovered in \cite{CN22a} for nonlinear one-dimensional diffusions.
We also emphasis that the strong \(\usc_b\)--Feller property entails the \(C_b\)--Feller property, i.e., \(T_t (C_b) \subset C_b\) for all \(t \in \bR_+\). This property is considered to be of fundamental importance for the study of sublinear Markovian semigroups via pointwise generators, see \cite[Section 2.4]{CN22b} for some comments. 

Let us now comment on some technical aspects of our proofs.
To establish the DPP, we invoke a general theorem from \cite{ElKa15} which requires that we check a measurable graph condition and stability under conditioning and concatenation. Hereby, we benefit from technical ideas developed in \cite{CN22a, neufeld2017nonlinear}.
In the context of classical optimal control, the relation of a value function and a nonlinear semigroup was discovered in \cite{nisio75}. For a nonlinear Markov process framework, the Markov property \eqref{eq: non MP} appeared, to the best of our knowledge, for the first time in the thesis \cite{hol16} for nonlinear Markovian semimartingales. We adapt the proof from \cite{hol16} to our setting.
For the \(\usc_b\)--Feller property, we rely on the theory of correspondences and, for the strong Markov selection principle, we invoke a technique from \cite{nicole1987compactification,hausmann86} for controlled diffusions, which is based on ideas of Krylov \cite{krylov1973selection} and Stroock and Varadhan \cite{stroock2007multidimensional} about Markovian selection. Again, we also benefit from technical ideas developed in \cite{CN22b} for one-dimensional nonlinear diffusions.
The strong \(\usc_b\)--Feller property for nonlinear Cauchy problems is proved by means of a strong Feller selection principle, which adapts ideas from \cite{CN22b, CN22c,stroock2007multidimensional} for finite-dimensional (nonlinear) diffusions to an infinite-dimensional non-semimartingale setting.

We end this section with a summary of the structure of this paper.
In Section \ref{sec: 1}, we introduce our framework and formulate the DPP. Thereafter, in Section~\ref{sec: 3.2}, we comment on nonlinear L\'evy processes and we introduce the new class of NSPDEs. 
We discuss nonlinear Markov processes in Section~\ref{sec: 2}. In particular, in Section~\ref{sec: 3.3}, we establish the \(\usc_b\)--Feller property and, in Section \ref{sec: 3.4}, we provide the strong Markov selection principle. In Section \ref{sec: NSPDE}, these results are tailored to the special case of NSPDEs  and the strong \(\usc_b\)--Feller property for nonlinear Cauchy problems with drift uncertainty is established. 
The remaining sections of this paper are devoted to the proofs of our results. The location of each proof is indicated before the statement of the respective result.  

%\newpage 

\section{Nonlinear Processes and the Dynamic Programming Principle} \label{sec: 1}
\subsection{The Ingredients}\label{sec: setting}
Let \(F\) be a Polish space and define $\Omega$ to be either the space of all \cadlag functions \(\bR_+ \to F\) or the space of all continuous functions \(\mathbb{R}_+ \to F\) endowed with the Skorokhod \(J_1\) topology.\footnote{When restricted to the space of continuous functions \(\bR_+ \to F\), the Skorokhod \(J_1\) topology coincides with the local uniform topology.}
The canonical process on $\Omega$ is denoted by \(X\), i.e., \(X_t (\omega) = \omega (t)\) for \(\omega \in \Omega\) and \(t \in \mathbb{R}_+\). 
It is well-known that \(\mathcal{F} \triangleq \mathcal{B}(\Omega) = \sigma (X_t, t \geq 0)\).
We define $\F \triangleq (\mathcal{F}_t)_{t \geq 0}$ to be the canonical filtration generated by $X$, i.e., \(\mathcal{F}_t \triangleq \sigma (X_s, s \leq t)\) for \(t \in \mathbb{R}_+\). 
To lighten our notation, for two stopping times \(S\) and \(T\), we also define the stochastic interval
\[
\of S, T \of \hspace{0.1cm} := \{ (t,\omega) \in \bR_+ \times \Omega \colon S(\omega) \leq t < T(\omega) \}.
\]
The stochastic intervals \( \gs S, T \of, \of S, T \gs, \gs S, T \gs  \) are defined accordingly.
In particular, \(\of 0, \infty\of \hspace{0.1cm} = \bR_+ \times \Omega\).
The set of all probability measures on \((\Omega, \mathcal{F})\) is denoted by \(\mathfrak{P}(\Omega)\) and endowed with the usual topology of convergence in distribution, i.e., the weak topology. The shift operator on \(\Omega\) is denoted by \(\theta = (\theta_t)_{t \geq 0}\), i.e., \(\theta_t (\omega) = \omega (\, \cdot + t)\) for all \(\omega \in \Omega\) and~\(t \in \mathbb{R}_+\). Let \(\bK\) be either the real line \(\bR\) or the complex plane \(\bC\).
Further, let \(U\) be a countable index set and, for every \(u \in U\), let \(Y^u\) be a \(\bK\)-valued \cadlag \(\F\)-adapted process on \(\Omega\) such that \(Y^u_{\cdot + t} = Y^u \circ \theta_t\) for all \(t \in \mathbb{R}_+\). 
We endow \(\bK^U\) with the product Euclidean topology.
Finally, let \(\Theta \colon \mathbb{R}_+ \times \Omega \twoheadrightarrow \bK^U\) be a correspondence, i.e., a set-valued mapping.
\begin{SA} \label{SA: meas gr}
The correspondence \(\Theta\) has a measurable graph, i.e., the graph
\[
\on{gr} \Theta = \big\{ (t, \omega, g) \in \mathbb{R}_+ \times \Omega \times \bK^U \colon g \in \Theta (t, \omega) \big\}
\]
is Borel. 
\end{SA}

Recall that a \(\bK\)-valued \cadlag adapted process \(Y = (Y_t)_{t \geq 0}\) is said to be a \emph{special semimartingale} if \(Y = Y_0 + M + A\), where \(M\) is a local martingale and \(A\) is a predictable process of (locally) finite variation, both starting in zero. The decomposition is unique up to indistinguishability.
A \(\bK\)-valued \cadlag process \(Y\) is called a \emph{special semimartingale after a time \(t^* \in \mathbb{R}_+\)} if the process \(Y_{\cdot + t^*} = (Y_{t + t^*})_{t \geq 0}\) is a special semimartingale for the right-continuous natural filtration of \(X_{\cdot + t^*} = (X_{t + t^*})_{t \geq 0}\).
For \(P \in \mathfrak{P}(\Omega)\), the set of special semimartingales after time \(t^*\) on the probability space \((\Omega, \cF, P)\) is denoted by \(\fSss(t^*, P)\). We also write \(\fSss (P) \triangleq \fSss (0, P)\).
For a given measure \(P\) and a given process \(Y \in \fSss (t^*, P)\), denote by \(\A^P (Y_{\cdot + t^*})\) the predictable part of (locally) finite variation in the semimartingale decomposition of \(Y_{\cdot + t^*}\), and denote by \(\fSac (t^*, P)\) the set of all \(Y \in \fSss(t^*, P)\) such that \(\A^P (Y_{\cdot + t^*})\) is absolutely continuous w.r.t. to the Lebesgue measure~\(\llambda\). Again, we write \(\fSac (P) \triangleq \fSac (0, P)\). For \(\omega, \omega' \in \Omega\) and \(t \in \mathbb{R}_+\), we define the concatenation
\[
\omega \otimes_t \omega' \triangleq  \omega \1_{[ 0, t)} + (\omega (t) + \omega' - \omega' (t))\1_{[t, \infty)},
\]
and, finally, for $(t,\omega) \in \of 0, \infty\of$, we define $\cA(t,\omega) \subset \mathfrak{P}(\Omega)$ by
\begin{align*}
\cA(t,\omega) \triangleq \big\{ P \in \mathfrak{P} (\Omega) \colon P&(X = \omega \text{ on } [0, t]) = 1, \forall_{u \in U} \ Y^u \in \fSac (t, P), \\&
 (\llambda \otimes P)\text{-a.e. } (d \A^P (Y^u_{\cdot + t})/d \llambda)_{u \in U} \in \Theta (\, \cdot + t, \omega \otimes_t X) \big\}.
\end{align*}
\begin{SA} \label{SA: non empty}
\(\cA (t, \omega) \not = \emptyset\) for all \((t, \omega) \in \of 0, \infty\of\).
\end{SA}

\begin{remark}
	\quad
	\begin{enumerate}
		\item[\textup{(i)}] For every \((t, \omega) \in \of 0, \infty\of\), the set $\cA(t, \omega)$ depends only on the path $\omega$ up to time $t$.
	\item[\textup{(ii)}] Notice that
	\begin{align*}
	\cA(t,\omega) = \big\{ P \in  \mathfrak{P} (\Omega) \colon P&(X = \omega \text{ on } [0, t]) = 1, \forall_{u \in U} \ Y^u \in \fSac (t, P), \\& 
	(\llambda \otimes P)\text{-a.e. } (d \A^P (Y^u_{\cdot + t})/d \llambda)_{u \in U} \in \Theta (\, \cdot + t, X) \big\}.
	\end{align*}
	We define \(\cC (t, \omega)\) with \(\Theta (\, \cdot + t, \omega \otimes_t X)\), instead of \(\Theta (\, \cdot + t, X)\), to prevent confusion about measurability. Namely, both \((d \A^P (Y^u_{\cdot + t})/d \llambda)_{u \in U}\) and \(\Theta (\, \cdot + t, \omega \otimes_t X)\) are measurable w.r.t. \(\sigma (X_s, s \geq t)\), while \(\Theta (\, \cdot + t, X)\) might also depend on~\(X_s, s < t\).
	\end{enumerate}
\end{remark}

\begin{discussion}[Relation to Martingale Problems]
The idea behind the set \(\cC\) is inspired by the fact that many classes of stochastic processes can be modeled via martingale problems. Take a set \(A \subset C_b(F; \bR) \times B(F; \bR)\), which is often called \emph{pregenerator}. Here, \(B(F; \bR)\) denotes the set of all bounded Borel functions \(F \to \bR\). Following the monograph \cite{EK}, a probability measure \(P\) on \((\Omega, \cF)\) is said to solve the martingale problem associated to \(A\) if all processes of the form
\begin{align}\label{eq: test martingale}
f (X) - \int_0^\cdot g (X_s) ds, \quad (f, g) \subset A, 
\end{align}
are \(P\)-martingales. To see the connection to the set \(\cC\), suppose that \(P\) solves the martingale problem associated to \(A\). Then, for every \((f, g) \in A\), the process \(f (X)\) is a special \(P\)-semimartingale, as it decomposes into the \(P\)-martingale given in \eqref{eq: test martingale} and the continuous (hence, predictable) finite variation process
\(
\int_0^\cdot g (X_s) ds,
\)
which, in particular, means that \(d \A^P ( f (X) ) / d \llambda = g (X)\). Roughly speaking, the set \(\cC\) mimics the idea of the martingale problem but it allows some part of the pregenerator, namely the density of the compensator, to take values in the set \(\Theta\), which incorporates uncertainty. 
For the important special case where \(\Theta\) is generated by a controlled coefficient, this interpretation is made explicit by the following proposition.

\begin{proposition} \label{prop: structure of R}
Let \(G\) be a metrizable Souslin space\footnote{That is a metrizable space that is the continuous image of a Polish space. Any analytic subspace of a Polish space is a Souslin space.} and assume that there exists a predictable Carath\'eodory function \(\z = (\z^u)_{u \in U} \colon G \times \of 0, \infty\of \, \to \mathbb{K}^U\), i.e., \(\z\) is continuous in the first and predictable in the second variable, such that 
\[
\Theta (t, \omega) = \big\{ \z (g, t, \omega) \colon g \in G \big\}, \qquad (t, \omega) \in \of 0, \infty\of.
\]
Here, we still assume that Standing Assumption~\ref{SA: meas gr} holds for \(\Theta\), i.e., that \(\on{gr} \Theta\) is measurable.
For \(t \in \bR_+\), let \(\mathscr{P}_t\) be the predictable \(\sigma\)-field on \(\of 0, \infty\of\) corresponding to the filtration generated by \(X_{\cdot + t}\). 
Then, for every \((t, \omega) \in \of 0, \infty\of\) and \(P \in \cC(t, \omega)\), there exists a \(\mathscr{P}_{t}\)-measurable function \(\g \colon \of 0, \infty\of\, \to G\) such that, for all \(u \in U\), the process 
\[
Y^u_{\cdot +t} - Y^u_{t} - \int_0^\cdot \z^u (\g (s, X), s + t, \omega \otimes_{t} X) ds 
\]
is a \(P\)-local martingale for the right-continuous natural filtration of \(X_{\cdot + t}\). 
\end{proposition}
\begin{proof}
	For \((s, \alpha) \in \of 0, \infty\of\), we set 
	\(
	\alpha^{s - } \triangleq \alpha \1_{[0, s)} + \alpha (s-) \1_{[s, \infty)}
	\) and
		\begin{align*}
			H \triangleq \Big\{ (s, \alpha) \in \of 0, \infty \of \colon ( s + t, \omega \otimes_t \alpha^{(s + t)-}, ( d\A^P_{s} (Y^u_{\cdot + t}) (\alpha)/ d \llambda )_{u \in U}) \not \in \on{gr} \Theta \Big\}. 
		\end{align*}
	By Standing Assumption~\ref{SA: meas gr}, the graph \(\on{gr} \Theta\) is measurable and therefore, \(H\in \mathcal{B}(\of 0, \infty\of)\). Further, we get from \cite[Theorem~IV.97]{DM78} that \(H \in \mathscr{P}_t\), that the function \((g, s, \alpha) \mapsto \z (g, s + t, \omega \otimes_t \alpha)\) is continuous in the \(G\)-variable and \(\mathscr{P}_t\)-measurable in the \(\of 0, \infty\of\)-variable and that
	\begin{align} \label{eq: H eq}
		H = \Big\{ (s, \alpha) \in \of 0, \infty \of \colon ( s + t, \omega \otimes_t \alpha, ( d\A^P_{s} (Y^u_{\cdot + t}) (\alpha)/ d \llambda )_{u \in U}) \not \in \on{gr} \Theta \Big\}.
	\end{align}
		For some arbitrary, but fixed, \(g_0 \in G\), we set 
		\[
		\pi (s, \alpha) \triangleq \begin{cases} \z (g_0, s + t, \omega \otimes_t \alpha) ,& \text{if } (s, \alpha) \in H,\\
			( d\A^P_{s} (Y^u_{\cdot + t}) (\alpha)/ d \llambda )_{u \in U},& \text{if } (s, \alpha) \not \in H, \end{cases}
		\]
		which is a \(\mathscr{P}_t\)-measurable map such that \(\pi (s, \alpha) \in \Theta (s + t, \omega \otimes_t \alpha)\) for all \((s, \alpha) \in \of 0, \infty\of\).
		As \(\llambda \otimes P\) is \(\sigma\)-finite on \((\of 0, \infty\of, \mathscr{P}_t)\) and \(G\) is a metrizable Souslin space, 
		the measurable implicit function theorem \cite[Theorem 7.2]{himmelberg} yields the existence of a \(\mathscr{P}_t\)-measurable function \(\g \colon \of 0, \infty\of \, \to G\) such that \(
	\pi (s, \alpha) = \z (\g (s, \alpha), s + t, \omega \otimes_t \alpha)\) for \((\llambda \otimes P)\)-a.a. \((s, \alpha) \in \of 0, \infty\of\). As \(P \in \cC(t, \omega)\), by virtue of \eqref{eq: H eq}, we have \((\llambda \otimes P)\)-a.e. \(\pi = (d\A^P (Y^u_{\cdot + t})/d \llambda)_{u \in U}\) and 
	the claim follows. 
\end{proof} 

	For general martingale problems, the pregenerator \(A\) is not necessarily a countable set. In the diffusion-type setting of Stroock and Varadhan \cite{stroock2007multidimensional}, for example, the pregenertor \(A\) is given by 
	\[
	A = \Big\{ (f, g) \colon f \in C^2_c (\mathbb{R}^d; \bR), \ g = \langle b, \nabla f \rangle + \tfrac{1}{2} \on{tr} \big[ \sigma \sigma^* \nabla^2 f \big] \Big\}, 
	\]
	where \(b \colon \bR^d \to \bR^d\) and \(\sigma \colon \bR^d \to \bR^{d \times r}\) are Borel measurable coefficients. In the definition of the set~\(\cC\) we take only countably many test functions into account. In typical situations this is no restriction, as it is often possible to pass to a countable subset of the pregenerator that determines the same martingale problem. 
	For example, in the diffusion-type setting explained above, it is sufficient to consider the set 
	\[
		A = \Big\{ (f, g) \colon f (x) = x^k_i x_j^k, \ k = 0, 1, \, i, j = 1, \dots, d, \ \ g = \langle b, \nabla f \rangle + \tfrac{1}{2} \on{tr} \big[ \sigma \sigma^* \nabla^2 f \big] \Big\}, 
	\]
	cf. \cite[Proposition~5.4.6]{KaraShre}.
	We also refer to \cite[Proposition~4.3.1]{EK} for an abstract result in this direction. In Section~\ref{sec: nonlinear levy} below, we will also discuss a way to characterize the class of nonlinear L\'evy processes as introduced in \cite{neufeld2017nonlinear} with a countable pregenerator under parameter uncertainty. 
	
	Let us also comment on the question to what extend the pregenerator has to determine the martingale problem uniquely. In the classical theory on Feller--Dynkin processes, the (pre)generators characterize the law of the processes in a unique manner, cf., for example, \cite[Theorem 3.33]{liggett10}.
	In the classical theory on martingale problems (as in \cite{EK,stroock2007multidimensional}, for instance), such a uniqueness property is an important question rather than intrinsically given. In particular, there is interest in martingale problems that are not well-posed (i.e., do not have unique solutions for each deterministic initial value). Indeed, by the strong Markov selection principle (see Chapter~12 in \cite{stroock2007multidimensional}), many non-well-posed martingale problems give rise to strong Markov families which can be selected from the set of solutions. 
	Our approach is fully independent of a uniqueness property, i.e., we can also introduce uncertainty to martingale problems that have infinitely many solutions (by convexity, a martingale problem has none, one or infinitely many solutions). 
	In fact, considering martingale problems with infinitely many solutions can lead to interesting effects. We explain an example of such an effect in Remark \ref{rem: SV example} below. 

\end{discussion}

%-------------------------------
%-------------------------------

\subsection{Dynamic Programming Principle} \label{sec: DPP}
We now define a nonlinear conditional expectation and we provide the corresponding dynamic programming principle (DPP), i.e., the tower property of the nonlinear conditional expectation.
Suppose that \(\psi \colon \Omega \to [- \infty, \infty]\) is an upper semianalytic function, i.e., the set \(\{\omega \in \Omega \colon \psi (\omega) > c\}\) is analytic for every \(c \in \mathbb{R}\),
and define the \emph{value function} \(v\) by
\[
v (t, \omega) \triangleq \sup_{P \in \cA (t, \omega)} E^P \big[ \psi \big], \quad (t, \omega) \in \of 0, \infty\of.
\]
The following theorem is proved in Section \ref{sec: pf DPP}.
\begin{theorem}[Dynamic Programming Principle] \label{theo: DPP}
The value function \(v\) is upper semianalytic. Moreover, for every \((t, \omega) \in \of 0, \infty\of \) and every stopping time \(\tau\) with \(t \leq \tau < \infty\), we have 
\begin{align} \label{eq: DPP}
v(t, \omega) = \sup_{P \in \cA (t, \omega)} E^P \big[ v (\tau, X) \big]. 
\end{align}
\end{theorem}

The value function \(v\) induces a nonlinear expectation \(\mathcal{E}\) via the formula
\[
\mathcal{E}_t (\psi) (\omega) \triangleq v (t, \omega), \qquad (t, \omega) \in \of 0, \infty\of,
\]
and the DPP provides its tower property, i.e., \eqref{eq: DPP} means that 
\[
\mathcal{E}_t (\psi) = \mathcal{E}_t (\mathcal{E}_\tau (\psi)), \quad t \leq \tau \leq \infty.
\]
By the pathwise structure of this property, it follows also immediately that 
\[
\mathcal{E}_{\rho} (\psi) = \mathcal{E}_{\rho} (\mathcal{E}_\tau (\psi))
\]
for all finite stopping times \(0 \leq \rho \leq \tau < \infty\).

For \(x \in F\), we define 
\begin{align*}
\mathcal{R} (x) \triangleq \big\{ P \in \mathfrak{P} (\Omega) \colon &P(X_0 = x) = 1, \forall_{u \in U} \ Y^u \in \fSac (P), \\& \hspace{1cm}
 (\llambda \otimes P)\text{-a.e. } (d\A^P (Y^u)/d \llambda)_{u \in U} \in \Theta \big\},
\end{align*}
and 
\[
\mathcal{E}^x (\psi) \triangleq \sup_{P \in \cR (x)} E^P \big[ \psi \big].
\]
Thinking of classical martingale problems, we interpret \(\{\mathcal{E}^x \colon x \in F\}\) as a \emph{nonlinear stochastic process}. In Section \ref{sec: 2}, we specify our setting to a Markovian situation and we establish more properties of \(\{\cE^x \colon x \in F\}\) to justify our interpretation.
Before we start this program, let us discuss some important examples.

\section{Examples} \label{sec: 3.2}
In this section we explain how our framework can be used to model \emph{nonlinear L\'evy processes} in the sense of \cite{neufeld2017nonlinear} and we introduce some new classes of nonlinear processes, namely the class of \emph{nonlinear Markov chains} and the class of \emph{nonlinear (in the sense of uncertainty) stochastic partial differential equations~(NSPDEs)}. 

\subsection{Nonlinear L\'evy processes} \label{sec: nonlinear levy}
Nonlinear L\'evy processes have been introduced in \cite{neufeld2017nonlinear} via an uncertain set of L\'evy--Khinchine triplets. We now explain how such nonlinear processes can be constructed in our framework. Let us emphasis that the discussion can be extended to more general classes of semimartingales under parameter uncertainty. 

To fix the basic setting, we consider \(F \triangleq \mathbb{R}^d, \mathbb{K} \triangleq \mathbb{C}, U \triangleq \mathbb{Q}^d\) and we take \(\Omega\) to be the space of all \cadlag functions \(\bR_+ \to F = \bR^d\). 
Let \(\mathbb{S}^d\) be the space of symmetric non-negative definite real-valued \(d \times d\) matrices. Further, let \(\mathcal{L}\) be the space of all L\'evy measure on \(\bR^d\), i.e., the space of all measures \(\K\) on \((\bR^d, \mathcal{B}(\bR^d))\) such that 
\[\K (\{0\}) = 0 \ \text{ and } \ \int (1 \wedge \|x\|^2) \K (dx) < \infty.\] 
As in \cite{neufeld2014measurability,neufeld2017nonlinear}, we endow \(\mathcal{L}\) with the weakest topology under which all maps 
\[\K \mapsto \int f(x) (1 \wedge \|x\|^2) \K (dx), \quad f \in C_b (\bR^d; \bR),\] 
are continuous. By Lemma \ref{lem: Levy polish} below, the space \(\mathcal{L}\) is Polish.
Let \(\widetilde{\Theta}\subset \bR^d \times \mathbb{S}^d \times \mathcal{L}\) be an analytic subset, where the product space is endowed with the product topology (and \(\mathbb{S}^d\) and \(\bR^d\) are endowed with their usual topologies), such that
\begin{align} \label{eq: bound Levy triplet}
\sup_{ (b, c, \K) \in \widetilde{\Theta}}  \Big( \|b\| + \|c\| + \int ( 1 \wedge \|x\|^2) \K (dx) \Big) < \infty.
\end{align}
For every \(u \in \mathbb{Q}^d\), we set \(Y^u \triangleq e^{i \langle u,  X\rangle}\) and we define
\[
A^u (b, c, \K) \triangleq - i \langle u, b \rangle + \frac{1}{2} \langle u, c u\rangle - \int \big( e^{i \langle u, x\rangle} - 1 - i \langle u, h (x)\rangle \big) \K (dx),
\]
and 
\begin{align*}
\Theta (t, \omega) \triangleq \big\{ \big(e^{i \langle u, \omega (t-)\rangle} A^u (b, c, \K)\big)_{u \in \mathbb{Q}^d} \colon (b, c, \K) \in \widetilde{\Theta}\big\}, \quad (t, \omega) \in \of 0, \infty\of,
\end{align*}
where \(h \colon \bR^d \to \bR^d\) is a fixed continuous truncation function and \(i\) denotes the imaginary number. We also assume that Standing Assumption~\ref{SA: meas gr} holds, which is for instance the case when \(\widetilde{\Theta}\) is compact (by \cite[Lemma~2.12]{CN22a}).

We now relate this setting to those from \cite{neufeld2017nonlinear}. First, for a given \(P \in \cR (x_0)\), we show that the coordinate process \(X\) is an \(\mathbb{R}^d\)-valued \(P\)-semimartingale whose characteristics \((B^P, C^P, \nu^P)\) are absolutely continuous with densities \((b^P, c^P, \K^P)\) such that \((\llambda \otimes P)\)-a.e. \((b^P, c^P, \K^P) \in \widetilde{\Theta}\). 
Take \(P \in \cR (x_0)\). For every \(k \in \{1, \dots, d\}\) and \(n \in \mathbb{N}\), the process \(\sin ( X^{(k)} / n)\) is a \(P\)-semimartingale. Here, \(X^{(k)}\) denotes the \(k\)-th coordinate of \(X\).
There exists a function \(f \in C^2 (\bR; \bR)\) such that \(f ( \sin (x) ) = x\) for all \(|x| \leq 1/2\). Hence, for \[\tau_n \triangleq \inf \big\{t \geq 0 \colon |X^{(k)}_t| > n / 2\big\},\] the process \(X^{(k)} = n f (\sin (X^{(k)} / n))\) is a \(P\)-semimartingale on \(\of 0, \tau_n \of\). Consequently, by part c) of \cite[Proposition I.4.25]{JS}, it follows that \(X^{(k)}\) is a \(P\)-semimartingale.
Of course, this means that \(X\) is an \(\bR^d\)-valued \(P\)-semimartingale.
For every \(u \in \bR^d\), by Lemma~\ref{lem: continuity in L} below, the map 
\begin{align*} 
\bR^d \times \mathbb{S}^d \times \mathcal{L} \ni (b, c, \K) \mapsto A^u (b, c, \K) \in \mathbb{C}
\end{align*}
is continuous. 
Using this observation, we deduce from Proposition~\ref{prop: structure of R} that there exists a predictable function \(\g = \g (P) \colon \of 0, \infty\of \, \to \widetilde{\Theta}\) such that 
\begin{align*}
e^{i \langle u, X\rangle} - e^{i \langle u, x_0 \rangle} - \int_0^\cdot e^{i \langle u, X_s \rangle} A^u ( \g (s, X) ) ds, \quad u \in \mathbb{Q}^d, 
\end{align*}
are \(P\)-local martingales. 
For every \(u \in \mathbb{Q}^d\), we define 
\[
\mathfrak{A}^P (u) \triangleq - i \langle u, B^P \rangle + \tfrac{1}{2} \langle u, C^P u \rangle - \int \big( e^{i \langle u, x\rangle} - 1- i \langle u, h(x)\rangle \big) \nu^P([0, \cdot \hspace{0.05cm}] \times dx).
\]
By virtue of \cite[Theorem II.2.42]{JS}, we get that \(P\)-a.s.
\begin{align} \label{eq: equality for rationals}
 \mathfrak{A}^P (u) = \int_0^\cdot A^u ( \g (s, X) ) ds, \quad u \in \mathbb{Q}^d.
\end{align}
It is clear that \(P\)-a.s. the map \(u \mapsto \mathfrak{A}^P_t (u)\) is continuous for every \(t \in \bR_+\). Further, thanks to \eqref{eq: bound Levy triplet}, the dominated convergence theorem implies that \(u \mapsto \int_0^t A^u (\g (s, X)) ds\) is continuous for every \(t \in \bR_+\). Hence, the equality in \eqref{eq: equality for rationals} even holds for all \(u \in \bR^d\) and we may use the uniqueness lemma of Gnedenko and Kolmogorov (\cite[Lemma~II.2.44]{JS}) to conclude that 
\[
P\text{-a.s. } (B^P, C^P, \nu^P) \ll \llambda \text{ with \((\llambda \otimes P)\)-a.e. } (b^P, c^P, F^P) = \g (\, \cdot\,, X).
\]
This completes the proof of the first direction.

Conversely, if \(P\) is the law of a semimartingale (for its canonical filtration), starting at \(x_0 \in \mathbb{R}^d\), with absolutely continuous characteristics whose densities \((b^P, c^P, F^P)\) are such that \((\llambda \otimes P)\)-a.e. \((b^P, c^P, F^P) \in \widetilde{\Theta}\), then \cite[Theorem II.2.42]{JS} yields that \(P \in \cR(x_0)\).

\subsection{Nonlinear Markov Chains} \label{sec: nonlinear MC}
A (continuous-time) Markov chain is a strong Markov process on some countable discrete state space \(F\). Typically, Markov chains are modeled via \(Q\)-matrices, which encode their infinitesimal dynamics. In particular, provided the Markov chain is a Feller--Dynkin process, the \(Q\)-matrix corresponds to the generator of the chain (see \cite{ReuterRiley}). In the following we adapt the underlying martingale problem (see, e.g., \cite[Theorem IV.20.6]{RogersWilliams} or \cite[Lemma 2.4]{YinZhang}) to a nonlinear setting. For simplicity, we will restrict our attention to nonlinear Markov chains with a \emph{finite} state space~\(F\).

Let \(U \triangleq F \triangleq \{1, \dots, N\}\) with \(N \in \mathbb{N}\), let \(\Omega\) be the space of all \cadlag functions \(\bR_+ \to F\), and set \( Y^u = \1_{\{X = u\}}\) for \(u \in U\). 
Denote by \(\mathbb{M}_q\) the set of all \(Q\)-matrices on~\(F\), i.e., the set of all matrices \(Q = (q (k, n))_{k, n = 1}^N \subset \bR^{N \times N}\) such that \(q (k, n) \geq 0\) for all \(k \not = n\) and \(Q \1= 0\).
Let \(G\) be a set, let \(Q = (q (k, n))_{k, n = 1}^N \colon G \to \mathbb{M}_q\) be a map and define, for~\((t, \omega) \in \of 0, \infty\of\), 
\[
\Theta (t, \omega) \triangleq \big\{ (q (\omega (t-), u) (g) )_{u \in U} \colon g \in G\big\}.
\]
In case \(G\) is the singleton \(\{g_0\}\), then \(\cR(x_0)\) is also a singleton whose only element is the law of the continuous-time Markov chain with starting value \(x_0\) and \(Q\)-matrix \(Q (g_0)\), see, e.g., \cite[Lemma 2.4]{YinZhang}.
By virtue of Proposition \ref{prop: structure of R}, this observation confirms our interpretation as a nonlinear Markov chain, i.e., a Markov chain with uncertain \(Q\)-matrix.

\subsection{Nonlinear SPDEs} \label{sec: nonlinear SPDE}
Let \((H_1, \|\cdot\|_{H_1}, \langle \, \cdot, \cdot \,\rangle_{H_1})\) and \((H_2, \|\cdot\|_{H_2}, \langle \, \cdot, \cdot\, \rangle_{H_2})\) be two separable Hilbert spaces over the real numbers, and let \((A, D(A))\) be the generator of a \(C_0\)-semigroup on \(H_2\). In the following we introduce drift and volatility uncertainty to the class of (semilinear) stochastic partial differential equations (SPDEs) of the type 
\[
d Y_t = A Y_t dt + \mu (Y_t) dt + \sigma (Y_t) d W_t, 
\]
where \(W\) is a cylindrical Brownian motion over \(H_1\) and \(\mu\) and \(\sigma\) are suitable measurable coefficients. 
To define our setting, we rely on the cylindrical martingale problem associated to semilinear SPDEs, see, e.g., \cite{criens20,criens21} for more details.

We denote the set of linear bounded operators from \(H_1\) into \(H_2\) by \(L (H_1, H_2)\).
Let \(F \triangleq H_2\), let \(\Omega\) be the space of all continuous functions \(\bR_+ \to F\), take a set \(G\) and let \(\mu \colon G \times H_2 \to H_2\) and \(\sigma \colon G \times H_2 \to L (H_1, H_2)\) be functions. The adjoint of \((A, D(A))\) is denoted by \((A^*, D(A^*))\).
By virtue of \cite[Lemma 7.3]{criens22}, there exists a countable set \(D \subset D (A^*) \subset H_2\) such that, for every \(y \in D(A^*)\), there exists a sequence \((y_n)_{n = 1}^\infty \subset D\) such that \(y_n \to y\) and \(A^* y_n \to A^* y\). In other words, \(D\) is a countable subset of \(D(A^*)\) which is dense in \(D(A^*)\) for the graph norm.
We set \(f_i (x) \triangleq x^i\) for \(x \in \bR\) and \(i = 1,2\), \[U \triangleq \big\{ (H_2 \ni x \mapsto f_i (\langle y, x \rangle_{H_2} )) \colon y \in D, i = 1, 2\big\}, \qquad Y^u \triangleq u (X),\ u \in U,\] and 
\begin{align*}
\mathscr{L}^{u, g} (x) \triangleq \big( \langle A^* y, x \rangle_{H_2} + \langle y, \mu (g, x) \rangle_{H_2} \big) f_i' ( \langle y, x\rangle_{H_2} ) + \tfrac{1}{2} \| \sigma^* (g, x) y \|^2_{H_1} f''_i (\langle y, x\rangle_{H_2}), 
\end{align*}
for \((x, g)\in H_2 \times G\) and \(u = f_i (\langle y, \cdot\, \rangle_{H_2}) \in U\). Finally, for \((t, \omega) \in \of 0, \infty\of\), we define  
\[
\Theta (t, \omega) \triangleq \big\{ ( \mathscr{L}^{u, g} ( \omega (t) ) )_{u \in U} \colon g \in G \big\}.
\]
In Lemma \ref{lem: SPDE representation} below, we show that, under suitable assumptions on the coefficients and the parameter space \(G\), any \(P \in \cR(x_0)\) is the law of a mild solution processes (see \cite[Section 6.1]{DaPrato}) to an SPDE of the form 
\[
d Y_t = A Y_t dt + \mu (\g (t, Y), Y_t) dt + \sigma (\g (t, Y), Y_t) d W_t, \quad Y_0 = x_0, 
\]
where \(\g = \g (P) \colon \of 0, \infty\of \hspace{0.1cm} \to G\) is a predictable function.
This observation confirms our interpretation as an SPDE with uncertain coefficients.

%-------------------------------
%-------------------------------

\section{Nonlinear Markov Processes} \label{sec: 2}
In this section we investigate a class of nonlinear Markov processes. That is, we study the family \(\{\mathcal{E}^x \colon x \in F\}\) under the following standing assumption.
\begin{SA} \label{SA: markov}
For every \((t, \omega) \in \of 0, \infty\of\), the set \(\Theta (t, \omega)\) depends on \((t, \omega)\) only through \(\omega (t-)\), i.e., \(\Theta (t, \omega) \equiv \Theta (\omega (t-))\), where \(\Theta \colon F \twoheadrightarrow \mathbb{K}^U\).
\end{SA}
The program of this section is the following.
First, we explain the Markov property of \(\{\mathcal{E}^x \colon x \in F\}\), which further leads to a sublinear Markovian semigroup \((T_t)_{t \geq 0}\) on the cone of bounded upper semianalytic functions. Afterwards, we establish conditions for \((T_t)_{t \geq 0}\) to be a sublinear Markovian semigroup on the cone of bounded upper semicontinuous functions and we derive a strong Markov selection principle.

\subsection{The nonlinear Markov property}
The following proposition provides the Markov property of \(\{\cE^x \colon x \in F\}\). The proof is given in Section \ref{sec: pf NMP}.
\begin{proposition} \label{prop: markov property}
For every upper semianalytic function \( \psi \colon \Omega \to [- \infty, \infty] \), the equality
\[
\cE^x( \psi \circ \theta_t) = \cE^x ( \cE^{X_t} (\psi))
\]
holds for every \((t, x) \in \bR_+ \times F\).
\end{proposition}

To the best of our knowledge, in the context of nonlinear stochastic processes, the nonlinear Markov property was first observed in \cite[Lemma 4.32]{hol16} for nonlinear Markovian semimartingales, see also \cite[Proposition~2.8]{CN22b} for a setting with continuous paths.

Similar to the fact that linear Markov processes have a canonical relation to linear semigroups, nonlinear Markov processes can be related to nonlinear semigroups.

\begin{definition} \label{def: nonlinear MSG}
Let \( \mathcal{H} \) be a convex cone of functions \( f \colon F \to \bR \) containing all constant functions.
A family of sublinear operators \( T_t \colon \mathcal{H} \to \mathcal{H}, \hspace{0,05cm} t \in \bR_+,\) is called a \emph{sublinear Markovian semigroup} on \( \mathcal{H} \) if it satisfies the following properties:
\begin{enumerate} [leftmargin=1cm]
    \item[\textup{(i)}] \( (T_t)_{t \geq 0} \) has the semigroup property, i.e.,
          \( T_s T_t = T_{s+t} \) for all \(s, t \in \bR_+ \) and
          \( T_0 = \on{id} \);
    \item[\textup{(ii)}] \( T_t \) is monotone for each \( t \in \bR_+\), i.e., 
    \( f, g \in \mathcal{H} \) with \( f \leq g \) implies \(T_t f \leq T_t g \);
    \item[\textup{(iii)}] \( T_t \) preserves constants for each  \( t \in \bR_+\), i.e.,
    \( T_t(c) = c \) for each \( c \in \bR  \).
\end{enumerate}
\end{definition}
For a bounded upper semianalytic function \(\psi \colon F \to \bR, x \in F\) and \(t \in \mathbb{R}_+\), we set 
\begin{equation} \label{eq: semigroup}
    T_t(\psi)(x) \triangleq \cE^x(\psi(X_t)) = \sup_{P \in \cR(x)} E^P \big[ \psi(X_t) \big].
\end{equation}
The following proposition should be compared to \cite[Remark 4.33]{hol16} and \cite[Proposition~2.9]{CN22b}, where it has been established for nonlinear Markovian semimartingale frameworks. Its proof can be found in Section \ref{sec: pf SG}.
\begin{proposition} \label{prop: semigroup}
The family \( (T_t)_{t \geq 0} \) defines a sublinear Markovian semigroup on the set of bounded upper semianalytic functions from \(F\) into \(\bR\).
\end{proposition}

		Linear semigroups with suitable regularity properties are well-known to be uniquely characterized by their (infinitesimal) generators. 
		For nonlinear L\'evy processes, a unique characterization of their nonlinear semigroups via (pointwise) generators has been established in \cite[Proposition~6.5]{K21}. 
		It is very interesting to prove such a result also for other classes of nonlinear processes such as NSPDEs. 
		We leave this question for future investigations. 

In the following section we establish a regularity preservation property of \((T_t)_{t \geq 0}\) in a continuous path setting. More precisely, we extend Proposition~\ref{prop: semigroup} via conditions for \((T_t)_{t \geq 0}\) to be a sublinear Markovian semigroup on the space of bounded upper semicontinuous functions from \(F\) into \(\bR\).

\subsection{The \(\usc_b\)--Feller Property} \label{sec: 3.3}

It is natural to ask whether the sublinear Markovian semigroup \((T_t)_{t \geq 0}\) preserves some regularity. 
In case \(F\) is endowed with the discrete topology, as it is the case for the class of nonlinear Markov chains from Section \ref{sec: nonlinear MC}, it is trivial that \(T_t\) is a selfmap on the space of bounded continuous functions for every \(t \in \bR_+\). In general, however, such a preservation property is non-trivial to establish. In the following we provide general conditions for a type of regularity preservation property of nonlinear stochastic processes with \emph{continuous} paths, namely the so-called \emph{\(\usc_b\)--Feller property}. Before we formulate our conditions, let us introduce a last bit of notation. We fix an \(\bR_+\)-valued continuous adapted process \(L = (L_t)_{t \geq 0}\) on \((\Omega, \cF, \mathbf{F})\) and, for \(M > 0\), we set 
\[
\rho_M (\omega) \triangleq \inf \{t \geq 0 \colon L_t \geq M\} \wedge M, \quad \omega \in \Omega.
\]
When \(\Omega\) is the Wiener space of continuous paths, which will be assumed below, a typical choice for \(L\) could be \(d_F (X, y_0)\), where \(d_F\) is a metric on \(F\) and \(y_0 \in F\) is an arbitrary reference point.
\begin{condition} \label{cond: sammel}
\quad 
\begin{enumerate}[leftmargin=1cm]
\item[\textup{(i)}] The underlying path space \(\Omega\) is the space of all continuous functions \(\bR_+ \to F\), for every \(u \in U\), the process \(Y^u\) has continuous paths and \(\Omega \ni \omega \mapsto Y^u(\omega) \in C(\bR_+; \mathbb{K})\) and \(\Omega \ni \omega \mapsto L (\omega) \in C(\bR_+; \bR)\) are continuous (where the image and the inverse image spaces are endowed with the local uniform topology). Furthermore, for every \(u \in U, M > 0\) and any compact set \(K \subset F\), the process \(Y^u_{\cdot \wedge \rho_M} \1_{\{X_0 \in K\}}\)
is bounded.
\item[\textup{(ii)}] The correspondence \(\Theta \colon F \twoheadrightarrow \bK^U\) is convex-valued.
\item[\textup{(iii)}] For every \(M> 0\), there exists a family \(\{K^{u, M} \colon u \in U\} \subset \mathbb{K}\) of bounded sets such that \(\Theta (\omega (t)) \subset \prod_{u \in U} K^{u, M}\) for all \((t, \omega)\in  \ \gs 0, \rho_M\of\). Moreover, for every compact set \(K \subset F\) and every \(T > 0\), 
\begin{align} \label{eq: moment bound relax}
\lim_{M\to \infty} \sup_{x \in K} \sup_{P \in \cR(x)} P \Big( \sup_{s \in [0, T]} L_s \geq M \Big) = 0.
\end{align}
\item[\textup{(iv)}] For every \(\omega \in \Omega\), the correspondence \(\bR_+ \ni t \mapsto \Theta (\omega (t))\) is upper hemicontinuous and compact-valued, and, for every \(t \in \bR_+\) and \(m \in \mathbb{N}\), the correspondence
\[\omega \mapsto \oconv \Theta ([t, t + 1/m], \omega) \triangleq \oconv \Big[\bigcup_{s \in [t, t + 1/m]} \Theta (\omega (s)) \Big]\] 
is upper hemicontinuous and compact-valued. Here, \(\oconv \hspace{-0.1cm}\) denotes the closure of the convex hull.
\item[\textup{(v)}] For every compact set \(K \subset F\), the set \(\bigcup_{x \in K} \cR (x) \subset \mathfrak{P}(\Omega)\) is relatively compact.
\end{enumerate}
\end{condition}

For nonlinear one-dimensional diffusions, a version of the following theorem was established in~\cite{CN22b}. Moreover, for nonlinear semimartingales \emph{with jumps} some conditions were proved in \cite[Theorem~4.41, Lemma~4.42]{hol16}.\footnote{It seems that there is a gap in the proof of \cite[Lemma 4.42]{hol16}. Indeed, it is claimed that the map \(\omega \mapsto \omega (t)\) is upper semicontinuous on the Skorokhod space. However, this is not the case. Indeed, by linearity, upper semicontinuity would already imply continuity, which is false. }
	 In the context of controlled diffusions, upper semicontinuity of the value function has been established in \cite{nicole1987compactification}.
We provide a result which reaches beyond semimartingale settings. 
Its proof can be found in Section \ref{sec: pf USC FP}.

\begin{theorem} \label{thm: USC Feller property}
Suppose that Condition \ref{cond: sammel} holds. Then, \((T_t)_{t \geq 0}\) is a sublinear Markovian semigroup on the space \(\usc_b (F; \bR)\) of bounded upper semicontinuous functions~\(F \to \mathbb{R}\).
\end{theorem}

\begin{remark} \label{rem: SV example}
	It is interesting to note that Condition \ref{cond: sammel} is not sufficient for the \(C_b\)--Feller property of \((T_t)_{t \geq 0}\), i.e., it does not imply \(T_t (C_b( F; \bR )) \subset C_b (F; \bR )\) for all \(t > 0\). 
	A counterexample is given in \cite[Exercise~12.4.2]{stroock2007multidimensional}. 
		In fact, the counterexample reveals an interesting effect when passing from linear to nonlinear settings. To explain this, we recall the example.
			Let \(b \colon \bR \to \bR\) be a bounded continuous function such that \(b (x) = \on{sgn} (x) \sqrt{|x|}\) for \(|x| \leq 1\) and which is continuously differentiable off \((-1, 1)\). For \(x \in \bR\), define
			\[
			\cR (x) \triangleq \Big\{ \delta_{g} \colon g \in C^1 (\bR_+; \bR),\ d g (t) = b (g (t)) dt, \, g (0) = x \Big\} \subset \mathfrak{P}(C(\bR_+; \bR)).
			\]
			In other words, the set \(\cR(x)\) contains all laws of solutions to the (deterministic) ordinary differential equation 
			\[
			d Y_t = b (Y_t) d t, \quad Y_0 = x.
			\]
			Notice that this is a special case of the NSPDE framework discussed in Section~\ref{sec: nonlinear SPDE}. We say that a family \((P_x)_{x \in \bR} \subset \mathfrak{P}(C(\bR_+; \bR))\) is \(C_b\)--Feller if the map \(x \mapsto E^{P_x} [ f (X_t) ]\) is continuous for every \(f \in C_b (\bR; \bR)\) and \(t > 0\).
			According to \cite[Exercise~12.4.2]{stroock2007multidimensional}, it is not possible to select a \(C_b\)--Feller family \((P_x)_{x \in \mathbb{R}}\) such that \(P_x \in \cR (x)\). By the linearity of the expectation map \(P \mapsto E^P[\, \cdot\,]\), it is easy to see that a family \((P_x)_{x \in \mathbb{R}}\) of probability measures is a \(C_b\)--Feller family if and only if it is a \(\usc_b\)--Feller family in the sense that \(x \mapsto E^{P_x} [ f (X_t) ]\) is upper semicontinuous for every \(f \in \usc_b (\bR; \bR)\) and \(t > 0\). Thus, \cite[Exercise~12.4.2]{stroock2007multidimensional} also tells us that it is not possible to select a \(\usc_b\)--Feller family \((P_x)_{x \in \mathbb{R}}\) such that \(P_x \in \cR (x)\). 
			However, Theorem~\ref{thm: USC Feller property} (or Corollary~\ref{coro:SMSP NSPDE} below) shows that the nonlinear semigroup 
			\[
			T_t (\psi) (x) = \sup_{P \in \cR (x)} E^P \big[ \psi (X_t) \big]
			\]
			has the \(\usc_b\)--Feller property (in the sense that \((T_t)_{t \geq 0}\) is a nonlinear semigroup on the cone \(\usc_b(\bR; \bR)\)). Broadly speaking, passing from the linear to the nonlinear setting provides a smoothing effect in the sense that the nonlinear semigroup has the \(\usc_b\)--Feller property although it is not possible to select a linear semigroup with this property from the uncertainty set.
\end{remark}
In Section \ref{sec: NSPDE} below, we return to the class of NSPDEs from Section \ref{sec: nonlinear SPDE} and we provide some general parametric conditions which imply Condition \ref{cond: sammel}. From a practical point of view, (i) -- (iv) from Condition \ref{cond: sammel} are mainly structural assumptions, while checking part (v) requires an argument.

\subsection{The strong Markov selection principle} \label{sec: 3.4}

For a probability measure \(P\) on \((\Omega, \mathcal{F})\), a kernel \(\Omega \ni \omega \mapsto Q_\omega \in \mathfrak{P}(\Omega)\), and a finite stopping time \(\tau\), we define the pasting measure
\begin{align} \label{eq: pasting measure}
(P \otimes_\tau Q) (A) \triangleq \iint \1_A (\omega \otimes_{\tau(\omega)} \omega') Q_\omega (d \omega') P(d \omega)
\end{align}
for all \(A \in \cF\).

\begin{definition}[Time inhomogeneous Markov Family]
	A family \(\{P_{(s, x)} \colon (s, x) \in \bR_+ \times F\} \subset \mathfrak{P}(\Omega)\) is said to be a \emph{strong Markov family} if \((t, x) \mapsto P_{(t, x)}\) is Borel and the strong Markov property holds, i.e., for every \((s, x) \in \bR_+ \times F\) and every finite stopping time \(\tau \geq s\),
	\[
	P_{(s, x)} (\, \cdot\, | \cF_\tau) (\omega) = \omega \otimes_{\tau (\omega)} P_{(\tau (\omega), \omega (\tau (\omega)))}
	\]
	for \(P_{(s, x)}\)-a.a. \(\omega \in \Omega\).
\end{definition}

We introduce a correspondence \(\mathcal{K} \colon \bR_+ \times F \twoheadrightarrow \mathfrak{P}(\Omega)\) by 
\begin{align*}
\cK (t, x) \triangleq \Big\{ P \in \mathfrak{P} (\Omega) \colon P(X &= x \text{ on } [0, t]) = 1, \forall_{u \in U} \ Y^u \in \fSac (t, P), \\& \qquad
(\llambda \otimes P)\text{-a.e. } (d \A^P (Y^u_{\cdot + t})/d \llambda)_{u \in U} \in \Theta (\, \cdot + t, X) \Big\},
\end{align*}
where \((t, x) \in \bR_+ \times F\). The following theorem is proved in Section \ref{sec: pf MSP}.

\begin{theorem}[Strong Markov Selection Principle]  \label{theo: strong Markov selection}
	Suppose that Condition~\ref{cond: sammel} holds.
	For every \(\psi \in \usc_b(\bR; F)\) and every \(t > 0\), there exists a strong Markov family \(\{P_{(s, x)} \colon\) \((s, x) \in \bR_+ \times F\}\) such that, for all \((s, x)\in \bR_+ \times F\), \(P_{(s, x)} \in \cK (s, x)\) and 
	\[
	E^{P_{ (s, x) }} \big[ \phi (X_t) \big] = \sup_{P \in \cK (s, x)} E^P \big[ \phi (X_t) \big].
	\]
	In particular, for every \(x \in F\), 
	\[
	T_t (\psi) (x) = E^{P_{(0, x)}} \big[ \psi (X_t) \big]. 
	\]
\end{theorem}
In a relaxed framework for finite-dimensional controlled diffusions, a strong Markov selection principle has been established in \cite{nicole1987compactification}.
A strong Markov selection principle for nonlinear one-dimensional diffusions was proved in \cite{CN22b}. Theorem \ref{theo: strong Markov selection} covers the result from \cite{CN22b} as a special case (see Section~\ref{sec: NSPDE} below).

In the next section we tailor the Theorems \ref{thm: USC Feller property} and \ref{theo: strong Markov selection} to the NSPDE setting from Section~\ref{sec: nonlinear SPDE}. Furthermore, we are going to discuss the strong \(\usc_b\)--Feller property, which provides a smoothing effect.

\section{NSPDEs: \(\usc_b\)--Feller Properties and Strong Markov Selections} \label{sec: NSPDE}
We return to the class of nonlinear SPDEs from Section \ref{sec: nonlinear SPDE} and provide explicit conditions (in terms of the drift and diffusion coefficients) for the \(\usc_b\)--Feller property and the strong Markov selection principle. 

We recall the precise setting.
Let \((H_1, \|\cdot\|_{H_1}, \langle\, \cdot, \cdot\, \rangle_{H_1})\) and \((H_2, \|\cdot\|_{H_2}, \langle\, \cdot, \cdot\,\rangle_{H_2})\) be two separable Hilbert spaces over the real numbers, and let \((A, D(A))\) be the generator of a \(C_0\)-semigroup \((S_t)_{t \geq 0}\) on the Hilbert space \(H_2\).
Moreover, let \(\Omega\) be the space of all continuous functions \(\bR_+ \to F\), let \(G\) be a topological space and let \(\mu \colon G \times H_2 \to H_2\) and \(\sigma \colon G \times H_2 \to L (H_1, H_2)\) be Borel functions, where, for the latter, we mean that \(\sigma h_1\) is a Borel function for every \(h_1 \in H_1\).
Further, let \(D \subset D (A^*)\) be a countable set such that for every \(y \in D(A^*)\) there exists a sequence \((y_n)_{n = 1}^\infty \subset D\) such that \(y_n \to y\) and \(A^* y_n \to A^* y\). 
Recall that such a set exists by \cite[Lemma 7.3]{criens22}.
We set \(f_i (x) \triangleq x^i\) for \(x \in \bR\) and \(i = 1,2\), 
\[U \triangleq \big \{ (H_2 \ni x \mapsto f_i (\langle y, x \rangle_{H_2} )) \colon y \in D, i = 1, 2 \big\}, \qquad Y^u \triangleq u (X),\ u \in U,\] 
and 
\begin{align*}
\mathscr{L}^{u, g} (x) \triangleq \big( \langle A^* y, x \rangle_{H_2} + \langle y, \mu (g, x) \rangle_{H_2} \big) f_i' ( \langle y, x\rangle_{H_2} ) + \tfrac{1}{2} \| \sigma^* (g, x) y \|^2_{H_1} f''_i (\langle y, x\rangle_{H_2}),
\end{align*}
for \((x, g)\in H_2\times G\) and \(u = f_i (\langle y, \cdot\, \rangle_{H_2}) \in U\). 
Finally, for \(x \in H_2\) and \((t, \omega) \in \of 0, \infty\of\), we define  
\[
\Theta (x) \triangleq \big\{ ( \mathscr{L}^{u, g} ( x ) )_{u \in U} \colon g \in G \big\}, \qquad \Theta (t, \omega) \triangleq \Theta (\omega (t)).
\]
We denote the operator norm on \(L (H_1, H_2)\) by \(\|\cdot\|_{L (H_1, H_2)}\) and the Hilbert--Schmidt norm by \(\|\cdot \|_{L_2(H_1, H_2)}\).
\begin{condition} \label{cond: main SPDE}
	\quad 
	\begin{enumerate} [leftmargin=1cm]
		\item[\textup{(i)}]
		\(G\) is a compact metrizable space.
		\item[\textup{(ii)}] 
		For every \(x \in H_2\), the set  
		\(\{ (\mu (g, x), (\sigma \sigma^*)(g, x)) \colon g \in G\}
		\)
		is convex.
		\item[\textup{(iii)}]
		For every \(y \in D\), the functions
		\(\langle y, \mu \rangle_{H_2}\) and \(\|\sigma^* y\|_{H_1}\) are continuous.
		\item[\textup{(iv)}]
		There exists a constant \(\C> 0\) such that 
		\begin{align*}
		\|\mu (g, x)\|_{H_2} + \|\sigma (g, x)\|_{L (H_1, H_2)} \leq \C (1 + \|x\|_{H_2})
		\end{align*}
		for all \((g, x) \in G \times H_2\).
		\item[\textup{(v)}]
		The semigroup \(S\) is compact, i.e., \(S_t\) is compact for every \(t > 0\), and there exists an \(\alpha \in (0, 1/2)\) and a Borel function \(\f \colon (0, \infty) \to [0, \infty]\) such that
		\[
		\int_0^T \Big[ \frac{\f (s)}{s^\alpha} \Big]^2 ds < \infty, \quad \forall \hspace{0.025cm}T > 0, 
		\]
		and, for all \(t > 0, x \in H_2\) and \(g \in G\), 
		\[
		\|S_t \sigma (g, x)\|_{L_2 (H_1, H_2)} \leq \f (t) (1 + \|x\|_{H_2}).
		\]
	\end{enumerate}
\end{condition}
\begin{lemma} 
	If Condition \ref{cond: main SPDE} holds, then both Standing Assumptions \ref{SA: meas gr} and \ref{SA: non empty} are satisfied.
 \end{lemma}
\begin{proof}
	Thanks to (i) and (iii) of Condition \ref{cond: main SPDE}, Standing Assumption \ref{SA: meas gr} is implied by \cite[Lemma 2.12]{CN22a}.
	Moreover, thanks to (iii) -- (v) from Condition \ref{cond: main SPDE}, Standing Assumption~\ref{SA: non empty} follows from \cite[Theorem 2.5]{criens22} and Lemmata \ref{lem: p^t} and \ref{lem: iwie Markov} below.
\end{proof}

The following proposition is proved in Section \ref{sec: pf NSPDE}.
\begin{proposition} \label{prop: conds hold for SPDE}
Condition \ref{cond: main SPDE} implies Condition \ref{cond: sammel} with \(L \equiv \|X\|_{H_2}\).
\end{proposition}
Thanks to Proposition \ref{prop: conds hold for SPDE}, the Theorems \ref{thm: USC Feller property} and \ref{theo: strong Markov selection} yield the following result.
\begin{corollary} \label{coro:SMSP NSPDE}
	Suppose that Condition~\ref{cond: main SPDE} holds.
	Then, \((T_t)_{t \geq 0}\) is a sublinear Markovian semigroup on \(\usc_b(H_2; \bR)\) and, for every \(\psi \in \usc_b(H_2; \bR)\) and every \(t > 0\), there exists a strong Markov family \(\{P_{(s, x)} \colon  (s, x) \in \bR_+ \times H_2\}\) such that, for every \((s, x) \in \bR_+ \times H_2\), \(P_{(s, x)} \in \cK (s, x)\) and 
		\[
	E^{P_{ (s, x) }} \big[ \phi (X_t) \big] = \sup_{P \in \cK (s, x)} E^P \big[ \phi (X_t) \big].
	\]
	In particular, \(T_t (\psi) (x) = E^{P_{(0, x)}}[\psi (X_t)]\) for every \(x \in H_2\).
\end{corollary}

Finally, we establish a smoothing effect for stochastic nonlinear Cauchy problems with drift uncertainty, which are infinite-dimensional processes of the form
\[
d Y_t = A Y_t dt + \mu (Y_t) dt + d W_t, 
\]
with drift uncertainty. 

\begin{definition}
	We say that \((T_t)_{t \geq 0}\) has the \emph{strong \(\usc_b\)--Feller property} if \[T_t (\usc_b (F; \bR)) \subset C_b (F; \bR)\] for every \(t > 0\). 
\end{definition}

To the best of our knowledge, the strong \(\usc_b\)--Feller property of sublinear Markovian semigroups was discovered in \cite{CN22b} for nonlinear one-dimensional elliptic diffusions. In~\cite{CN22c} it was established for nonlinear multidimensional diffusions with fixed (strongly) elliptic diffusion coefficients. Adapting some ideas from \cite{CN22b,CN22c,stroock2007multidimensional}, we establish the seemingly first result for infinite-dimensional nonlinear stochastic processes.

\begin{condition} \label{cond: strong USC Feller NSPDE}
		\quad 
	\begin{enumerate} [leftmargin=1cm]
		\item[\textup{(i)}]
		\(G\) is a compact metrizable space and \(H_1 \equiv H_2 \equiv H\).
		\item[\textup{(ii)}] 
		For every \(x \in H\), the set  
		\(
		\{ \mu (g, x) \colon g \in G\}
		\)
		is convex.
		\item[\textup{(iii)}]
		For every \(y \in D\), \(\langle y, \mu\rangle_H\) is continuous and \(\sigma\) equals constantly the identity operator on \(H\), i.e.,  \(\sigma (g, x) \equiv \on{id}\) for all \((g, x) \in G \times H\).
		\item[\textup{(iv)}]
		There exists a constant \(\C> 0\) such that
		\[
		\|\mu (g, x)\|_{H} \leq \C( 1 + \|x\|_H) 
		\]
		for all \((g, x) \in G \times H\).
		\item[\textup{(v)}]
		There exists an \(\alpha \in (0, 1/2)\) such that
		\[
		\int_0^T \frac{\|S_{s}\|^2_{L_2(H, H)} ds}{s^{2\alpha}} < \infty, \quad \forall \hspace{0.025cm}T > 0.
		\]
	\end{enumerate}
\end{condition}

\begin{remark}[\cite{criens20}]
	Part (v) of Condition \ref{cond: strong USC Feller NSPDE} implies that \(S\) is a compact semigroup. Moreover, the condition
		\[
	\int_0^T \frac{\|S_{s}\|^2_{L_2(H, H)} ds}{s^{2\alpha}} < \infty
	\]
	holds for \emph{all} \(T > 0\) once it holds for \emph{some} \(T > 0\).
\end{remark}

The proof for the next theorem is given in Section \ref{sec: pf strong USC}.
\begin{theorem} \label{theo: strong USC NSPDE}
	If Condition \ref{cond: strong USC Feller NSPDE} holds, then \((T_t)_{t \geq 0}\) has the strong \(\usc_b\)--Feller property. 
\end{theorem}

Under suitable additional conditions (see \cite[Hypothesis 9.1]{DaPrato}), Theorem \ref{theo: strong USC NSPDE} can be extended to NSPDEs with \emph{certain} diffusion coefficient, i.e., to the case \(\sigma (g, x) \equiv \sigma (x)\).

\begin{example}[\cite{gatarekgoldys}]
	Let \(\mathcal{O}\subset \mathbb{R}^d\) be a bounded domain with smooth boundary. Part \textup{(v)} of Condition \ref{cond: strong USC Feller NSPDE} holds when \(H = L^2 (\mathcal{O})\) and \(A\) is a strongly elliptic operator of order \(2m > d\), see \cite[Remark 3, Example 3]{gatarekgoldys} for some details. In particular, this is the case when \(d = 1\) and \(A\) is the Laplacian, i.e., Condition \ref{cond: strong USC Feller NSPDE} holds for one-dimensional stochastic heat equations with drift uncertainty.
\end{example}

%-------------------------------
%-------------------------------

%-------------------------------
%-------------------------------

\section{Proof of the Dynamic Programming Principle: Theorem \ref{theo: DPP}} \label{sec: pf DPP}
Recall that the Standing Assumptions \ref{SA: meas gr} and \ref{SA: non empty} are in force.
The proof of Theorem \ref{theo: DPP} is based on an application of the general \cite[Theorem 2.1]{ElKa15}, which provides three abstract conditions on the set \(\cA\) implying the DPP. In the following we verify these conditions. For reader's convenience, let us restate them. 
\begin{enumerate} [leftmargin=1cm]
\item[\textup{(i)}]
\emph{Measurable graph condition:} The set 
\(\{ (t, \omega, P) \in \of 0, \infty\of \hspace{0.05cm} \times\hspace{0.05cm} \mathfrak{P}(\Omega) \colon P \in \cA(t, \omega) \}
\)
is analytic.
\item[\textup{(ii)}]
\emph{Stability under conditioning:} For any \(t \in \mathbb{R}_+\), any stopping time \(\tau\) with \(t \leq \tau < \infty\), and any \(P \in \cA(t, \alpha)\) there exists a family \(\{P (\, \cdot\, | \mathcal{F}_\tau) (\omega) \colon\) \(\omega \in \Omega \}\) of regular \(P\)-conditional probabilities given \(\mathcal{F}_\tau\) such that \(P\)-a.s. \(P (\, \cdot\, | \mathcal{F}_\tau) \in \cA(\tau, X)\).
\item[\textup{(iii)}]
\emph{Stability under pasting:} For any \(t \in \mathbb{R}_+\), any stopping time \(\tau\) with \(t \leq \tau < \infty\), any \(P \in \cA(t, \alpha)\) and any \(\mathcal{F}_\tau\)-measurable map \(\Omega \ni \omega \mapsto Q_\omega \in \mathfrak{P}(\Omega)\) the following implication holds:
\[
P\text{-a.s. } Q \in \cA (\tau, X)\quad \Longrightarrow \quad P \otimes_\tau Q \in \cA(t, \alpha).
\]
Here, the pasting measure \(P \otimes_\tau Q\) was defined in \eqref{eq: pasting measure}.
\end{enumerate}
In the following three sections we check these properties. In the fourth (and last) section, we finalize the proof of Theorem \ref{theo: DPP}. Hereby, although technically different, we follow a strategy from \cite{neufeld2017nonlinear} and we also use some technical ideas from \cite{CN22a} about the measurability of the semimartingale property in time.

\subsection{Measurable graph condition}
The proof of the measurable graph condition is split into several parts.

\begin{lemma} \label{lem: measurable in zero}
For every \(\mathbb{K}\)-valued \cadlag \(\mathbf{F}\)-adapted process \(Y^*\), the set \[\fPas (Y^*) \triangleq \big \{P \in \mathfrak{P}(\Omega) \colon Y^* \in \fSac (P)\big \}\] is Borel. Moreover, there exists a Borel map
\[
\of 0, \infty\of \hspace{0.05cm} \times \hspace{0.05cm} \fPas (Y^*) \ni (t, \omega, P) \mapsto a^P_t (Y^*) (\omega) \in \mathbb{K}
\]
such that, for every \(P \in \fPas (Y^*)\), \(a^P\) is predictable and \((\llambda \otimes P)\)-a.e. \(a^P (Y^*) = d \A^P (Y^*)/d \llambda\).
\end{lemma}
\begin{proof}
	For \(P \in \mathfrak{P}(\Omega)\), denote the set of \(\mathbb{K}\)-valued \(P\)-\(\mathbf{F}_+\)-semimartingales by~\(\mathcal{S}(P)\). 
	The set \(\fPs (Y^*) \triangleq \{P \in \mathfrak{P}(\Omega) \colon Y^* \in \mathcal{S}(P)\}\) is Borel by \cite[Theorem 2.5]{neufeld2014measurability}. 
	Furthermore, by the very same theorem, there are Borel maps (with values in a suitable Polish space; see Lemma~\ref{lem: Levy polish} below)
	\[
	 \mathbb{R}_+ \times \Omega \times \fPs (Y^*) \ni (t, \omega, P) \mapsto (B^P_t (\omega), C_t (\omega), \nu^P (\omega))
	\]
	such that \((B^P, C, \nu^P)\) are the \(P\)-\(\mathbf{F}_+\)-semimartingale characteristics of \(Y^*\) (corresponding to a given truncation function \(h\)).
	By virtue of \cite[Proposition II.2.29]{JS}, we have 
	\begin{align*}
	\mathfrak{P}_{\on{sp} \on{sem}} (Y^*) &\triangleq \big\{P \in \mathfrak{P}(\Omega) \colon Y^* \in \fSss (P) \big\} 
	\\&= \big\{P \in \fPs (Y^*)\colon \forall_{t \in \mathbb{N}} \ P\text{-a.s. } (|x|^2 \wedge |x|) * \nu^P_t < \infty \big\},
	\end{align*}
	where, as usual, 
	\[
	(|x|^2 \wedge |x|) * \nu^P_t \triangleq \int (|x|^2 \wedge |x|)\hspace{0.05cm} \nu^P ([0, t] \times dx).
	\]
	Consequently, \(\mathfrak{P}_{\on{sp} \on{sem}} (Y^*)\) is Borel by \cite[Theorem 8.10.61]{bogachev}. Thanks again to \cite[Proposition~II.2.29]{JS}, in case \(Y^* \in \fSss (P)\), the predictable part of (locally) finite variation in the semimartingale decomposition is given by 
	\begin{align*} 
	\A^P (Y^*) = B^P + (x - h(x)) * \nu^P.
	\end{align*}
	Hence, 
	\begin{align*}
	\fPas (Y^*)=\big \{ P \in \mathfrak{P}_{\on{sp} \on{sem}} (Y^*)\colon P\text{-a.s. } B^P + (x - h(x)) * \nu^P \ll \llambda \big\}.
	\end{align*}
	We deduce from \cite[Theorem 8.4.4]{benedetto} that \(P\)-a.s. there exists a decomposition 
	\[
	B^P + (x - h(x)) * \nu^P = \int_0^\cdot \phi^P_s ds + \psi^P, 
	\]
	where \(t \mapsto \psi^P_t\) is singular w.r.t. the Lebesgue measure, \((t, \omega, P) \mapsto \phi^P_t (\omega) \in \mathbb{K}\) is Borel and \(\phi^P\) is predictable.
	Now, 
	\begin{align*}
	\fPas(Y^*)=\Big \{ P \in \mathfrak{P}_{\on{sp} \on{sem}} (Y^*)\colon \forall_{t \in \mathbb{Q}_+} \ P\text{-a.s. } B^P_t + (x - h(x)) * \nu^P_t = \int_0^t \phi^P_s ds \Big\},
	\end{align*}
	and the latter set is Borel by \cite[Theorem 8.10.61]{bogachev}.
	Finally, we can take \(a^P (Y^*) = \phi^P\) and the proof is complete.
\end{proof}

\begin{lemma} \label{lem: map prod meas}
    The map \((t, P) \mapsto P \circ \th_t^{-1} \triangleq P_t\) is Borel. 
\end{lemma}
\begin{proof}
As the map \((t, \omega) \mapsto \theta_t (\omega)\) is Borel, the claim follows from \cite[Theorem 8.10.61]{bogachev}. 
\end{proof}
The following lemma is a version of \cite[Lemma 3.3]{CN22a} for \(F\)-valued c\`adl\`ag, instead of real-valued continuous, processes. The proof is verbatim the same and omitted here.

\begin{lemma} \label{lem: filtraation shift}
    Let \(Y = (Y_t)_{t \geq 0}\) be an \(F\)-valued \cadlag process and set \(\mathcal{F}^Y_t \triangleq \sigma (Y_s, s\leq t)\) for \(t \in \mathbb{R}_+\). Then, \(
    Y^{-1} (\mathcal{F}_{s+}) = \mathcal{F}^Y_{s+}\) for all \(s \in \mathbb{R}_+.\)
\end{lemma}

The following lemma is a restatement of \cite[Lemma 2.9 a)]{jacod80}. 

\begin{lemma} \label{lem: jacod restatements}
	Take two filtered probability spaces \(\B^* = (\Omega^*, \mathcal{F}^*, \F^* = (\mathcal{F}^*_t)_{t \geq 0}, P^*)\) and  \(\B' = (\Omega', \mathcal{F}', \F' = (\mathcal{F}'_t)_{t \geq 0}, P')\) with right-continuous filtrations and the property that there is a map \(\phi \colon \Omega' \to \Omega^*\) such that
	\(
	\phi^{-1} (\mathcal{F}^*) \subset \mathcal{F}',P^* = P' \circ \phi^{-1}\) and \(\phi^{-1} (\mathcal{F}^*_t) = \mathcal{F}'_t\) for all \(t \in \mathbb{R}_+\).
Then, \(X^*\) is a \(d\)-dimensional semimartingale on~\(\B^*\) if and only if \(X' = X^* \circ \phi\) is a \(d\)-dimensional semimartingale on~\(\B'\). Moreover, \((B^*, C^*, \nu^*)\) are the characteristics of \(X^*\) if and only if \((B^* \circ \phi, C^* \circ \phi, \nu^* \circ \phi)\) are the characteristics of \(X' = X^* \circ \phi\).
\end{lemma}

\begin{lemma} \label{lem: MGC Borel}
   The set \begin{align*}
   \big \{ (t, \omega, P) \in \of 0, \infty\of\hspace{0.05cm} \times \hspace{0.05cm} \mathfrak{P}&(\Omega) \colon \forall_{u \in U} \hspace{0.15cm} Y^u \in \fSac (t, P), \\ & (\llambda \otimes P)\text{-a.e. } (d\A^P (Y^u_{\cdot + t}) / d \llambda)_{u \in U} \in \Theta (\, \cdot + t, \omega \otimes_t X) \big\}
   \end{align*} 
   is Borel.
\end{lemma}
\begin{proof}
	For \(\omega, \omega' \in \Omega\) and \(t \in \mathbb{R}_+\), we define the concatenation
	\[
	\omega\ \widetilde{\otimes}_t\ \omega' \triangleq  \omega \1_{[ 0, t)} + (\omega (t) + \omega' (\, \cdot - t) - \omega' (0)) \1_{[t, \infty)},
	\]
	and we notice that \((\omega\ \widetilde{\otimes}_t\ X) \circ \theta_t = \omega \otimes_t X\) for all \((t, \omega) \in \of 0, \infty\of\). Further, recall that, for every \(u \in U\), we assume that \(Y^u_{\cdot + t} = Y^u \circ \theta_t\) for all \(t \in \mathbb{R}_+\). Let \(a^P_t (Y^*) (\omega)\) be the map from Lemma~\ref{lem: measurable in zero}.
Thanks to Lemmata \ref{lem: filtraation shift} and \ref{lem: jacod restatements}, we get that
	\begin{align*}
	\big\{ 
	(&t, \omega, P) \colon \forall_{u \in U} \, Y^u \in \fSac (t, P), \, (\llambda \otimes P)\text{-a.e. }  (d\A^P (Y^u_{\cdot + t}) / d \llambda)_{u \in U} \in \Theta (\, \cdot + t, \omega \otimes_t X)
	\big\}
	\\&= \big\{ (t, \omega, P) \colon \forall_{u \in U} \hspace{0.1cm} P_t \in \fPas (Y^u), (\llambda \otimes P_t)\text{-a.e. }  (a^{P_t} (Y^u))_{u \in U} \in \Theta (\, \cdot + t, \omega \ \widetilde{\otimes}_t\ X)
	\big\}.
	\end{align*}
The latter set is Borel by Lemmata~\ref{lem: measurable in zero} and \ref{lem: map prod meas}, and \cite[Theorem 8.10.61]{bogachev}.
\end{proof}

\begin{corollary} \label{coro: MGC}
The measurable graph condition holds.
\end{corollary}

%-----------------------------------
%-----------------------------------

\subsection{Stability under Conditioning}
Next, we check stability under conditioning. In this section, we fix \((t^*, \omega^*) \in \of 0, \infty\of\), a stopping time \(\tau\) with \(t^* \leq \tau <\infty\), and a probability measure \(P\) on \((\Omega, \mathcal{F})\) such that \(P (X = \omega^* \text{ on } [0, t^*]) = 1\). We denote by \(P(\,\cdot\, | \mathcal{F}_\tau)\) a version of the regular conditional \(P\)-probability given~\(\mathcal{F}_\tau\).
We recall some simple observations.
\begin{lemma} \label{lem: collection easy observations}
\quad 
\begin{enumerate} [leftmargin=1cm]
    \item[\textup{(i)}]
There exists a \(P\)-null set \(N \in \mathcal{F}_\tau\) such that \(P (A | \mathcal{F}_\tau) (\omega) = \1_A (\omega)\) for all \(A \in \mathcal{F}_\tau\) and~\(\omega \not \in N\).
\item[\textup{(ii)}]
	If \(\omega \mapsto Q_\omega\) is a kernel from \(\mathcal{F}\) into \(\mathcal{F}\) such that 
	\(Q_\omega ( \omega = X \text{ on } [0, \tau(\omega)]) = 1\) for \(P\)-a.a. \(\omega \in \Omega\), then 
	\(
	Q_\omega (\omega \otimes_{\tau(\omega)} X = X) = 1
	\)
	for \(P\)-a.a. \(\omega \in \Omega\).
	In particular, there exists a \(P\)-null set \(N \in \mathcal{F}_\tau\) such that \(P (\omega \otimes_{\tau(\omega)} X = X | \mathcal{F}_\tau) (\omega) = 1\) for all \(\omega \not \in N\).
\item[\textup{(iii)}]
A measurable process \(H = (H_t)_{t \geq 0}\) on \((\Omega, \mathcal{F})\) is optional (resp. predictable) if and only if, for all \(t \in \bR_+\), \(H_t (\omega)\) depends on \(\omega\) only through the values \((\omega (s))_{s \leq t}\) (resp. \((\omega(s))_{s < t}\)). 
    \end{enumerate}
\end{lemma}
\begin{proof}
	Part (i) is classical (\cite[Theorem 1.1.8]{stroock2007multidimensional}), part (ii) is obvious and part (iii) follows from \cite[Theorem IV.97]{DM78}.
\end{proof}

In the following we use the standard notation \(M^s = M_{\cdot \wedge s}\) for a process \(M\) and a time~\(s\).
The next lemma is implied by \cite[Theorem~1.2.10]{stroock2007multidimensional}.
\begin{lemma} \label{lem: stroock varadhan cond mg}
	Let \(\mathbf{G} = (\mathcal{G}_t)_{t \geq 0}\)  be the filtration on \((\Omega, \mathcal{F})\) that is generated by a \cadlag process with values in a Polish space. Furthermore, let \(\rho\) be a \(\mathbf{G}\)-stopping time such that \(\rho \geq t^*\).
    If \(M - M^{t^*}\) is a \(P\)-\(\mathbf{G}\)-martingale, then there exists a \(P\)-null set \(N\) such that \(M - M^{\rho (\omega)}\) is a \(P(\,\cdot\, | \mathcal{G}_\rho) (\omega)\)-\(\mathbf{G}\)-martingale for all \(\omega \not \in N\).
\end{lemma}

\begin{lemma} \label{lem: characteristics under conditioning}
Suppose that \(Y \in \fSss (t^*, P)\). Then, there exists a \(P\)-null set \(N\) such that, for every \(\omega \not \in N\), \(Y \in \fSss (\tau (\omega), P (\,\cdot\, | \cF_\tau)(\omega))\) and \(P(\,\cdot\,| \cF_\tau)(\omega)\)-a.s.
\[
\A^{P (\,\cdot\, | \cF_\tau) (\omega)} (Y_{\cdot + \tau(\omega)}) = (\A^P_{\cdot + \tau (\omega) - t^*} (Y_{\cdot + t^*}) - \A^P_{\tau (\omega) - t^*} (Y_{\cdot + t^*})) (\omega \otimes_{\tau(\omega)} X).
\]
Moreover, when \(\nu^P\) denotes the third \(P\)-characteristic of \(Y_{\cdot + t^*}\), then, for every \(\omega \not \in N\), the third \(P(\,\cdot\,| \cF_\tau)(\omega)\)-characteristic of \(Y_{\cdot + \tau(\omega)}\) (for the right-continuous natural filtration generated by \(X_{\cdot + \tau (\omega)}\)) is given by 
\[
\nu^P  (\omega \otimes_{\tau(\omega)} X; \tau (\omega) - t^* + dt, dx).
\]
\end{lemma}
\begin{proof}
Let \(\mathbf{G}^{t^*} = (\mathcal{G}^{t^*}_t)_{t \geq 0}\) be the filtration generated by \(X_{\cdot + t^*}\) and let \((\tau_n)_{n = 1}^\infty\) be a localizing sequence for the \(P\)-\(\mathbf{G}^{t^*}\)-local martingale
\[
M \triangleq  Y_{\cdot + t^*}  - \A^P (Y_{\cdot + t^*}).
\] 
By Lemma~\ref{lem: stroock varadhan cond mg}, there exists a \(P\)-null set \(N\) such that \(M^{\tau_n} - M^{\tau_n \wedge (\tau(\omega) - t^*)}\) is a \(P(\,\cdot\,| \mathcal{G}^{t^*}_{\tau - t^*})(\omega)\)-\(\mathbf{G}^{t^*}\)-martingale for every \(\omega \not \in N\). As \(P (X = \omega^* \text{ on } [0, t^*]) = 1\), we have \(P\)-a.s. \(P (\, \cdot\, | \mathcal{G}^{t^*}_{\tau - t^*}) = P (\, \cdot\, | \mathcal{F}_\tau)\). Hence, possibly making \(N\) a bit larger, \(M^{\tau_n} - M^{\tau_n \wedge (\tau(\omega) - t^*)}\) is a \(P(\,\cdot\,| \cF_\tau)(\omega)\)-\(\mathbf{G}^{t^*}\)-martingale for every \(\omega \not \in N\).
Inserting \(s \mapsto s + \tau (\omega) - t^*\) with \(\omega \not \in N\) shows that the process
\[
Y_{\cdot  \wedge \rho_n + \tau (\omega)} - Y_{\tau (\omega)} - (\A^P_{\cdot \wedge \rho_n + \tau(\omega) - t^*} (Y_{\cdot + t^*}) - \A^P_{\tau (\omega) - t^*} (Y_{\cdot + t^*}))
\]
is a \(P(\,\cdot\, | \mathcal{F}_\tau) (\omega)\)-\(\mathbf{G}^{t^*}_{\cdot  + \tau(\omega) - t^*}\)-martingale, where 
\[
\rho_n \triangleq [ \tau_n - \tau(\omega) + t^*] \vee 0, \quad n = 1, 2, \dots.
\]
By Lemma~\ref{lem: collection easy observations}, possibly making \(N\) again larger, we have \(P (\omega \otimes_{\tau(\omega)} X = X | \cF_\tau)(\omega) = 1\) for all \(\omega \not \in N\). Hence, for all \(\omega \not \in N\), using that \(\mathbf{G}^{\tau (\omega)} \subset \mathbf{G}^{t^*}_{\cdot + \tau(\omega) - t^*}\) and Lemma~\ref{lem: collection easy observations}~(iii), the tower rule shows that the process
\[
(Y_{\cdot \wedge \rho_n + \tau (\omega)} - Y_{\tau (\omega)} - (\A^P_{\cdot \wedge \rho_n + \tau(\omega) - t^*} (Y_{\cdot + t^*}) - \A^P_{\tau (\omega) - t^*} (Y_{\cdot + t^*}))) (\omega \otimes_{\tau (\omega)} X)
\]
is a \(P(\,\cdot\, | \mathcal{F}_\tau) (\omega)\)-\(\mathbf{G}^{\tau(\omega)}\)-martingale. Notice that \(\rho_n (\omega \otimes_{\tau(\omega)} X)\) is a \(\mathbf{G}^{\tau(\omega)}\)-stopping time by Galmarino's test (\cite[Theorem~IV.100]{DM78}).
Finally, since
\[
E^P \Big[ P \Big( \liminf_{n \to \infty} \tau_n < \infty \mid \cF_\tau \Big) \Big] = P \Big( \liminf_{n \to \infty} \tau_n < \infty \Big) = 0,
\]
we can enlarge \(N\) a last time such that also \(P(\tau_n \to \infty | \cF_\tau) (\omega) = 1\) for all \(\omega \not \in N\). Of course, \(\tau_n \to \infty\) implies that \(\rho_n \to \infty\).
In summary, by virtue of Lemma~\ref{lem: collection easy observations}~(iii), for all \(\omega \not \in N\), the process \(Y_{\cdot + \tau(\omega)}\) is a special \(P (\, \cdot\, | \cF_\tau)(\omega)\)-\(\mathbf{G}^{\tau(\omega)}\)-semimartingale with predictable compensator 
\[
(\A^P_{\cdot + \tau(\omega) - t^*} (Y_{\cdot + t^*}) - \A^P_{\tau (\omega) - t^*} (Y_{\cdot + t^*}))) (\omega \otimes_{\tau (\omega)} X).
\]
This finishes the proof of the first claim.

To prove the second claim, let \(g \colon \mathbb{R} \to \mathbb{R}\) be a bounded Borel function which vanishes around the origin. Then, arguing as above, possibly making \(N\) a bit larger, for every \(\omega \not \in N\), the process
\[
\sum_{s \leq \cdot} g ( \Delta (Y_{\cdot + \tau(\omega)})_s) - \int_0^\cdot \hspace{-0.1cm} \int g (x) \nu^P  (\omega \otimes_{\tau(\omega)} X; \tau (\omega) - t^* + dt, dx)
\]
is a local \(P (\, \cdot\, | \cF_\tau)(\omega)\)-\(\mathbf{G}^{\tau(\omega)}\)-martingale. The third characteristic can be characterized by these local martingale properties for countably many test functions \(g\), see \cite[Theorem~II.2.21]{JS}. Thus, the formula for the third \(P(\, \cdot\, | \cF_\tau) (\omega)\)-characteristic follows. 
The proof is complete.
\end{proof}

We are in the position to deduce stability under conditioning. 
\begin{lemma} \label{coro: stab cond}
    If \(P \in \cA (t^*, \omega^*)\), then \(P\)-a.s. \(P (\, \cdot\, | \mathcal{F}_\tau) \in \cA (\tau, X)\). 
\end{lemma}
\begin{proof}
    By Lemma \ref{lem: collection easy observations} (i), \(P (\tau = \tau (\omega)| \mathcal{F}_\tau)(\omega) = 1\) and 
    \(
    P(X = \omega \text{ on } [0, \tau(\omega)] | \cF_\tau )(\omega) = 1
    \)
    for \(P\)-a.a. \(\omega \in \Omega\).
    Thanks to Lemma \ref{lem: characteristics under conditioning}, as the set \(U\) is assumed to be countable, for \(P\)-a.a. \(\omega \in \Omega\) and all \(u \in U\), we have \(Y^u \in \fSac (\tau (\omega), P(\, \cdot\, | \cF_\tau)(\omega))\) and \(P(\, \cdot\, | \cF_\tau)(\omega)\)-a.s.
    \[
    d\A^{P (\, \cdot\, | \cF_\tau) (\omega)} (Y^u_{\cdot + \tau(\omega)}) / d \llambda = (d \A^P (Y^u_{\cdot + t^*}) / d \llambda) (\, \cdot + \tau (\omega) - t^*, \omega \otimes_{\tau(\omega)} X).
    \]
    We get from Fubini's theorem, the tower rule and \(P \in \cA(t^*, \omega^*)\) that 
    \begin{align*}
       &\iint_{0}^\infty P\big( (d\A^{P (\, \cdot\, | \cF_\tau) (\omega)} (Y^u_{\cdot + \tau(\omega)}) / d \llambda)_{u \in U} \hspace{0.05cm} (t, X) \not \in \Theta (t + \tau(\omega),X) | \mathcal{F}_\tau \big) (\omega) dt P(d \omega) 
        \\&\hspace{0.4cm}= \iint_{0}^\infty P \big( (d \A^P (Y^u_{\cdot + t^*}) / d \llambda)_{u \in U} \hspace{0.05cm} (t + \tau(\omega) - t^*, X) \not \in \Theta (t + \tau(\omega), X)  | \mathcal{F}_\tau \big) (\omega) dt P(d \omega) 
        \\&\hspace{0.4cm} = \int  E^P \Big[ \int_{\tau (\omega) - t^*}^\infty \1 \{ (d \A^P (Y^u_{\cdot + t^*}) / d \llambda)_{u \in U} \hspace{0.05cm} (t, X) \not\in \Theta (t + t^*, X) \} dt \big| \mathcal{F}_\tau\Big] (\omega) P(d\omega)
        \\&\hspace{0.4cm} = \int  E^P \Big[ \int_{\tau - t^*}^\infty \1 \{ (d \A^P (Y^u_{\cdot + t^*}) / d \llambda)_{u \in U} \hspace{0.05cm} (t, X) \not\in \Theta (t + t^*, X) \} dt \big| \mathcal{F}_\tau\Big] (\omega) P(d\omega)
        \\&\hspace{0.4cm} = E^P \Big[ \int_{\tau - t^*}^\infty \1 \{ (d \A^P (Y^u_{\cdot + t^*}) / d \llambda)_{u \in U} \hspace{0.05cm} (t, X) \not \in \Theta (t + t^*, X) \} dt \Big] = 0.
    \end{align*}
    This observation completes the proof of \(P\)-a.s. \(P(\, \cdot\, | \mathcal{F}_\tau) \in \cA(\tau, X)\). 
\end{proof}

%-----------------------------------
%-----------------------------------

\subsection{Stability under Pasting}
In this section we check stability under pasting. The proof is split into several steps. Throughout this section, we fix \((t^*, \omega^*) \in \of 0, \infty\of\), a probability measure \(P\) on \((\Omega, \mathcal{F})\) such that \(P (X = \omega^* \text{ on } [0, t^*]) = 1\), a stopping time \(\tau\) with \(t^* \leq \tau < \infty\), and an \(\mathcal{F}_\tau\)-measurable map \(\Omega \ni \omega \mapsto Q_\omega \in \mathfrak{P}(\Omega)\) such that 
	\(Q_\omega ( X = \omega \text{ on } [0, \tau(\omega)]) = 1\) for \(P\)-a.a. \(\omega \in \Omega\). To simplicity our notation, we set \(\overline{P} \triangleq P \otimes_\tau Q\). 
The following corollary is an immediate consequence of part (ii) of Lemma \ref{lem: collection easy observations}.
\begin{corollary} \label{coro: meas identity}
		\(\overline{P} = E^P \big[ Q (\, \cdot\, ) \big]\).
\end{corollary}

The next lemma follows identically to \cite[Lemma 3.14]{CN22a}. We omit the details.
\begin{lemma} \label{lem: identiy up to tau}
\(P = \overline{P}\) on \(\mathcal{F}_\tau\) and  \(\overline{P}\)-a.s. \(\overline{P} (\, \cdot\, | \mathcal{F}_\tau) = Q\).
\end{lemma}

Clearly, Lemma \ref{lem: identiy up to tau} yields that \(\overline{P} ( X = \omega^* \text{ on } [0, t^*] ) = P ( X = \omega^* \text{ on } [0, t^*] ) = 1.\)

\begin{lemma} 
	\label{lem_pasting_s}
	Suppose that \(Y \in \fSac (t^*, P)\). If \(Y \in \fSac (\tau (\omega), Q_\omega)\) for \(P\)-a.a. \(\omega \in \Omega\), then \(Y \in \fSac (t^*, \overline{P})\).
\end{lemma} 

\begin{proof}
	\emph{Step 1: Semimartingale Property.}
    Let \(T > t^*\), and take a sequence \((H^n)_{n = 0}^\infty\) of simple predictable processes on \( \of t^*, T \gs \) such that \(H^n \to H^0\) uniformly in time and \(\omega\). Then, by Lemma~\ref{lem: identiy up to tau}, we get
	\begin{align*}
	E^{\overline{P}} \Big[ \Big| \int_{t^*}^T(H^n_s &- H_s) d Y_s\Big| \wedge 1 \Big] 
	\\&\leq E^{P} \Big[ \Big| \int_{t^*}^{\tau \wedge T} (H^n_s - H_s) d Y_s \Big| \wedge 1 \Big] \\&\qquad \qquad\qquad + \int E^{Q_\omega} \Big[ \Big|\int_{\tau (\omega) \wedge T}^T (H^n_s - H_s) (\omega \otimes_{\tau (\omega)} X) d Y_s\Big| \wedge 1 \Big] P(d \omega).
	\end{align*}
	The first term converges to zero as \(Y \in \fSss (t^*, P)\) and thanks to the Bichteler--Dellacherie (BD) Theorem (\cite[Theorem III.43]{protter}). The second term converges to zero by dominated convergence, the assumption that \(P\)-a.s. \(Y \in \fSss (\tau, Q)\) and, by virtue of part (iii) of Lemma~\ref{lem: collection easy observations}, again the BD Theorem. Consequently, invoking the BD Theorem a third time, but this time the converse direction, yields that \(Y \in \mathcal{S}(t^*, \overline{P})\).
	
	\emph{Step 2: Special Semimartingale Property}.
	Denote by \(\nu^{\overline{P}}\) (resp. \(\nu^P\)) the third characteristic of \(Y_{\cdot + t^*}\) under \(\overline{P}\) (resp. under \(P\)).
	By Lemma \ref{lem: identiy up to tau}, we obtain 
	\begin{align} \label{eq: first part finite special}
	    \overline{P} \big( (|x|^2 \wedge |x|) * \nu^{\overline{P}}_{\tau - t^*} < \infty \big) &= P \big( (|x|^2 \wedge |x|) * \nu^P_{\tau - t^*} < \infty \big) = 1. 
	\end{align}
	Furthermore, as \(Y \in \mathcal{S} (t^*, \overline{P})\) and \(\overline{P}\)-a.s. \(\overline{P}(\, \cdot\, |  \cF_\tau) = Q\) by Lemma \ref{lem: identiy up to tau}, we deduce from Lemma~\ref{lem: characteristics under conditioning} that, for \(\overline{P}\)-a.a. \(\omega \in \Omega\), the third \(Q_\omega\)-characteristic of \(Y_{\cdot + \tau(\omega)}\) is given by 
	\[
	\nu^{\overline{P}}  (\omega \otimes_{\tau(\omega)} X; \tau (\omega) - t^* + dt, dx).
	\]
	As \(P\)-a.s. \(Y \in \fSss (\tau, Q)\),  we have, for \(\overline{P}\)-a.a. \(\omega \in \Omega\) and all \(T > \tau (\omega) - t^*\), that
	\[
	Q_\omega \Big( \int_{\tau (\omega) - t^*}^T \int (|x|^2 \wedge |x|)\, \nu^{\overline{P}}  (\omega \otimes_{\tau(\omega)} X; dt, dx) < \infty \Big)= 1.
	\]
	By Lemmata \ref{lem: collection easy observations} [(i), (ii)] and \ref{lem: identiy up to tau}, this yields, for all \(T > 0\), that 
	\begin{equation} \label{eq: second part finite special}
	\begin{split}
	\overline{P} &\Big( \int_{\tau - t^*}^T \int (|x|^2 \wedge |x|)\, \nu^{\overline{P}}  (dt, dx) < \infty, T > \tau - t^* \Big) 
	\\&= \int_{\{T > \tau - t^*\}} \overline{P} \Big( \int_{\tau (\omega) - t^*}^T \int (|x|^2 \wedge |x|)\,\nu^{\overline{P}}  (\omega \otimes_{\tau(\omega)} X; dt, dx) < \infty \big| \cF_\tau \Big) (\omega) \overline{P}(d \omega)
	\\&= \int_{\{T > \tau - t^*\}} Q_\omega \Big( \int_{\tau (\omega) - t^*}^T \int (|x|^2 \wedge |x|)\, \nu^{\overline{P}}  (\omega \otimes_{\tau(\omega)} X; dt, dx) < \infty \Big) \overline{P}(d \omega)
	\\ &= \overline{P}(T > \tau - t^*).
	\end{split}
	\end{equation}
	Finally, \eqref{eq: first part finite special} and \eqref{eq: second part finite special} yield that 
	\[
	\overline{P} \big( (|x|^2 \wedge |x|) * \nu^{\overline{P}}_T < \infty \big) = 1, \quad T > 0.
	\]
	By virtue of \cite[Proposition II.2.29]{JS}, this implies that \(Y \in \fSss (t^*, \overline{P})\).
	
	\emph{Step 3: Absolutely Continuous  Compensator.} Let
	\[ \A^{\overline{P}} (Y_{\cdot + t^*}) = \int_{0}^\cdot \phi_s ds + \psi \]
	be the Lebesgue decomposition of (the paths) of $ \A^{\overline{P}} (Y_{\cdot + t^*})$, cf. \cite[Theorem 8.4.4]{benedetto}.
	Since \(Y \in \fSac (t^*, P)\) and $\overline{P} = P$ on $\cF_\tau$ by Lemma~\ref{lem: identiy up to tau}, we get that \( \A^{\overline{P}} (Y_{\cdot + t^*}) \ll \llambda\) on \(\of 0, \tau - t^*\gs\).
	Hence, by Corollary~\ref{coro: meas identity}, it suffices to show that
	$$ D \triangleq \Big\{ \A^{\overline{P}}_{\cdot + \tau - t^*} (Y_{\cdot + t^*}) -  \A^{\overline{P}}_{\tau - t^*} (Y_{\cdot + t^*}) \neq 
	          \int_{\tau - t^*}^{\cdot + \tau - t^*} \phi_s ds\Big
		    \}
    $$ is a \(Q_\omega\)-null set for \(\overline{P}\)-a.a. \(\omega \in \Omega\).
    Due to Lemmata \ref{lem: collection easy observations} (ii), \ref{lem: characteristics under conditioning} and \ref{lem: identiy up to tau}, for \(\overline{P}\)-a.a. \(\omega \in \Omega\), we have \(Q_\omega\)-a.s. 
    \[
    \A^{\overline{P}}_{\cdot + \tau (\omega) - t^*} (Y_{\cdot + t^*}) - \A^{\overline{P}}_{\tau (\omega) - t^*} (Y_{\cdot + t^*}) = \A^{Q_\omega} (Y_{\cdot + \tau(\omega)}).
    \]
    Since \(P\)-a.s. \(Y \in \fSac (\tau, Q)\), the uniqueness of the Lebesgue decomposition yields that, for \(\overline{P}\)-a.a. \(\omega \in \Omega\),
    \begin{align*} Q_\omega \Big( \A^{\overline{P}}_{\cdot + \tau (\omega) - t^*} (Y_{\cdot + t^*}) - \A^{\overline{P}}_{\tau (\omega) - t^*} &(Y_{\cdot + t^*}) \neq 
	          \int_{\tau (\omega) - t^*}^{\cdot + \tau (\omega) - t^*} \phi_s ds
                  \Big) 
                  \\&= 
                  Q_\omega \Big( \A^{Q_\omega} (Y_{\cdot + \tau(\omega)}) \neq 
                  \int_{\tau (\omega) - t^*}^{\cdot + \tau (\omega) - t^*} \phi_s ds
                  \Big)
                  = 0.
    \end{align*}
    Since \(Q_\omega (\tau = \tau(\omega)) = 1\) for \(\overline{P}\)-a.a. \(\omega \in \Omega\), 
    we have \(\overline{P}\)-a.s. \(Q(D) = 0\). This shows that \(\A^{\overline{P}}(Y_{\cdot + t^*}) \ll \llambda\) and hence, completes the proof of \(Y \in \fSac (t^*, \overline{P})\).
\end{proof}

\begin{lemma} \label{lem: stab pasting}
	Let $P \in \cA(t^*,\omega^*)$.
	If $Q_\omega \in \cA(\tau(\omega), \omega)$ for $P$-a.e. $\omega \in \Omega$, then
	$\overline{P} \in \cA(t^*,\omega^*)$.
\end{lemma}

\begin{proof}
Lemma \ref{lem_pasting_s} implies that \(Y^u \in \fSac  (t^*, \overline{P})\) for all \(u \in U\).  
Recall that \[\overline{P} ( X = \omega^* \text{ on } [0, t^*] ) = P ( X = \omega^* \text{ on } [0, t^*] ) = 1.\]
Consequently, as $\overline{P} = P$ on $\cF_\tau$, it suffices to show that
$$ R \triangleq \big\{ (t,\omega) \in \of\tau - t^*, \infty \of \hspace{0.02cm} \colon (d\A^{\overline{P}} (Y^u_{\cdot + t^*})  / d \llambda)_{u \in U} \hspace{0.05cm} (t, \omega)\notin \Theta(t + t^*,\omega)   \big\} $$
is a $(\llambda \otimes \overline{P})$-null set.
By virtue of Lemmata \ref{lem: characteristics under conditioning} and \ref{lem: identiy up to tau}, the assumption that \(P\)-a.s. \(Q \in \cA (\tau, X)\) yields, for  \(P\)-a.a. \(\omega \in \Omega\), that
\begin{align*}
 (\llambda &\otimes Q_\omega) \big((t,\omega') \in \of\tau (\omega) - t^*,  \infty \of \hspace{0.02cm} \colon (d\A^{\overline{P}} (Y^u_{\cdot + t^*}) / d \llambda)_{u \in U} \hspace{0.05cm} (t, \omega')\notin \Theta(t + t^*,\omega')\big)
 \\&= (\llambda \otimes Q_\omega) \big( (t, \omega') \in \of 0, \infty\of \hspace{0.02cm} \colon (d \A^{\overline{P}}_{\cdot + \tau (\omega) - t^*} (Y^u_{\cdot + t^*}) / d \llambda)_{u \in U} \hspace{0.05cm} (t, \omega')\not \in \Theta (t + \tau (\omega), \omega') \big)
 \\&= (\llambda \otimes Q_\omega) \big( (t, \omega') \in \of 0, \infty\of \hspace{0.02cm} \colon (d\A^{Q_\omega} (Y^u_{\cdot + \tau (\omega)}) / d \llambda)_{u \in U} \hspace{0.05cm} (t, \omega') \not \in \Theta (t + \tau (\omega), \omega') \big) = 0.
\end{align*}
Finally, as \(Q_\omega (\tau = \tau (\omega)) = 1\) for \(P\)-a.a. \(\omega \in \Omega\), Corollary \ref{coro: meas identity} and Fubini's theorem yield that
\begin{align*}
(\llambda \otimes \overline{P}) (R) &= E^{\overline{P}} \Big[ \int_{\tau - t^*}^{\infty} \1 \{ (d\A^{\overline{P}} (Y^u_{\cdot + t^*}) / d \llambda)_{u \in U} (t, X) \notin \Theta(t + t^*, X) \} d t \Big] 
\\&= E^P \Big[ E^Q \Big[ \int_{\tau - t^*}^{\infty} \1 \{ (d\A^{\overline{P}} (Y^u_{\cdot + t^*}) / d \llambda)_{u \in U} (t, X) \notin \Theta(t + t^*, X) \} d t \Big] \Big]
\\&= E^P \Big[ \int_{\tau - t^*}^{\infty} Q\big( (d\A^{\overline{P}} (Y^u_{\cdot + t^*}) / d \llambda)_{u \in U} (t, X) \notin \Theta(t + t^*, X) \big) dt \Big] = 0.
\end{align*}
This completes the proof.
\end{proof}
\subsection{Proof of Theorem \ref{theo: DPP}} \label{sec: pf DPD}
Corollary \ref{coro: MGC} and Lemmata \ref{coro: stab cond} and \ref{lem: stab pasting} yield that the prerequisites of \cite[Theorem 2.1]{ElKa15} are fulfilled and this theorem implies the DPP. \qed

%-------------------------------
%-------------------------------

\section{Proof of the Strong Markov Selection Principle: Theorem \ref{theo: strong Markov selection}} \label{sec: pf MSP}
The proof of the strong Markov selection principle is based on some fundamental ideas of Krylov \cite{krylov1973selection} on Markovian selections as worked out in the monograph \cite{stroock2007multidimensional} by Stroock and Varadhan. The main technical steps in the argument are to establish stability under conditioning and pasting for a suitable sequence of correspondences. Hereby, we adapt the proof of \cite[Theorem 2.19]{CN22b} to our more general framework.

This section is split into two parts. In the first, we study some properties of the correspondence \(\cK\) and thereafter, we finalize the proof. 

\subsection{Preparations}
Let us stress that in this section we do not assume Condition~\ref{cond: sammel}. In each of the following results, we indicate precisely which prerequisites are used (except for our Standing Assumptions \ref{SA: meas gr}, \ref{SA: non empty} and \ref{SA: markov}, which are always in force).

For \(t \in \mathbb{R}_+\), we define \(\gamma_t \colon \Omega \to \Omega\) by \(\gamma_t (\omega) \triangleq \omega ( (\, \cdot - t)^+ ) \) for \(\omega \in \Omega\).
Moreover, for \(P \in \mathfrak{P}(\Omega)\) and \(t \in \mathbb{R}_+\), we set \[P^t \triangleq P \circ \gamma_t^{-1}.\]
Recall also the notation \(P_t = P \circ \theta_t^{-1}\).
In case the path space \(\Omega\) is the space of continuous functions from \(\bR_+\) into \(F\), we have the following extension of Lemma \ref{lem: map prod meas}.

\begin{lemma} \label{lem: p^t cont} Suppose that \(\Omega = C(\bR_+; F)\).
	The maps \((t, P) \mapsto P_t\) and \((t, P) \mapsto P^t\) are continuous.
\end{lemma}
\begin{proof}
	Notice that \((t, \omega) \mapsto \theta_t (\omega)\) and \((t, \omega) \mapsto \gamma_t (\omega)\) are continuous by the Arzel\`a--Ascoli theorem.
	Now, the claim follows from \cite[Theorem 8.10.61]{bogachev}. 
\end{proof}

\begin{lemma} \label{lem: implication c^*}
	For every \((t, \omega) \in \of 0, \infty\of\), \(P \in \cC(t, \omega)\) implies \(P_t \in \cK (0, \omega(t))\).
\end{lemma}
\begin{proof}
	Let \((t, \omega) \in \of 0, \infty\of\), take \(P \in \cC(t, \omega)\) and fix \(u \in U\). It is clear that \( P_t \circ X_0^{-1} = \delta_{\omega(t)} \). Since, by hypothesis, \(Y^u_{\cdot + t} = Y^u \circ \theta_t\), we deduce from Lemma~\ref{lem: jacod restatements} and \cite[Proposition~II.2.29]{JS} that \(Y^u \in \fSac (P_t)\) and \(P\)-a.s.
	\(
\A^P (Y^u_{\cdot + t}) = \A^{P_t} (Y^u) \circ \theta_t .
	\)
	Hence, using also Standing Assumption \ref{SA: markov},
	it follows that
	\begin{align*}
	(\llambda \otimes P_t) ( (\A^{P_t} (Y^u))_{u \in U} \not \in \Theta ) &=  (\llambda \otimes P) ( (\A^{P_t} (Y^u) \circ \theta_t)_{u \in U} \not \in \Theta (\, \cdot\,, X \circ \theta_t) ) \\
	&= (\llambda \otimes P) ( (\A^{P} (Y^u_{\cdot + t}))_{u \in U}  \not \in \Theta (\, \cdot + t, X) ) = 0.
	\end{align*}
	We conclude that \(P_t \in \cK(0,\omega(t)) \).
\end{proof}

\begin{lemma} \label{lem: p^t}
	For every \( (t, x) \in \bR_+ \times F\), we have
	\( P \in \cK(0,x) \) if and only if \( P^t \in \cK(t,x) \).
\end{lemma}
\begin{proof}
	Take \( P \in \cK(0,x) \). As 
	\(
	\gamma_t^{-1} (\{ X = x \text{ on } [0,t] \} ) = \{ X_0 = x \},
	\)
	we have \(P^t( X = x \text{ on } [0,t]) = 1\). Take \(u \in U\) and notice that \(Y^u = Y^u \circ \theta_t \circ \gamma_t = Y^u_{\cdot + t} \circ \gamma_t\), as \(\theta_t \circ \gamma_t = \on{id}\).
	Hence, by Lemma~\ref{lem: jacod restatements} and \cite[Proposition II.2.29]{JS}, \(Y^u \in \fSac (t, P^t)\) and \(P\)-a.s.
	\(
	\A^P (Y^u) = \A^{P^t} (Y^u_{\cdot + t}) \circ \gamma_t.
	\)
	From the last equality, and Standing Assumption \ref{SA: markov}, we deduce that
	\begin{align*}
	(\llambda \otimes P^t) ( (\A^{P^t} (Y^u_{\cdot + t}))_{u \in U} &\not \in \Theta (\, \cdot + t, X) )
	\\&= (\llambda \otimes P) ( (\A^{P^t} (Y^u_{\cdot + t}))_{u \in U} \circ \gamma_t \not \in \Theta ((X \circ \gamma_t) (\, \cdot + t)) )
	\\&= (\llambda \otimes P)( (\A^{P} (Y^u))_{u \in U}  \not \in \Theta ) 
	\\&= 0.	\end{align*}
This proves that \( P^t \in \cK(t,x) \).
	Conversely, take \(P^t \in \cK(t,x) \). Due to the identity \( \theta_t \circ \gamma_t = \on{id} \), we have
	\( P =  (P^t)_t \) and Lemma \ref{lem: implication c^*} implies that \(P \in \cK(0,x) \). The proof is complete.
\end{proof}

\begin{lemma} \label{lem: k idenity}
	For every \((t, x) \in \bR_+ \times F\), we have \(\cK (t, x) = \{P^t \colon P \in \cK(0, x)\}\).
\end{lemma}
\begin{proof}
	Lemma \ref{lem: p^t} yields the inclusion \(\{P^t \colon P \in \cK(0, x)\} \subset \cK(t, x)\). Conversely, take \(P \in \cK(t, x)\). As 
	\( P( X = x \text{ on } [0,t]) = 1 \), we have
	\(P = (P_t)^t\). Now, since \(P_t \in \cK(0, x)\) by Lemma~\ref{lem: implication c^*}, we get \(\cK(t, x) \subset \{P^t \colon P \in \cK(0, x)\}\). 
\end{proof}

\begin{lemma} \label{lem: lip const lowe semi}
	Suppose that \(\Omega = C (\bR_+; F)\) and that \(\Omega \ni \omega \mapsto L (\omega) \in C(\bR_+; \bR_+)\) is continuous. Let \(P\) be a Borel probability measure on \(\Omega \times C(\bR_+; \bR)\) and set
		\[
	\zeta_M (\omega, \alpha) \triangleq \sup \Big\{ \frac{|\alpha (t \wedge \rho_M (\omega)) - \alpha(s \wedge \rho_M(\omega))|}{t - s} \colon 0 \leq s < t\Big\},
	\]
	for \((\omega, \alpha) \in \Omega \times C(\bR_+; \bR)\), where
	\[
	\rho_M = \inf \{t \geq 0 \colon L_t \geq M\}, \quad M > 0.
	\]
	 Then, there exists a dense set \(D \subset \bR_+\) such that, for every \(M \in D\), there exists a \(P\)-null set \(N = N(M)\) such that \(\zeta_M\) is lower semicontinuous at every \(\omega \not \in N\).
\end{lemma}
\begin{proof}
	In the first part of this proof, we adapt an argument from \cite[Lemma 11.1.2]{stroock2007multidimensional}.
	The map \(\omega \mapsto \rho_M (\omega)\) is lower semicontinuous. To see this, recall that
	\[
	\{\rho_M \leq t\} = \Big\{ \inf_{s \in \mathbb{Q} \cap [0, t]} \inf_{z \geq M} | L_s - z | = 0 \Big\}, \quad t \in \bR_+,
	\] 
	and notice that the set on the right hand side is closed, as \(\omega \mapsto \inf_{s \in \mathbb{Q} \cap [0, t]} \inf_{z \geq M} | L_s (\omega) - z |\) is continuous. 
	Next, set 
	\[
	\rho^+_M \triangleq \inf \{t \geq 0 \colon L_t > M\}.
	\]
	Since, for every \(t > 0\), 
	\[
	\{\rho^+_M < t\} = \bigcup_{\substack{s < t \\ s \in \mathbb{Q}_+}} \{ L_s > y \},
	\]
	where the set on the right hand side is open by the continuity of \(\omega \mapsto L(\omega)\), we conclude that \(\rho^+_M\) is upper semicontinuous. Notice that \(\rho^+_M = \rho_{M +} \triangleq \lim_{K\searrow M} \rho_K\). Define \(\Psi \colon \Omega \times C (\bR_+; \bR) \to \Omega\) by \(\Psi (\omega, \alpha) = \omega\), and \(\phi \colon \bR_+ \to [0, 1]\) by \(\phi (M) \triangleq E^P [e^{- \rho_M \circ \Psi}]\), and set 
	\begin{align*}
	D \triangleq \Big\{ M \in \bR_+ \colon E^P \big[ e^{- \rho_M \circ \Psi} \big] = E^P \big[ e^{- \rho_{M +} \circ \Psi} \big] \Big\}.
	\end{align*}
	The dominated convergence theorem yields that 
	\[
	D = \big\{ M \in \bR_+ \colon \phi (M) = \phi (M + )\big\}, \qquad \phi (M +) \triangleq \lim_{K \searrow M} \phi (K).
	\]
	The set \(\bR_+ \backslash D\) is countable, as monotone functions, such as \(\phi\), have at most countably many discontinuities. Since \(\rho_M \leq \rho^+_M\), for every \(M \in D\) we have
	\(P\)-a.s. \(\rho_M \circ \Psi = \rho^+_M \circ \Psi\). 
	To see this, notice that 
	\[
	E^P \big[ \big( e^{- \rho_M^+ \circ \Psi} - e^{- \rho_M \circ \Psi} \big) \1_{\{\rho_M \circ \Psi < \rho^+_M \circ \Psi\}} \big] = E^P \big[ e^{- \rho_M^+ \circ \Psi} - e^{- \rho_M \circ \Psi} \big] = 0,
	\]
	which implies that \(P (\rho_M \circ \Psi < \rho^+_M \circ \Psi) = 0\).

	Now, take \(\omega \in \{\rho_M = \rho^+_M\}\). Since \(\rho_M \leq \rho^+_M\), for every sequence \((\omega^n)_{n = 1}^\infty \subset \Omega\) with \(\omega^n \to \omega\), we obtain
	\[
	\rho_M(\omega) \leq \liminf_{n \to \infty} \rho_M(\omega^n) \leq \limsup_{n \to \infty} \rho	_M (\omega^n) \leq \limsup_{n \to \infty} \rho^+_M (\omega^n) \leq \rho^+_M (\omega) = \rho_M(\omega),
	\]
	which implies that \(\rho_M\) is continuous at \(\omega\). 
	
	Finally, let \(\Omega \times C(\bR_+; \bR) \ni (\omega^n, \alpha^n) \to (\omega, \alpha) \in \{\rho_M \circ \Psi = \rho^+_M \circ \Psi\}\). Then, 
	\begin{align*}
	\zeta_M (\omega, \alpha) &= \sup \Big\{ \liminf_{n \to \infty} \frac{| \alpha^n (t \wedge \rho_M(\omega^n)) - \alpha^n (s \wedge \rho_M (\omega^n))|}{t - s} \colon 0 \leq s < t\Big\}
	\\&\leq \liminf_{n \to \infty} \zeta_M (\omega^n, \alpha^n).
	\end{align*}
	The proof is complete.
\end{proof}

\begin{proposition} \label{prop: closedness}
	Assume that \textup{(i) -- (iv)} from Condition \ref{cond: sammel} hold.
Then, for every closed set \(K \subset F\), the set \(\cR(K) \triangleq \bigcup_{x \in K} \cR(x)\) is closed (in \(\mathfrak{P}(\Omega)\) endowed with the weak topology).
\end{proposition}
\begin{proof}
	We adapt the proof strategy from \cite[Proposition 3.8]{CN22b}. 
	Let \((P^n)_{n = 1}^\infty \subset \cR(K)\) be such that \(P^n \to P\) weakly. 
	By definition of \(\cR(K)\), for every \(n \in \mathbb{N}\), there exists a point \(x^n \in K\) such that \(P^n \in \cR(x^n)\). Since \(P^n \to P\) and \(\{\delta_x \colon x \in K\}\) is closed (\cite[Theorem~15.8]{charalambos2013infinite}), there exists a point \(x^0 \in K\) such that \(P \circ X^{-1}_0 = \delta_{x^0}\). In particular, \(x^n \to x^0\). We now prove that \(P \in \cR(x^0)\). Before we start our program, let us fix some auxiliary notation. We set \(\Omega^* \triangleq \Omega \times C (\bR_+; \bR)^U\) and endow this space with the product local uniform topology. In this case, the Borel \(\sigma\)-field \(\mathcal{B}(\Omega^*) \triangleq \cF^*\) coincides with \(\sigma (Z_t, t \geq 0)\), where \(Z = (Z^{(1)}, (Z^{(2, u)})_{u \in U})\) denotes the coordinate process on \(\Omega^*\). Furthermore, we define \(\mathbf{F}^* \triangleq (\cF^*_t)_{t \geq 0}\) to be the right-continuous filtration generated by the coordinate process \(Z\).
	
	\emph{Step 1.} In this first step, we show that the  family \(\{P^n \circ (X, (\A^{P^n}(Y^u))_{u \in U})^{-1} \colon n \in \mathbb{N}\}\) is tight when seen as a family of probability measures on \((\Omega^*, \cF^*)\). As \(\Omega^*\) is endowed with a product topology and since \(P^n \to P\), it suffices to prove that, for every fixed \(u \in U\), the family \(\{P^n \circ \A^{P^n} (Y^u)^{-1} \colon n \in \mathbb{N}\}\) is tight when seen as a family of Borel probability measures on \(C (\bR_+; \bR)\), endowed with the local uniform topology. 
	First, for every \(M > 0\), we deduce tightness of \(\{P^n \circ \A^{P^n}_{\cdot \wedge \rho_M}(Y^u)^{-1} \colon n \in \mathbb{N}\}\) from Kolmogorov's tightness criterion (\cite[Theorem 23.7]{Kallenberg}).
	By the first part from Condition \ref{cond: sammel} (iii), there exists a constant \(\C> 0\) such that, for all~\(s < t\),
	\begin{align*}
	\sup_{n \in \mathbb{N}}E^{P^n} \big[ |\A^{P^n}_{s \wedge \rho_M} (Y^u) - \A^{P^n}_{t \wedge \rho_M} (Y^u)|^2 \big] 
	&\leq \sup_{n \in \mathbb{N}} E^{P^n} \Big[ \Big(\int_{s \wedge \rho_M}^{t \wedge \rho_M} \Big| \frac{d \A^{P^n} (Y^u)}{d \llambda} \Big| d \llambda \Big)^{2} \Big]
\\&\leq \C (t - s)^2.
	\end{align*}
We conclude the tightness of \(\{P^n \circ \A^{P^n}_{\cdot \wedge \rho_M} (Y^u)^{-1} \colon n \in \mathbb{N}\}\), for every \(M > 0\). 
	Using the Arzel\`a--Ascoli tightness criterion given by \cite[Theorem 23.4]{Kallenberg}, we transfer this observation to the global family \(\{P^n \circ \A^{P^n} (Y^u)^{-1} \colon n \in \mathbb{N}\}\). For a moment, fix \(\varepsilon, N > 0\). By virtue of Condition \ref{cond: sammel} (iii), more precisely \eqref{eq: moment bound relax}, there exists an \(M_o > N\) such that
	\begin{align*}
	\sup_{n \in \mathbb{N}} P^n (\rho_{M_o} \leq N) &= \sup_{n \in \mathbb{N}} P^n \Big( \sup_{s \in [0, N]} L_s \geq M_o\Big) 
	\leq \varepsilon / 2.
	\end{align*}
	Thanks to the tightness of the family \(\{P^n \circ \A^{P^n}_{\cdot \wedge \rho_{M_o}}(Y^u)^{-1} \colon n \in \mathbb{N}\}\), there exists a compact set \(K_N \subset \bR\) such that 
	\[
	\sup_{n \in \mathbb{N}} P^n ( \A^{P^n}_{N \wedge \rho_{M_o}} (Y^u) \not \in K_N) \leq \varepsilon / 2.
	\]
	Consequently, we obtain that 
	\begin{align*}
	\sup_{n \in \mathbb{N}} P^n ( \A^{P^n}_N (Y^u) \not \in K_N) &\leq \sup_{n \in \mathbb{N}} P^n ( \A^{P^n}_{N \wedge \rho_{M_o}} (Y^u) \not \in K_N) + \sup_{n \in \mathbb{N}} P^n ( \rho_{M_o} \leq N) 
	\\&\leq \varepsilon /2 + \varepsilon / 2 = \varepsilon.
	\end{align*}
	This observation shows that the family \(\{P^n \circ \A^{P^n}_N(Y^u)^{-1} \colon n \in \mathbb{N}\}\) is tight.
	For \(r > 0, \omega \in C(\bR_+; \bR)\) and \(h > 0\), we define 
	\[
	w_{[0, r]} (\omega, h) \triangleq \sup \big\{ |\omega(s) - \omega(t)| \colon 0 \leq s, t \leq r, |s- t| \leq h\big\}.
	\]
	As \(\{P^n \circ \A^{P^n}_{\cdot \wedge \rho_{M_o}}(Y^u)^{-1} \colon n \in \mathbb{N}\}\) is tight, \cite[Theorem 23.4]{Kallenberg} yields that
	\begin{align*}
	\sup_{n \in \mathbb{N}} E^{P^n} \Big[ w_{[0, N]} &(\A^{P^n} (Y^u), h) \wedge 1 \Big] 
\\&\leq \sup_{n \in \mathbb{N}} E^{P^n} \Big[ w_{[0, N]} (\A^{P^n}_{\cdot \wedge \rho_{M_o}} (Y^u), h) \wedge 1 \Big] + \sup_{n \in \mathbb{N}} P^n (\rho_{M_o} \leq N)
	\\&\leq \sup_{n \in \mathbb{N}} E^{P^n} \Big[ w_{[0, N]} (\A^{P^n}_{\cdot \wedge \rho_{M_o}} (Y^u), h) \wedge 1 \Big] + \varepsilon/2 \to \varepsilon /2 
	\end{align*}
	as \(h \to 0\).
Using  \cite[Theorem 23.4]{Kallenberg} once again, but this time the converse direction, we conclude tightness of \(\{ P^n \circ \A^{P^n} (Y^u)^{-1} \colon n \in \mathbb{N}\}\), which implies those of \(\{P^n \circ (X, (\A^{P^n} (Y^u))_{u \in U})^{-1} \colon n \in \mathbb{N}\}\) (on the respective probability space). 
Up to passing to a subsequence, we can assume that \((P^n \circ (X, (\A^{P^n} (Y^u))_{u \in U})^{-1})_{n = 1}^\infty\) converges weakly (on the space \((\Omega^*, \cF^*)\)) to a probability measure \(Q\). 

	\emph{Step 2.} 
	Recall that \(Z = (Z^{(1)}, (Z^{(2, u)})_{u \in U})\) denotes the coordinate process on \(\Omega^*\) and fix some \(u \in U\).
	Next, we show that \(Z^{(2, u)}\) is \(Q\)-a.s. locally Lipschitz continuous. 
Thanks to Lemma~\ref{lem: lip const lowe semi}, there exists a dense set \(D \subset \mathbb{R}_+\) such that, for every \(M \in D\), the map \(\zeta_M\) is \(Q \circ (Z^{(1)}, Z^{(2, u)})^{-1}\)-a.s. lower semicontinuous. By part (iii) of Condition \ref{cond: sammel}, for every \(M > 0\), there exists a constant \(\C = \C (M) > 0\) such that \(P^n(\zeta_M (X, \A^{P^n} (Y^u)) \leq \C) = 1\) for all \(n \in \mathbb{N}\). Hence, for every \(M \in D\), using the \(Q \circ (Z^{(1)}, Z^{(2, u)})^{-1}\)-a.s. lower semicontinuity of \(\zeta_M\), we deduce from \cite[Example 17, p. 73]{pollard} that 
\[
0 = \liminf_{n \to \infty} P^n (\zeta_M (X, \A^{P^n} (Y^u)) > \C) \geq Q (\zeta_M (Z^{(1)}, Z^{(2, u)}) > \C). 
\]
Further, since \(D\) is dense in \(\mathbb{R}_+\), we can conclude that \(Z^{(2, u)}\) is \(Q\)-a.s. locally Lipschitz continuous and hence, in particular, locally absolutely continuous and of finite variation.
	
	\emph{Step 3.} 
	Define the map \(\Phi \colon \Omega^* \to \Omega\) by \(\Phi (\omega^{(1)}, \omega^{(2)}) \triangleq \omega^{(1)}\) for \(\omega = (\omega^{(1)}, \omega^{(2)}) \in \Omega^*\).
	In this step, we prove that \((\llambda \otimes Q)\)-a.e. \((d Z^{(2, u)}/ d \llambda)_{u \in U} \in \Theta \circ \Phi\). 
	By virtue of \cite[Theorems~II.4.3 and II.6.2]{sion}, \(P^n\)-a.s. for all \(t \in \mathbb{R}_+\), we have 
	\begin{equation}\label{eq: P as inclusion theta}
	\begin{split}
	\big(m (\A^{P^n}_{t + 1/m} (Y^u) - \A^{P^n}_t (Y^u))\big)_{u \in U}  &\in \oconv  \big( (d\A^{P^n} (Y^u) / d \llambda)_{u \in U} ([ t, t + 1/m ])\big) \\&\subset \oconv \Theta ([t, t + 1/m], X).
	\end{split}
	\end{equation}
	Recall that \(\oconv\) denotes the closure of the convex hull.
	By Skorokhod's coupling theorem (\cite[Theorem 5.31]{Kallenberg}), there exist random variables \[(X^0, (B^{0, u})_{u \in U}), (X^1, (B^{1, u})_{u \in U}), (X^2, (B^{2, u})_{u \in U}), \dots\] defined on some probability space \((\Sigma, \mathcal{G}, \mathsf{P})\) such that, for every \(n \in \mathbb{N}\), \((X^n, (B^{n, u})_{u \in U})\) has distribution \(P^n \circ (X, (\A^{P^n} (Y^u))_{u \in U})^{-1}\), \((X^0, (B^{0, u})_{u \in U})\) has distribution \(Q\), and \(\mathsf{P}\)-a.s. \((X^n, (B^{n, u})_{u \in U}) \to (X^0, (B^{0, u})_{u \in U})\). 
	Thanks to the assumption (see (iv) of Condition \ref{cond: sammel}) that the correspondence 
	\[
	\omega \mapsto \oconv \Theta ([t, t + 1/m], \omega)
	\]
	is upper hemicontinuous with compact values, we deduce from \eqref{eq: P as inclusion theta} and \cite[Theorem~17.20]{charalambos2013infinite} that, for every \(m \in \mathbb{N}\), \(\mathsf{P}\)-a.s. for all \(t \in \bR_+\)
	\begin{align} \label{eq: new incl}
	(m(B^{0, u}_{t + 1/m} - B^{0, u}_t))_{u \in U} &= \lim_{n \to \infty} (m(B^{n, u}_{t + 1/m} - B^{n, u}_t))_{u \in U}
	\in \oconv \Theta ([t, t + 1/m], X^0).
	\end{align}
	Notice that \((\llambda \otimes \mathsf{P})\)-a.e.
\begin{align} \label{eq: lebesgue theorem}
	(d B^{0, u} / d \llambda)_{u \in U} = \lim_{m \to \infty}  (m(B^{0 , u}_{\cdot + 1/m} - B^{0, u}_\cdot))_{u \in U}.
\end{align}
	Using \eqref{eq: new incl} and \eqref{eq: lebesgue theorem}, we deduce from (ii) and (iv) of Condition \ref{cond: sammel}, and \cite[Lemma 3.4]{CN22b}, that \(\mathsf{P}\)-a.s. for \(\llambda\)-a.a. \(t \in \mathbb{R}_+\) 
	\[
	(d B^{0, u} / d \llambda)_{u \in U} (t) \in \bigcap_{m \in \mathbb{N}} \oconv \Theta ([t, t + 1/m], X^0) \subset \Theta (t, X^0).
	\]
	This proves that \((\llambda \otimes Q)\)-a.e. \((d Z^{(2, u)} /d \llambda)_{u \in U} \in \Theta \circ \Phi\).
	
	\emph{Step 4.} In the final step of the proof, we show that \(Y^u \in \fSac (P)\) and we relate \(B^{0, u}\) to \(\A^P (Y^u)\). 
Notice that \(Q \circ \Phi^{-1} = P\). For a moment, we fix \(u \in U\).
As in the proof of Lemma \ref{lem: lip const lowe semi}, we obtain the existence of a dense set \(D = D^u \subset \bR_+\) such that \(\rho_M \circ \Phi\) is \(Q\)-a.s. continuous for all \(M \in D\). Take some \(M \in D\). Since \(Y^u \in \fSac (P^n)\), the process \(Y^u_{\cdot \wedge \rho_M} - \A^{P^n}_{\cdot \wedge \rho_M} (Y^u)\) is a \(P^n\)-\(\F_+\)-local martingale. Furthermore, by (i) and (iii) from Condition~\ref{cond: sammel}, we see that \(Y^u_{\cdot \wedge \rho_M} - \A^{P^n}_{\cdot \wedge \rho_M} (Y^u)\) is \(P^n\)-a.s. bounded by a constant independent of \(n\), which, in particular, implies that it is a true \(P^n\)-\(\F_+\)-martingale. Recall from part (i) of Condition~\ref{cond: sammel}, that \(\omega \mapsto Y^u(\omega)\) is continuous. Now, it follows\footnote{\cite[Proposition~IX.1.4]{JS} is stated for \(F = \bR^d\) but the argument needs no change to work for more general state spaces.} from \cite[Proposition~IX.1.4]{JS} that \[Y^{u}_{\cdot \wedge \rho_M} \circ \Phi - Z^{(2, u)}_{\cdot \wedge \rho_M \circ \Phi}\] is a \(Q\)-\(\F^*\)-martingale. Since \(Z^{(2, u)}\) is \(Q\)-a.s. locally absolutely continuous by Step 2, this means that 
\(Y^u \circ \Phi\) is a \(Q\)-\(\F^*\)-semimartingale whose first characteristic is given by~\(Z^{(2, u)}\). 
Next, we relate this observation to the probability measure \(P\) and the filtration~\(\F_+\). For a process \(A\) on \((\Omega^*, \cF^*)\), we denote by \(A^{p, \Phi^{-1}(\F_+)}\) its dual predictable projection to the filtration \(\Phi^{-1}(\F_+)\). Recall from \cite[Lemma~10.42]{jacod79} that, for every \(t \in \bR_+\), a random variable \(V\) on \((\Omega^*, \cF^*)\) is \(\Phi^{-1}(\cF_{t+})\)-measurable if and only if it is \(\cF^*_t\)-measurable and \(V (\omega^{(1)}, \omega^{(2)})\) does not depend on \(\omega^{(2)}\).
Thanks to Stricker's theorem (\cite[Lemma~2.7]{jacod80}), the process \(Y^u \circ \Phi\) is a \(Q\)-\(\Phi^{-1} (\F_+)\)-semimartingale. 
Recall from Step~3 that \((\llambda \otimes Q)\)-a.e. \((d Z^{(2, u)}/ d \llambda)_{u \in U} \in \Theta \circ \Phi\). 
By virtue of (iii) from Condition \ref{cond: sammel}, for every \(M \in D\), we have
\[
E^Q \big[ \on{Var} (Z^{(2, u)})_{\rho_M \circ \Phi} \big] = E^Q \Big[ \int_0^{\rho_M \circ \Phi} \Big| \frac{d Z^{(2, u)}}{d \llambda} \Big| d \llambda \Big] < \infty,
\]
where \(\on{Var} (\, \cdot\,)\) denotes the total variation process.
Hence, we get from \cite[Proposition 9.24]{jacod79} that the first \(Q\)-\(\Phi^{-1}(\F_+)\)-characteristic of \(Y^{u} \circ \Phi\) is given by \((Z^{(2, u)})^{p, \Phi^{-1}(\F_+)}\). 
Lemma \ref{lem: jacod restatements} yields that \(Y^u\) is a \(P\)-\(\F_+\)-semimartingale whose first characteristic \(\A^P (Y^u)\) satisfies 
\[\A^P (Y^u) \circ \Phi = (Z^{(2, u)})^{p, \Phi^{-1}(\F_+)}.\] 
Consequently, we deduce from the Steps~2 and 3 that \(P\)-a.s. \(\A^P (Y^u) \ll \llambda\) and 
\begin{align*}
(\llambda \otimes P) ( (\A^P (Y^u))_{u \in U} \not \in \Theta ) 
&= (\llambda \otimes Q \circ \Phi^{-1}) ( (d \A^P (Y^u) /d \llambda)_{u \in U} \not \in \Theta )
\\&= (\llambda \otimes Q) ( E^Q [ (d Z^{(2, u)} / d \llambda)_{u \in U} | \Phi^{-1} (\F_+)_-] \not \in \Theta \circ \Phi ) = 0,
\end{align*}
where we use (ii) and (iii) from Condition \ref{cond: sammel} and \cite[Theorems II.4.3 and II.6.2]{sion} for the final equality.
This observation completes the proof.
\end{proof}

Given the previous observations, the following result can be proved similar to \cite[Proposition 5.7]{CN22b}. For reader's convenience, we provide the details. 

\begin{proposition} \label{prop: K upper hemi and compact}
	Suppose that Condition \ref{cond: sammel} holds.
	The correpondence \((t, x) \mapsto \cK(t, x)\) is upper hemicontinuous with nonempty and compact values.
\end{proposition}
\begin{proof}
	The correspondence \( x \mapsto \cK(0, x) \) has nonempty values by Standing Assumption~\ref{SA: non empty}. Further, it has compact values by Proposition \ref{prop: closedness} and part (v) of Condition \ref{cond: sammel}. Thus, Lemmata~\ref{lem: p^t cont} and \ref{lem: k idenity} yield that the same is true for \((t, x) \mapsto \cK (t, x)\).
	
	It remains to show that \(\cK\) is upper hemicontinuous. 
	Let \( G \subset \mathfrak{P}(\Omega) \) be closed. We need to show that
	\( \cK^l(G) = \{ (t, x) \in \bR_+ \times F \colon \cK(t, x) \cap G \neq \emptyset \} \) is closed.
	Suppose that the sequence \( (t^n, x^n)_{n = 1}^\infty \subset \cK^l(G) \) converges to \((t^0, x^0) \in \bR_+ \times F\).
	For each \(n \in \mathbb{N} \), there exists a probability measure \(P^n \in \cK (t^n, x^n) \cap G \).
	By part (v) of Condition \ref{cond: sammel} and Proposition~\ref{prop: closedness}, the set \(\cR^\circ \triangleq \bigcup_{n = 0}^\infty \cR(x^n)\) is compact.
	Hence, by Lemma \ref{lem: p^t cont}, so is the set 
	\[
	\cK^\circ \triangleq \{ P^t \colon (t, P) \in \{t^n \colon n \in \mathbb{Z}_+\} \times \cR^\circ \}.
	\]
	By virtue of Lemma \ref{lem: k idenity}, we conclude that \(\{P^n \colon n \in \mathbb{N}\} \subset \cK^\circ\) is relatively compact. 
	Hence, passing to a subsequence if necessary, we can assume that \( P^n \to P \) weakly for some 
	\( P \in \cK^\circ \cap G\).
	Let \(d_F\) be a metric on \(F\) which induces its topology.
	 For every \(\varepsilon \in (0, t^0)\), the set \(\{d_F (X_s, x^0) \leq \varepsilon \text{ for all } s \in [0, t^0 - \varepsilon]\} \subset \Omega\) is closed. Consequently, by the Portmanteau theorem, for every \(\varepsilon \in (0, t^0)\), we get
	\begin{align*}
	1 &= \limsup_{n \to \infty} P^n( d_F (X_s, x^0) \leq \varepsilon \text{ for all } s \in [0, t^0 - \varepsilon] ) 
	\\&\leq P(d_F (X_s, x^0) \leq \varepsilon \text{ for all } s \in [0, t^0 - \varepsilon] ).
	\end{align*}
	It follows that \(P( X = x^0 \text{ on } [0, t^0] ) = 1\), which implies that \(P = (P_{t^0})^{t^0}\).
	By Lemmata~\ref{lem: p^t cont} and \ref{lem: implication c^*}, we have \((P^n)_{t^n} \in \cK(0, x^n)\) and \((P^n)_{t^n} \to P_{t^0}\) weakly. Further, since \(P_{t^0} \circ X_0^{-1} = \delta_{x^0}\), Proposition \ref{prop: closedness} yields that \(P_{t^0} \in \cK(0, x^0)\). 
	Thus, by virtue of Lemma~\ref{lem: p^t}, \(P \in \cK^\circ \cap G \cap \cK(t^0, x^0) = \cK(t^0, x^0) \cap G,\) which implies
	\((t^0, x^0) \in \cK^l(G) \). We conclude that \(\cK\) is upper hemicontinuous.
\end{proof}

\begin{lemma} \label{lem: r^* compact}
Suppose that part \textup{(ii)} of Condition \ref{cond: sammel} holds.
	The correspondence \( (t, x) \mapsto \cK(t, x) \) has convex values.
\end{lemma}
\begin{proof}
	By virtue of Lemma \ref{lem: k idenity}, it suffices to prove that \(\cK (0, x)\) is convex for every \(x \in F\).
	Indeed, for every \(P, Q \in \cK(t, x)\) and \(\alpha \in (0, 1)\), there are probability measures \(\p, \q \in \cK(0, x)\) such that \(\p^t = P\) and \(\q^t = Q\). Then, \(\alpha P + (1 - \alpha) Q = (\alpha \p + (1 - \alpha) \q)^t\) and consequently, from Lemma \ref{lem: p^t}, we get \(\alpha P + (1 - \alpha) Q \in \cK(t, x)\) once \(\alpha \p + (1 - \alpha) \q \in \cK(0, x)\).
	As \(\Theta\) is assumed to be convex-valued by Condition~\ref{cond: sammel} (ii), the convexity of \(\cK(0, x)\) follows from \cite[Lemma~III.3.38, Theorem~III.3.40]{JS}.
	To make the argument precise, take \(\p, \q \in \mathcal{K}(0, x)\) and \(\alpha \in (0, 1)\), and set \(\r \triangleq \alpha \p + (1 - \alpha)\q\). Clearly, \(\r \circ X^{-1}_0 = \delta_x\). Moreover, we have \(\p \ll \r\) and \(\q \ll \r\) and, by \cite[Lemma~III.3.38]{JS}, there are versions \(Z^\p\) and \(Z^\q\) of the Radon--Nikodym densities \(d \p/d \r\) and \(d \q/ d \r\), respectively, such that 
	\[
	\alpha Z^\p + (1 - \alpha) Z^\q = 1, \qquad 0 \leq Z^\p \leq 1 /\alpha, \qquad 0 \leq Z^\q \leq 1 / (1 - \alpha).
	\]
	For every \(u \in U\), \cite[Proposition II.2.29, Theorem~III.3.40]{JS} yields that \(Y^u \in \fSac (\r)\) and that \((\llambda \otimes \r)\)-a.e.
	\[
	d \A^\r (Y^u) /d \llambda = \alpha Z^\p d \A^\p (Y^u) / d\llambda + (1 - \alpha) Z^\q d \A^\q (Y^u) / d \llambda.
	\]
	Notice that 
	\begin{align*}
	\iint Z^\p \1 \{ (d\A^\p (Y^u) &/ d \llambda)_{u \in U} \not \in \Theta, Z^\p > 0\} d (\llambda \otimes \r) 
	\\&= (\llambda \otimes \p) ( (d \A^\p (Y^u) / d \llambda)_{u \in U} \not \in \Theta, Z^\p > 0) = 0.
	\end{align*}
	Consequently, \((\llambda \otimes \r)\)-a.e. \(\1 \{ (d \A^\p (Y^u) / d \llambda)_{u \in U} \not\in \Theta, Z^\p > 0\} = 0\). In the same manner, we obtain that \((\llambda \otimes \r)\)-a.e. \(\1 \{ (d \A^\q (Y^u) / d \llambda)_{u \in U} \not\in \Theta, Z^\q > 0\} = 0\).
	Finally, as \(\alpha Z^\p + (1 - \alpha) Z^\q = 1\), using the convexity of \(\Theta\), we get that 
	\begin{align*}
	(\llambda  \otimes \r) ( &( d \A^\r (Y^u) / d \llambda)_{u \in U} \not \in \Theta ) 
	\\&= (\llambda \otimes \r) ( (d\A^\p (Y^u) / d\llambda)_{u \in U} \not\in \Theta, Z^\p > 0, Z^\q = 0)
	\\&\qquad \quad + (\llambda \otimes \r) ( (d \A^\q (Y^u) / d \llambda)_{u \in U} \not\in \Theta, Z^\p = 0, Z^\q > 0 )
	\\&\qquad \quad + (\llambda \otimes \r) ( (d \A^\r (Y^u) / d \llambda)_{u \in U} \not\in \Theta, Z^\p > 0, Z^\q > 0 )
	\\&= (\llambda \otimes \r) ( (\alpha Z^\p ( d \A^\p (Y^u) / d \llambda)_{u \in U} + (1 - \alpha) Z^\q ( d \A^\q (Y^u) / d \llambda )_{u \in U} ) \not\in \Theta, 
	\\&\hspace{1.8cm} ( d \A^\p (Y^u) / d \llambda)_{u \in U} \in \Theta, ( d \A^\q (Y^u) / d \llambda)_{u \in U} \in \Theta, Z^\p > 0, Z^\q > 0 )
	\\&= 0.
	\end{align*}
	We conclude that \(\r \in \mathcal{K}(0, x)\). The proof is complete.
\end{proof}

\begin{lemma} \label{lem: iwie Markov}
	Let \(Q \in \mathfrak{P}(\Omega)\) and take \(t \in \bR_+, \omega, \alpha \in \Omega\) such that \(\omega (t) = \alpha (t)\).    Then, 
	\[
	\delta_\alpha \otimes_t Q \in \cC(t, \alpha) \quad \Longleftrightarrow \quad \delta_\omega \otimes_t Q \in \cC (t, \omega).
	\]
\end{lemma}
\begin{proof}
	Set \(\q \triangleq \delta_\alpha \otimes_t Q\) and \(\p \triangleq \delta_\omega \otimes_t Q\).
	Suppose that \( \q \in \cC(t,\alpha) \).
	Thanks to Lemma~\ref{lem: implication c^*}, we have \(\q_t \in \cC(0, \alpha (t)) = \cC(0, \omega (t))\). Since \(\q_t = \p_t\), we also have \(\p_t \in \cC(0, \omega (t))\). We deduce from Lemma \ref{lem: jacod restatements} that \(Y^u \in \fSac (t, \p)\) and that \(\p\)-a.s. \(\A^{\p} (Y^u_{\cdot + t}) = \A^{\p_t} (Y^u) \circ \theta_t\). 	Hence, using Standing Assumption \ref{SA: markov}, we get that
	\begin{align*}
	(\llambda \otimes \p) ( (\A^{\p} (Y^u_{\cdot + t}))_{u \in U} \not \in \Theta (\, \cdot + t, X) ) 
	&= (\llambda \otimes \p) ( (\A^{\p_t} (Y^u) \circ \theta_t)_{u \in U} \not \in \Theta (\, \cdot\,, X \circ \theta_t) )
	\\&= (\llambda \otimes \p_t) ( (\A^{\p_t} (Y^u))_{u \in U} \not \in \Theta  ) = 0.
	\end{align*}
	In summary, \( \p \in \cC(t, \omega) \). 
	The converse implication follows by symmetry.
\end{proof}

\begin{definition}
	A correspondence \(\cU \colon \bR_+ \times F \twoheadrightarrow \mathfrak{P}(\Omega)\) is said to be 
	\begin{enumerate}
		\item[\textup{(i)}]
		\emph{stable under conditioning} if for any \((t, x) \in \bR_+ \times F\), any stopping time \(\tau\) with \(t \leq \tau < \infty\), and any \(P \in \cU(t, x)\), there exists a \(P\)-null set \(N \in \cF_\tau\) such that 
		\(\delta_{\omega (\tau(\omega))} \otimes_{\tau (\omega)} P (\, \cdot\, | \mathcal{F}_\tau) (\omega) \in \cU(\tau (\omega), \omega (\tau (\omega)))\) for all \(\omega \not \in N\);
		\item[\textup{(ii)}]
		\emph{stable under pasting} if for any \((t, x) \in \bR_+ \times F\), any stopping time \(\tau\) with \(t \leq \tau < \infty\), any \(P \in \cU(t, x)\) and any \(\mathcal{F}_\tau\)-measurable map \(\Omega \ni \omega \mapsto Q_\omega \in \mathfrak{P}(\Omega)\) the following implication holds:
		\[
		\qquad \quad P\text{-a.a. } \omega \in \Omega \ \ \delta_{\omega (\tau(\omega))} \otimes_{\tau (\omega)} Q_\omega \in \cU (\tau (\omega), \omega (\tau(\omega)))\ \Longrightarrow \ P \otimes_\tau Q \in \cU(t, x).
		\]
	\end{enumerate}
\end{definition}

\begin{lemma} \label{lem: K stable under both}
	The correspondence \(\cK\) is stable under conditioning and pasting.
\end{lemma}
\begin{proof}
	Stability under conditioning follows from Lemmata \ref{coro: stab cond} and~\ref{lem: iwie Markov}, and stability under pasting follows from Lemmata \ref{lem: stab pasting} and \ref{lem: iwie Markov}.
\end{proof}

Recall from \cite[Definition 18.1]{charalambos2013infinite} that a correspondence \(\mathcal{U} \colon \bR_+ \times F \twoheadrightarrow \mathfrak{P}(\Omega)\) is called \emph{measurable} if the lower inverse \(\{ (t, x) \in \bR_+ \times F \colon \mathcal{U} (t, x) \cap G \not = \emptyset\}\) is Borel measurable for every closed set \(G \subset \mathfrak{P}(\Omega)\). The proof of the following result is similar to those of \cite[Lemma 5.12]{CN22b}. We added it for reader's convenience and because it explains why both stability properties, i.e., stability under conditioning and pasting, are needed, see also Remark \ref{rem: both stability needed} below.

\begin{lemma} \label{lem: U to U*}
	Suppose that \(\cU \colon \bR_+ \times F \twoheadrightarrow \mathfrak{P}(\Omega)\) is a measurable correspondence with nonempty and compact values such that, for all \((t, x) \in \bR_+ \times F\) and \(P \in \cU (t, x)\), \(P (X = x \text{ on } [0, t])= 1\). Suppose that \(\cU\) is stable under conditioning and pasting. Then, for every \(\phi \in \usc_b (F; \bR)\), the correspondence 
	\[
	\cU^* (t, x) \triangleq \Big\{ P \in \cU (t, x) \colon E^P \big[ \phi (X_T) \big] = \sup_{Q \in \cU (t, x)} E^Q \big[ \phi (X_T) \big] \Big\}, \quad (t, x) \in \bR_+ \times F,
	\]
	is also measurable with nonempty and compact values and it is stable under conditioning and pasting. Further, if \(\cU\) has convex values, then so does \(\cU^*\).
\end{lemma}
\begin{proof}
	We adapt the proof of \cite[Lemma 12.2.2]{stroock2007multidimensional}, see also the proofs of \cite[Lemma~3.4~(c), (d)]{hausmann86}.
	As \(\psi\) is assumed to be upper semicontinuous, \cite[Theorem 2.43]{charalambos2013infinite} implies that \( \cU^* \) has nonempty and compact values. Further, \cite[Theorem~18.10]{charalambos2013infinite} and \cite[Lemma~12.1.7]{stroock2007multidimensional} imply that \(\cU^*\) is measurable.
	The final claim about convexity is obvious. It is left to show that \(\cU^*\) is stable under conditioning and pasting.
	Take \((t, x) \in \bR_+ \times F, P \in \cU^* (t, x)\) and let \(\tau\) be a stopping time such that \(t \leq \tau < \infty\). We define 
	\begin{align*}
	N &\triangleq \big\{ \omega \in \Omega \colon P_\omega \triangleq \delta_{\omega (\tau (\omega))} \otimes_{\tau (\omega)} P (\, \cdot\, | \cF_\tau) (\omega) \not \in \cU (\tau (\omega), \omega (\tau (\omega)))\big\}, \\
	A &\triangleq \big\{ \omega \in \Omega \backslash N \colon P_\omega \not \in \cU^* (\tau (\omega), \omega (\tau (\omega)))\big\}.
	\end{align*}
	As \(\cU\) is stable under conditioning, we have \(P(N) = 0\). By \cite[Lemma 12.1.9]{stroock2007multidimensional}, \(N, A \in \cF_\tau\). As we already know that \(\cU^*\) is measurable, by virtue of \cite[Theorem 12.1.10]{stroock2007multidimensional}, there exists a measurable map \((s, y) \mapsto R (s, y)\) such that \(R (s, y) \in \cU^* (s, y)\). We set \(R_\omega \triangleq R (\tau (\omega), \omega (\tau (\omega)))\), for \(\omega \in \Omega\), and note that \(\omega \mapsto R_\omega\) is \(\cF_\tau\)-measurable. Further, we set 
	\[
	Q_\omega \triangleq \begin{cases} R_\omega, & \omega \in N \cup A,\\
	P_\omega, & \omega \not \in N \cup A. \end{cases}
	\]
	By definition of \(R\) and \(N\), \(Q_\omega \in \cU (\tau(\omega), \omega (\tau (\omega)))\) for all \(\omega \in \Omega\).
	As \(\cU\) is stable under pasting, \(P \otimes_{\tau} Q\in \cU(t, x)\). We obtain
	\begin{align*}
	&\sup_{Q^* \in \cU (t, x)} E^{Q^*} \big[ \phi (X_T) \big] 
	\\&\qquad\geq E^{P \otimes_{\tau} Q} \big[ \phi (X_T) \big]
	\\&\qquad= \int_{N\cup A} E^{\delta_\omega \otimes_{\tau (\omega)} R_\omega} \big[ \phi (X_T) \big] P (d \omega) + E^P \big[ \1_{N^c \cap A^c}E^P \big[ \phi (X_T) | \cF_\tau \big] \big] 
	\\&\qquad= \int_{A} E^{\delta_\omega \otimes_{\tau (\omega)} R_\omega} \big[ \phi (X_T) \big] P (d \omega) + E^P \big[ \1_{A^c}E^P \big[ \phi (X_T) | \cF_\tau \big] \big]
	\\&\qquad= \int_{A} \big[ E^{\delta_\omega \otimes_{\tau (\omega)} R_\omega} \big[ \phi (X_T) \big] - E^{\delta_\omega \otimes_{\tau (\omega)} P_\omega} \big[ \phi (X_T) \big] \big] P (d \omega) + \sup_{Q^* \in \cU (t, x)} E^{Q^*} \big[ \phi (X_T) \big]
	\\&\qquad= \int_{A} \big[ E^{R_\omega} \big[ \phi (X_T) \big] - E^{P_\omega} \big[ \phi (X_T) \big] \big] P (d \omega) + \sup_{Q^* \in \cU (t, x)} E^{Q^*} \big[ \phi (X_T) \big].
	\end{align*}
	As \(E^{R_\omega} \big[ \phi (X_T) \big] > E^{P_\omega} \big[ \phi (X_T) \big]\) for all \(\omega \in A\), we conclude that \(P (A) = 0\). This proves that \(\cU^*\) is stable under conditioning. 
	
	Next, we prove stability under pasting. Let \((t, x) \in \bR_+ \times F\), take a stopping time \(\tau\) with \(t \leq \tau < \infty\), a probability measure \(P \in \cU^*(t, x)\) and an \(\mathcal{F}_\tau\)-measurable map \(\Omega \ni \omega \mapsto Q_\omega \in \mathfrak{P}(\Omega)\) such that, for \(P\)-a.a. \(\omega \in \Omega\), 
	\(\delta_{\omega (\tau(\omega))} \otimes_{\tau (\omega)} Q_\omega \in \cU^* (\tau (\omega), \omega (\tau (\omega)))\). As \(\cU\) is stable under pasting, we have \(P \otimes_\tau Q \in \cU (t, x)\). Further, recall that \(\delta_{\omega (\tau (\omega))} \otimes_{\tau (\omega)} P (\, \cdot\, | \cF_\tau) (\omega)\in \cU (\tau (\omega), \omega (\tau (\omega)))\) for \(P\)-a.a. \(\omega \in \Omega\), as \(\cU\) is stable under conditioning. Thus, we get 
	\begin{align*}
	\sup_{Q^* \in \cU (t, x)} &E^{Q^*} \big[ \phi (X_T) \big] \\&\geq E^{P \otimes_\tau Q} \big[ \phi (X_T) \big] 
	\\&= \int E^{\delta_\omega \otimes_{\tau (\omega)} Q_\omega} \big[ \phi (X_T) \big] P (d \omega)
	\\&= \int E^{\delta_{\omega (\tau (\omega))} \otimes_{\tau (\omega)} Q_\omega} \big[ \phi (X_T) \big] \1_{\{\tau (\omega) < T\}}P (d \omega) + E^P \big[ \phi (X_T) \1_{\{T \leq \tau\}}\big]
	\\&= \int \sup_{Q^* \in \cU (\tau (\omega), \omega (\tau (\omega)))} E^{Q^*} \big[ \phi (X_T) \big] \1_{\{\tau (\omega) < T\}} P (d \omega) + E^P \big[ \phi (X_T) \1_{\{T \leq \tau\}}\big]
	\\&\geq \int E^{\delta_{\omega (\tau (\omega))} \otimes_{\tau (\omega)} P (\, \cdot\, | \cF_\tau)(\omega)} \big[ \phi (X_T) \big] \1_{\{\tau (\omega) < T\}} P (d \omega) + E^P \big[ \phi (X_T) \1_{\{T \leq \tau\}}\big]
	\\&= E^P \big[ E^P \big[ \phi (X_T) | \cF_\tau \big] \1_{\{\tau < T\}} \big] + E^P \big[ \phi (X_T) \1_{\{T \leq \tau\}}\big]
	\\&= E^P \big[ \phi (X_T) \big]
	\\&= \sup_{Q^* \in \cU (t, x)} E^{Q^*} \big[ \phi (X_T) \big].
	\end{align*}
	This implies that \(P \otimes_\tau Q\in \cU^* (t, x)\). The proof is complete.
\end{proof}

\subsection{Proof of Theorem \ref{theo: strong Markov selection}}
	We follow the proof of \cite[Theorem 2.19]{CN22b}, which adapts the proofs of \cite[Theorems 6.2.3 and 12.2.3]{stroock2007multidimensional}, cf. also the proofs of \cite[Proposition 6.6]{nicole1987compactification} and \cite[Proposition 3.2]{hausmann86}. In the following, Condition~\ref{cond: sammel} is in force.
	
	We fix a finite time horizon \(T > 0\) and a function \(\psi \in \usc_b (F; \bR)\).
	As \(F\) is a Polish space, there exists an equivalent metric \(d_F\) with respect to which \(F\) is totally bounded (see \cite[p. 43]{par} or \cite[p. 9]{stroock2007multidimensional}). Let \(U_{d_F} (F)\) be the space of all bounded uniformly continuous functions on \((F, d_F)\) endowed with the uniform topology. Then, by \cite[Lemma 6.3]{par}, \(U_{d_F}(F)\) is separable. Furthermore, by \cite[Theorem 5.9]{par}, \(U_{d_F} (F)\) is measure determining.   
	Let \(\{\sigma_n \colon n \in \mathbb{N}\}\) be a dense subset of \((0, \infty)\), let \(\{\phi_n \colon n \in \mathbb{N}\}\) be a dense subset of \(U_{d_F}(F)\) and let \((\lambda_N, f_N)_{N = 1}^\infty\) be an enumeration of \(\{(\sigma_m, \phi_n) \colon n, m \in \mathbb{N}\}\). 
	For \((t, x) \in \bR_+ \times F\), define inductively  
	\[
	\cU_0 (t, x) \triangleq \Big\{ P \in \cK (t, x)\colon E^P \big[ \psi (X_T) \big] = \sup_{Q \in \cK(t, x)} E^Q \big[ \psi (X_T) \big] \Big\}
	\]
	and 
	\[
	\cU_{N + 1} (t, x) \triangleq \Big\{ P \in \cU_N (t, x) \colon E^P \big[ f_{N + 1} (X_{\lambda_{N + 1}}) \big] = \sup_{Q \in \cU_N (t, x)} E^Q \big[ f_{N + 1} (X_{\lambda_{N + 1}}) \big] \Big\}
	\]
	for \(N \in \mathbb{Z}_+\).
	Furthermore, we set \[\cU_\infty (t, x) \triangleq \bigcap_{N = 0}^\infty \cU_N (t, x).\]
	
	Thanks to Proposition \ref{prop: K upper hemi and compact} and Lemmata \ref{lem: r^* compact} and \ref{lem: K stable under both}, the correspondence \(\cK\) is measurable with nonempty convex and compact values and it is further stable under conditioning and pasting. Thus, by Lemma \ref{lem: U to U*}, the same is true for \(\cU_0\) and, by induction, also for every \(\cU_N, N \in \mathbb{N}\). 
	As (arbitrary) intersections of convex and compact sets are itself convex and compact, \(\cU_\infty\) has convex and compact values. By Cantor's intersection theorem, \(\cU_\infty\) has nonempty values, and, by \cite[Lemma~18.4]{charalambos2013infinite}, \(\cU_\infty\) is measurable. Moreover, it is clear that \(\cU_\infty\) is stable under conditioning, as this is the case for every \(\cU_N, N \in \mathbb{Z}_+\). 
	
	We now show that \(\cU_\infty\) is singleton-valued. Take \(P, Q \in \cU_\infty (t, x)\) for some \((t, x) \in \bR_+ \times F\). By definition of \(\cU_\infty\), we have 
	\[
	E^P \big[ f_N (X_{\lambda_N}) \big] = E^Q \big[ f_N (X_{\lambda_N})\big], \quad N \in \mathbb{N}.
	\]
	This implies that \(P \circ X_s^{-1} = Q \circ X_s^{-1}\) for all \(s \in \bR_+\). Next, we prove that 
	\[
	E^P \Big[ \prod_{k = 1}^n g_k (X_{t_k}) \Big] = E^Q \Big[ \prod_{k = 1}^n g_k (X_{t_k}) \Big] 
	\]
	for all \(g_1, \dots, g_n \in C_b (F; \bR), t \leq t_1 < t_2 < \dots < t_n < \infty\) and \(n \in \mathbb{N}\). We use induction over \(n\). For \(n = 1\) the claim is implied by the equality \(P \circ X_s^{-1} = Q \circ X_s^{-1}\) for all \(s \in \bR_+\). Suppose that the claim holds for \(n \in \mathbb{N}\) and take test functions \(g_1, \dots, g_{n + 1} \in C_b (F; \bR)\) and times \(t \leq t_1 < \dots < t_{n + 1} < \infty\).  We define 
	\[
	\mathcal{G}_n \triangleq \sigma (X_{t_k}, k = 1, \dots, n).
	\]
	Since 
	\[
	E^P \Big[ \prod_{k = 1}^{n + 1} g_k (X_{t_k}) \Big] = E^P \Big[ E^P \big[ g_{n + 1} (X_{t_{n + 1}}) | \mathcal{G}_n \big] \prod_{k = 1}^n g_k (X_{t_k}) \Big], 
	\]
	it suffices to show that \(P\)-a.s. 
	\[
	E^P \big[ g_{n + 1} (X_{t_{n + 1}}) | \mathcal{G}_n \big] = E^Q \big[ g_{n + 1} (X_{t_{n + 1}}) | \mathcal{G}_n \big].
	\]
	As \(\cU_\infty\) is stable under conditioning, there exists a null set \(N_1 \in \cF_{t_n}\) such that \(\delta_{\omega (t_n)} \otimes_{t_n} P (\, \cdot\, | \cF_{t_n}) (\omega) \in \cU_\infty (t_n, \omega (t_n))\) for all \(\omega \not \in N_1\). Notice that, by the tower rule, there exists a \(P\)-null set \(N_2 \in \mathcal{G}_n\) such that, for all \(\omega \not \in N_2\) and all \(A \in \cF\), 
	\begin{equation} \label{eq: convex computation}
	\begin{split}
	\int \delta_{\omega' (t_n)} \otimes_{t_n} &P (A | \cF_{t_n}) (\omega') P (d \omega' | \mathcal{G}_n) (\omega) 
	\\&= \iint \1_A (\omega' (t_n) \otimes_{t_n} \alpha) P (d \alpha | \cF_{t_n}) (\omega' ) P (d \omega' | \mathcal{G}_n) (\omega) 
	\\&= \iint \1_A (\omega (t_n) \otimes_{t_n} \alpha) P (d \alpha | \cF_{t_n}) (\omega' ) P (d \omega' | \mathcal{G}_n) (\omega) 
	\\&=  \int \1_A (\omega (t_n) \otimes_{t_n} \omega') P (d \omega' | \mathcal{G}_n) (\omega) 
	\\&= (\delta_{\omega (t_n)} \otimes_{t_n} P (\, \cdot\, | \mathcal{G}_n) (\omega)) (A).
	\end{split}
	\end{equation}
	Let \(N_3 \triangleq \{P (N_1 | \mathcal{G}_n) > 0\} \in \mathcal{G}_n\). Clearly, \(E^P [ P (N_1| \mathcal{G}_n)] = P(N_1) = 0\), which implies that \(P (N_3) = 0\). For a moment, take \(\omega \not \in N_2 \cup N_3\). Using that \(\cU_\infty\) has convex and compact values and that \(\delta_{\omega' (t_n)} \otimes_{t_n} P (\, \cdot\, | \cF_{t_n}) (\omega') \in \cU_\infty (t_n, \omega' (t_n))\) for all \(\omega' \not \in N_1\), we obtain that 
	\[
	\int \delta_{\omega' (t_n)} \otimes_{t_n} P (A | \cF_{t_n}) (\omega') P (d \omega' | \mathcal{G}_n) (\omega) \in \cU_\infty (t_n, \omega (t_n)).
	\]
	Consequently, by virtue of \eqref{eq: convex computation}, we also have \(\delta_{\omega (t_n)} \otimes_{t_n} P (\, \cdot\, | \mathcal{G}_n) (\omega) \in \cU_\infty (t_n, \omega (t_n))\). Similarly, there exists a \(Q\)-null set \(N_4 \in \mathcal{G}_n\) such that \(\delta_{\omega (t_n)} \otimes_{t_n} Q (\, \cdot\, | \mathcal{G}_n) (\omega) \in \cU_\infty (t_n, \omega (t_n))\) for all \(\omega \not \in N_4\). Set \(N \triangleq N_2 \cup N_3 \cup N_4\). As \(P = Q\) on \(\mathcal{G}_n\), we get that \(P (N) = 0\). For all \(\omega \not \in N\), the induction base implies that 
	\begin{align*}
	E^P\big[ g_{n + 1} (X_{t_{n + 1}}) | \mathcal{G}_n\big] (\omega) &= E^{\delta_{\omega (t_n)} \otimes_{t_n} P (\, \cdot\, | \mathcal{G}_n)(\omega)} \big[ g_{n + 1} (X_{t_{n + 1}})\big] 
	\\&= E^{\delta_{\omega (t_n)} \otimes_{t_n} Q (\, \cdot\, | \mathcal{G}_n)(\omega)} \big[ g_{n + 1} (X_{t_{n + 1}})\big]
	\\&= E^Q\big[ g_{n + 1} (X_{t_{n + 1}}) | \mathcal{G}_n\big] (\omega).
	\end{align*}
	The induction step is complete and hence, \(P = Q\). 
	
	We proved that \(\cU_\infty\) is singleton-valued and we write \(\cU_\infty (s, y) = \{P_{(s, y)}\}\). By the measurability of \(\cU_\infty\), the map \((s, y) \mapsto P_{(s,y)}\) is measurable. It remains to show the strong Markov property of the family \(\{P_{(s, y)} \colon (s, y) \in \bR_+ \times F\}\). Take \((s, y) \in \bR_+ \times F\). As \(\cU_\infty\) is stable under conditioning, for every finite stopping time \(\tau \geq s\), there exists a \(P_{(s, x)}\)-null set \(N\) such that, for all \(\omega \not \in N\),
	\[
	\delta_{\omega (\tau (\omega))} \otimes_{\tau (\omega)} P_{(s, y)} (\, \cdot\, | \cF_{\tau})(\omega) \in \cU_\infty (\tau (\omega), \omega (\tau (\omega))) = \{P_{(\tau (\omega), \omega (\tau (\omega)))}\}.\]
	This yields, for all \(\omega \not \in N\), that 
	\[
	P_{(s, y)} (\, \cdot\, | \cF_\tau)(\omega) = \delta_\omega \otimes_{\tau (\omega)} \big[ \delta_{\omega (\tau (\omega))} \otimes_{\tau (\omega)} P_{(s, y)} (\, \cdot\, | \cF_{\tau})(\omega)\big] = \delta_\omega \otimes_{\tau (\omega)} P_{(\tau (\omega), \omega (\tau (\omega)))}.
	\]
	This is the strong Markov property and consequently, the proof is complete.
\qed

\begin{remark} \label{rem: both stability needed}
	Notice that the strong Markov property of the selection \(\{P_{(s, y)} \colon (s, y) \in \bR_+ \times F\}\) follows solely from the stability under conditioning property of \(\cU_\infty\). We emphasis that the stability under pasting property of each \(\cU_N, N \in \mathbb{Z}_+,\) is crucial for its proof. Indeed, in Lemma \ref{lem: U to U*}, the fact that \(\cU\) is stable under pasting has been used to establish that \(\cU^*\) is stable under conditioning.
\end{remark}

\section{Proof of the Nonlinear Markov Property: Proposition \ref{prop: markov property}} \label{sec: pf NMP}
In this section, the Standing Assumptions \ref{SA: meas gr}, \ref{SA: non empty} and \ref{SA: markov} are in force.
We need a refinement of Lemma \ref{lem: implication c^*}.

\begin{lemma} \label{lem: needed for MP}
	For every \((t, \omega) \in \of 0, \infty\of\), \(\cK(0, \omega (t)) = \{P_t \colon P \in \cC(t, \omega)\}\).
\end{lemma}
\begin{proof}
	The inclusion  \(\cK(0, \omega (t)) \supset \{P_t \colon P \in \cC(t, \omega)\}\) follows from Lemma \ref{lem: implication c^*}. For the converse inclusion, take \(P \in \cK(0, \omega (t))\). Then, by Lemmata \ref{lem: p^t} and \ref{lem: iwie Markov}, we have \(Q \triangleq \delta_\omega \otimes_t P^t \in \cC (t, \omega)\). 
	Since, for every \(A \in \cF\), it holds that
	\begin{align*}
	Q_t (A) = \int \1_{\theta^{-1}_t A} (\omega \otimes_t \omega') P^t (d \omega')
	= (P^t)_t (A) = P (A), 
	\end{align*}
	 the proof is complete. 
\end{proof}

We can now follow the proof of \cite[Proposition 2.8]{CN22b} to deduce Proposition \ref{prop: markov property} from the DPP (Theorem \ref{theo: DPP}).

\begin{proof}[Proof of Proposition \ref{prop: markov property}]
For every upper semianalytic function \(\psi \colon \Omega \to [- \infty, \infty]\) and any \((t, \omega) \in \of 0, \infty\of\), we deduce from Lemma \ref{lem: needed for MP} that 
	\[ \cE_t(\psi \circ \theta_t)(\omega) = \sup_{P \in \cC(t, \omega)} E^{P_t} \big[ \psi \big] =
	 \sup_{P \in \cK(0, \omega(t))} E^P \big[ \psi\big],
	\]
	which means nothing else than 
	\begin{align*}
	 \cE_t(\psi \circ \theta_t)(\omega) = \cE^{\omega(t)}(\psi).
	\end{align*}
	Now, the DPP (Theorem \ref{theo: DPP}) yields that 
	\[
	\cE^x(\psi \circ \theta_t )  = \cE^x( \cE_t( \psi \circ \theta_t)) 
	= \cE^x( \cE^{X_t}(\psi)).
	\]
	The proof is complete.
\end{proof}

\section{Proof of the Sublinear Semigroup Property: Proposition \ref{prop: semigroup}} \label{sec: pf SG}
First, let us show that \(T_t\), for every \(t \in \bR_+\), is a selfmap on the space of bounded upper semianalytic functions. Take a bounded upper semianalytic function \(\psi \colon F \to \bR\) and \(t \in \bR_+\). Clearly, \(T_t (\psi)\) is bounded. It remains to show that \(T_t (\psi)\) is also upper semianalytic. Define a function \(\phi \colon F \times \Omega \to [- \infty, \infty)\) by
\[
\phi (x, \omega) \triangleq \psi (\omega (t)) \1_{\{\omega (0) = x\}} + (- \infty) \1_{\{\omega (0) \not = x\}}.
\]
Evidently, \(\Omega \ni \omega \mapsto \omega (t)\) and \(F \times \Omega \ni (x, \omega) \mapsto \1_{\{\omega (0) = x\}}\) are Borel. Hence, by \cite[Lemma~7.30]{bershre}, the function \(\phi\) is upper semianalytic.
We set 
\[
\cK \triangleq \big\{ P \in \mathfrak{P}(\Omega) \colon P \in \cC(0, \omega) \text{ for some } \omega \in \Omega\big\}.
\]
By Lemma~\ref{lem: MGC Borel}, the correspondence \(\cC\) has a Borel measurable graph. Thus, the set 
\[
K \triangleq \big\{ (\omega, P) \in \Omega \times \mathfrak{P}(\Omega) \colon P \in \cC (0, \omega) \big\}
\]
is Borel measurable, and consequently, the set \(\cK\) is analytic as it is the image of \(K\) under the projection to the second coordinate (\cite[Proposition 7.39]{bershre}). 
Notice that
\[
T_t (\psi) (x) = \sup_{P \in \cK} E^P \big[ \phi (x, X) \big], \quad x \in F.
\]
We conclude from \cite[Propositions 7.47 and 7.48]{bershre} that \(x \mapsto T_t (\psi) (x)\) is upper semianalytic.

Finally, we discuss the properties (i) -- (iii) from Definition \ref{def: nonlinear MSG}. The properties (ii) and (iii) are trivially satisfied and the first property (i) is implied by Proposition \ref{prop: markov property}. The proof is complete. \qed

\section{Proof of the \(\usc_b\)-Feller Property: Theorem \ref{thm: USC Feller property}} \label{sec: pf USC FP}
Fix \(\psi \in \usc_b(F; \mathbb{R})\) and \(t \in \bR_+\), and notice that 
\(
\omega \mapsto \psi(\omega(t))
\)
is upper semicontinuous and bounded. Thus, by \cite[Theorem~8.10.61]{bogachev}, the map 
\(
\mathfrak{P}(\Omega) \ni P \mapsto E^{P}[\psi(X_{ t}) ]
\)
is upper semicontinuous, too.
Thanks to Proposition \ref{prop: K upper hemi and compact}, the correspondence \(x \mapsto \cR (x)\) is upper hemicontinuous and compact-valued. 
Thus, upper semicontinuity of 
\( x \mapsto T_t (\psi) (x) \)
follows from \cite[Lemma~17.30]{charalambos2013infinite}. The proof is complete.
\qed

\section{Proof of Proposition \ref{prop: conds hold for SPDE}} \label{sec: pf NSPDE}
The following section is devoted to the proof of Proposition \ref{prop: conds hold for SPDE}, which is split into several parts. In this section, we presume that Condition \ref{cond: main SPDE} holds.
\subsection{Proof of parts (i) and (ii) from Condition \ref{cond: sammel}}
Take \(u = \langle y, \cdot\, \rangle_{H_2}^i \in U\).
By virtue of \cite[Problem 13, p. 151]{EK}, the maps \(\omega\mapsto Y^u (\omega)  = \langle y, \omega\rangle_{H_2}^i\) and \(\omega \mapsto L (\omega) = \|\omega\|_{H_2}\) are continuous. Furthermore, for every \(R, M > 0\), we have \[|Y^u_{\cdot \wedge \rho_M}| \1_{\{\|X_0\|_{H_2} \leq R\}} \leq \|y\|^i_{H_2} (R + M)^i.\]
Thus, part (i) of Condition~\ref{cond: sammel} holds.

The correspondence \(\Theta\) is convex-valued by part (ii) of Condition \ref{cond: main SPDE}. Hence, also part (ii) of Condition \ref{cond: sammel} holds. \qed
\subsection{Proof of part (iii) from Condition \ref{cond: sammel}}
The linear growth assumption (iv) from Condition \ref{cond: main SPDE} implies the first part of Condition \ref{cond: sammel} (iii). The second part, i.e., \eqref{eq: moment bound relax} for \(L = \|X\|_{H_2}\), follows directly from the following lemma and Chebyshev's inequality. \qed

\begin{lemma} \label{lem: uni moment bound}
	For every bounded set \(K \subset H_2,\) every \(T > 0\) and every \(p > 1 / \alpha\), where \(\alpha \in (0, 1/2)\) is as in part \textup{(v)} of Condition \ref{cond: main SPDE}, 
	\[
	\sup_{x \in K} \sup_{P \in \cR(x)} E^P \Big[ \sup_{s \in [0, T]} \|X_s\|^p_{H_2} \Big] < \infty.
	\]
\end{lemma}

Before we prove this lemma, we record another useful observation which is used in the proof.

\begin{lemma} \label{lem: SPDE representation}
	For every \(x \in H_2\) and \(P \in \cR(x)\), there exists a predictable map \(\g \colon \of 0, \infty\gs \to G\) such that, possibly on a standard extension of the filtered probability space \((\Omega, \cF, \mathbf{F}, P)\), the coordinate process \(X\) admits the dynamics
	\begin{align*}
	X_t = S_tx + \int_0^t S_{t-s} \mu (\g (s, X), X_s) ds + \int_0^t S_{t-s} \sigma (\g (s, X), X_s) d W_s, \quad t \in \bR_+,
	\end{align*}
	where \(W\) is a cylindrical standard Brownian motion.
\end{lemma}

\begin{proof}
		Take \(x \in H_2\) and \(P \in \cR (x)\). Thanks to Proposition~\ref{prop: structure of R}, there exists a predictable function \(\g \colon \of 0, \infty\of \hspace{0.05cm} \to G\) such that the processes
	\[
	u (X) - \int_0^\cdot \mathscr{L}^{u, \g (s, X)} (X_s) ds, \qquad u = \langle y, \cdot\, \rangle_{H_2}^i, y \in D, i = 1, 2, 
	\]
	are \(P\)-local martingales. Recall that \(D \subset D (A^*)\) has the property that for every \(y \in D(A^*)\) there exists a sequence \((y_n)_{n = 1}^\infty \subset D\) such that \(y_n \to y\) and \(A^* y_n \to A^* y\). Hence, as ucp (uniformly on compacts in probability) limits of continuous local martingales are again continuous local martingales (see \cite[Lemma~B.11]{CE}), we conclude that all processes of the form
	\begin{align} \label{eq: test pr SPDE rep}
	u (X) - \int_0^\cdot \mathscr{L}^{u, \g (s, X)} (X_s) ds, \quad u = \langle y, \cdot\, \rangle_{H_2}^i, y \in D (A^*), i = 1,2, 
	\end{align}
	are \(P\)-local martingales.
	Thanks to \cite[Theorem 13]{MR2067962}, in our setting the concepts of analytically weak and mild solutions are equivalent for the SPDE 
	\[
	d Y_t = A Y_t dt + \mu (\g (t, Y), Y_t) dt + \sigma (\g (t, Y), Y_t) d W_t,
	\]
	where \(W\) is a cylindrical standard Brownian motion.
	Using the \(P\)-local martingale properties of the processes from \eqref{eq: test pr SPDE rep}, the claim of the lemma now follows as in the proof of \cite[Lemma~3.6]{criens21}. We omit the details.
\end{proof}

\begin{proof}[Proof of Lemma \ref{lem: uni moment bound}]
	\emph{Step 1: Recap on the Factorization Method.} 
	In this step, let \(P\) be a probability measure on some filtered probability space which supports a cylindrical standard \(P\)-Brownian motion~\(W\) over the Hilbert space \(H_2\). Take \(1/p < \lambda \leq 1\) and set, for \(h \in L^p ([0, T]; H_2)\),
	\begin{align} \label{eq: R def}
	R_\lambda (h) (t) \triangleq \int_0^t (t - s)^{\lambda - 1} S_{t-s} h (s) ds, \quad t \in [0, T].
	\end{align}
	Notice that \(R_\lambda\) is well-defined, as, by H\"older's inequality,
	\begin{equation} \begin{split} \label{eq: fak bound 1}
	\int_0^t (t - s)^{\lambda - 1} &\|S_{t-s} h (s)\|_{H_2} ds 
	\\&\leq \Big(\int_0^T s^{p (\lambda - 1)/(p - 1)} \|S_{s}\|^{p/(p - 1)}_{L(H_2, H_2)} ds\Big)^{(p - 1)/p} \Big( \int_0^t \|h(s)\|^{p}_{H_2} ds \Big)^{1/p}.
	\end{split}
	\end{equation}
	This inequality shows that 
	\[
	\|R_\lambda h (t)\|_{H_2} \leq \Big(\int_0^T s^{p (\lambda - 1)/(p - 1)} \|S_{s}\|^{p/(p - 1)}_{L(H_2, H_2)} ds\Big)^{(p - 1)/p} \Big( \int_0^T \|h(s)\|^{p}_{H_2} ds \Big)^{1/p},
	\]
	which means that \(R_\lambda\) is a bounded linear operator on \(L^{p}([0, T]; H_2)\). The following lemma provides more properties of \(R_\lambda\), which turn out to be useful in Section \ref{sec: pf part (v)} below.
	\begin{lemma}[\textup{\cite[Proposition 1]{gatarekgoldys}}] \label{lem: R compact}
		The operator \(R_\lambda\) maps \(L^{p}([0, T]; H_2)\) into \(C([0, T]; H_2)\). Moreover, if the semigroup \(S\) is compact, then \(R_\lambda\) is compact.
	\end{lemma}

	We now recall the factorization formula from \cite{DaPrato_fak}.
	Take \(\alpha \in (0, 1/2)\) and let \(p > 2\) be large enough such that \(1/p < \alpha\).
		Moreover, let \(\f\colon (0, T] \to [0, \infty]\) be a Borel function such that
	\[
	\int_0^T \Big[ \frac{\f (s)}{s ^\alpha} \Big]^2 ds < \infty,
	\]
	let \(\phi\) be a predictable \(L_1 (H_1, H_2)\)-valued process and let \(\psi\) be a real-valued predictable process such that
	\[
	\|S_t \phi_s \|_{L_2 (H_1, H_2)} \leq \f (t) |\psi_s|,
	\]
	for all \(t, s \in (0, T]\), and 
	\[
	E^P \Big[ \int_0^T |\psi_s|^p ds \Big] < \infty.
	\]
	The factorization formula given by \cite[Theorem~5.10]{DaPrato} shows that 
	\begin{align} \label{eq: fac fub}
	\int_0^t S_{t-s} \phi_s d W_s = \frac{\sin (\pi \alpha)}{\pi} R_\alpha (Y) (t), \quad t \in [0, T],
	\end{align}
	with
	\begin{align*} 
	Y_t \triangleq \int_0^t (t - s)^{- \alpha} S_{t-s} \phi_s d W_s,
	\end{align*}
	see also \cite[Step 0 in Section 4]{criens22} for more details. 
	By Eq. (4.4) in \cite{criens22}, it also holds that
	\begin{equation} \begin{split} \label{eq: fak bound 2}
	E^P \Big[ \int_0^T \|Y_t\|^{p}_E dt \Big] 
	%&\leq c_{p} \int_0^T E^P \Big[ \Big(\int_0^t (t - s)^{- 2 \alpha} \|S_{t - s} \phi_s\|_{L_2(H_1, H_2)}^2 ds \Big)^{p/2} \Big] dt
	\leq c_{p} \Big( \int_0^T \Big[\frac{\f (s) }{s^{\alpha}} \Big]^2 ds \Big)^{p/2}  E^P\Big[ \int_0^T |\psi_s|^p ds \Big],
	\end{split}
	\end{equation}
	where the constant \(c_p\) only depends on \(p\).
	In particular, \eqref{eq: fak bound 1} and \eqref{eq: fak bound 2} yield, for all \(r \in [0, T]\), that
	\begin{align} \label{eq: max ineq convolution}
	E^P \Big[ \sup_{t \in [0, r]} \Big\| \int_0^t S_{t - s} \phi_s d W_s \Big\|^p_{H_2} \Big] \leq \C E^P\Big[ \int_0^r |\psi_s|^p ds \Big],
	\end{align}
	where \(\C > 0\) is a constant which only depends on \(T, S, \f, \alpha\) and \(p\).
	
	\smallskip
	\emph{Step 2: A Gronwall Argument.} We are in the position to prove Lemma \ref{lem: uni moment bound}. Take \(x \in K\) and \(P \in \cR (x)\) and let \(\g\) be as in Lemma \ref{lem: SPDE representation}. Then, using the inequalities \eqref{eq: fak bound 1} and \eqref{eq: max ineq convolution}, for every \(M > 0\) and all \(r \in [0, T]\), we obtain 
	\begin{align*}
	E^P \Big[ \sup_{t \in [0, r \wedge \rho_M]} \|X_t\|^p_{H_2} \Big] &\leq \C \Big( \|x\|^p_{H_2} + E^P \Big[ \sup_{t \in [0, r \wedge \rho_M]} \Big\| \int_0^t S_{t - s} \mu (\g (s, X), X_s) ds \Big\|^p_{H_2} \Big]
	\\&\hspace{1.85cm} + E^P \Big[ \sup_{t \in [0, r \wedge \rho_M]} \Big\| \int_0^t S_{t - s} \sigma (\g (s, X), X_s) d W_s \Big\|^p_{H_2} \Big] \Big)
	\\&\leq \C\Big( \|x\|^p_{H_2} + E^P \Big[ \int_0^r \|\mu ( \g (s, X), X_{s \wedge \rho_M} )\|_{H_2}^p ds \Big] 
		\\&\hspace{1.85cm} + E^P \Big[ \sup_{t \in [0, r]} \Big\| \int_0^t S_{t - s} \sigma (\g (s, X), X_{s \wedge \rho_M}) d W_s \Big\|^p_{H_2} \Big] \Big)
		\\&\leq \C\Big( \|x\|^p_{H_2} + 1 + \int_0^r E^P \Big[ \sup_{s \leq t \wedge \rho_M} \|X_s\|^p_{H_2} \Big] dt \Big), 
		\end{align*}
		where the constant \(\C> 0\) depends only on the constant from Condition \ref{cond: main SPDE} (iv), \(T, S, \f, \alpha\) and \(p\).
		Finally, we conclude the claim of the lemma from Gronwall's and Fatou's lemma.
		The proof is complete.
\end{proof}

\subsection{Proof of part (iv) from Condition \ref{cond: sammel}}
By part (iii) of Condition~\ref{cond: main SPDE}, the map \((g, x) \mapsto (\mathscr{L}^{u, g} (x))_{u \in U}\) is continuous. Thus, thanks to \cite[Lemma~3.2]{CN22b}, for every \((t, \omega) \in \of 0, \infty\of\) and \(m \in \mathbb{N}\), the correspondences \(s \mapsto \Theta (\omega (s))\) and \(\alpha \mapsto \Theta (\alpha ([t, t + 1/m]))\) are continuous with compact values,
and, since \(\bR^U\) is a Fr\'echet space (\cite[p. 206]{charalambos2013infinite}), it follows also that \(\oconv \Theta (\omega ([t, t + 1/m]))\) is compact (\cite[Theorem 5.35]{charalambos2013infinite}). 
Finally, we deduce from \cite[Theorem 17.35]{charalambos2013infinite} that \(\omega \mapsto \oconv \Theta ([t, t + 1/m], \omega)\) is upper hemicontinuous for every \(t \in \bR_+\) and \(m \in \mathbb{N}\). This completes the proof.
\qed

\subsection{Proof of part (v) from Condition \ref{cond: sammel}} \label{sec: pf part (v)}
We adapt the compactness method from~\cite{gatarekgoldys}. Fix a finite time horizon \(T > 0\) and take \(p > 1/\alpha\), where \(\alpha\) is as in part (v) of Condition~\ref{cond: main SPDE}.
Let \(K \subset H_2\) be a compact set and, for \(R > 0\), define 
\begin{align*}
G_R \triangleq \Big\{ \omega &\in C ([0, T]; H_2)\colon \omega = S x_0 + R_1 (\psi) + \frac{\sin (\pi \alpha)}{\pi}R_\alpha (\phi), \ x_0 \in K, \\& \phi, \psi \in L^{p} ([0, T]; H_2)  \text{ with } \int_0^T \|\psi (s)\|_{H_2}^{p} ds \leq R, \ \int_0^T \|\phi (s)\|_{H_2}^{p} ds \leq R \Big\},
\end{align*}
where the operators \(R_1\) and \(R_\alpha\) are defined as in \eqref{eq: R def}.
For every \(t \in [0, T]\), the set \(\{S_t x_0 \colon x_0 \in K\}\) is compact, because \(S\) is a compact semigroup by Condition~\ref{cond: main SPDE}~(v). By \cite[Lemma I.5.2]{EN00}, the map
\(
[0, T] \times K \ni (t, x) \mapsto S_t x \in E
\)
is uniformly continuous. Thus, the Arzel\`a--Ascoli theorem (\cite[Theorem A.5.2]{Kallenberg}) yields that the set \(\{([0, T] \ni t \mapsto S_t x_0) \colon x_0 \in K\}\) is relatively compact in \(C([0, T]; H_2)\). Recall from Lemma~\ref{lem: R compact} that \(R_1\) and \(R_\alpha\) are compact operators from \(L^p ([0, T]; H_2)\) into \(C ([0, T]; H_2)\). Hence, for every \(R > 0\), the set \(G_R\) is relatively compact in \(C([0, T]; H_2)\). 

For any \(P \in \cR (x_0)\), by Lemma \ref{lem: SPDE representation} and the factorization formula \eqref{eq: fac fub}, possibly on an enlargement of the filtered probability space \((\Omega, \cF, \mathbf{F}, P)\), we have a.s.
\[
X_t = S_t x_0 + R_1 (\mu (\g, X)) (t) + \frac{\sin(\pi \alpha)}{\pi} R_\alpha (Y) (t), \quad t \in [0, T], 
\]
where
\[
Y_t \triangleq \int_0^t (t - s)^{- \alpha} S_{t - s} \sigma (\g (s, X), X_s) d W_s, \quad t \in [0, T].
\]
Using this observation and Chebyshev's inequality, we deduce from \eqref{eq: fak bound 1} and \eqref{eq: fak bound 2}, the linear growth assumptions on \(\mu\) and \(\sigma\), and Lemma~\ref{lem: uni moment bound}, that, for every \(\varepsilon> 0\), there exists an \(R > 0\) such that 
\begin{align*}
\sup_{x \in K} \sup_{P\in \cR(x)} P ( ([0, T] \ni t \mapsto X_t) \in G_R ) \geq 1 - \varepsilon.
\end{align*}
By virtue of \cite[Theorem 23.4]{Kallenberg}, since \(T > 0\) was arbitrary, we conclude that the set \(\bigcup_{x \in K} \cR(x)\) is tight and hence, relatively compact by Prohorov's theorem. The proof is complete. \qed

\section{Proof of Strong \(\usc_b\)--Feller property: Theorem \ref{theo: strong USC NSPDE}} \label{sec: pf strong USC}
The following proof adapts the idea from \cite{CN22c} that the strong \(\usc_b\)--Feller property can be deduced from the the strong Markov selection principle and a change of measure.
Throughout this proof, we fix a function \(\psi \in \usc_b (F; \bR)\) and a time \(T > 0\). W.l.o.g., assume that \(|\psi| \leq 1\).
We start with some auxiliary observations. 
Notice that Condition~\ref{cond: strong USC Feller NSPDE} implies Condition~\ref{cond: main SPDE}. Hence, by Corollary \ref{coro:SMSP NSPDE}, there exists a strong Markov family \(\{P_{(s, x)} \colon (s, x) \in \bR_+ \times H\}\) such that \(P_{(s, x)} \in \cK(s, x)\) and \(T_T (\psi) (x) = E^{P_{(0, x)}} [ \psi (X_T)]\). To simplify our notation, we set \(P_x \triangleq P_{(0, x)}\).
By Lemma \ref{lem: SPDE representation}, with some abuse of notation, for every \(x \in H\), there exists a predictable map \(\g^x \colon \of 0, \infty\of \, \to G\) and a cylindrical standard \(P_x\)-Brownian motion \(W\) such that \(P_x\)-a.s.
\begin{align*} 
X_t = S_tx + \int_0^t S_{t-s} \mu (\g^x  (s, X), X_s) ds + \int_0^t S_{t-s} d W_s, \quad t \in \bR_+.
\end{align*}
We define 
\[
Z^{P_x} \triangleq \exp \Big( -\int_0^\cdot \langle \mu (\g^x (s, X), X_s), d W_s \rangle_H - \frac{1}{2} \int_0^\cdot \| \mu (\g^x (s, X), X_s)\|^2_H ds \Big), 
\]
and, for \(M > 0\), we set 
\begin{align*}
\rho_M \triangleq \inf \{t \geq 0 \colon \|X_t\|_H \geq M\} \wedge M.
\end{align*}
For every \(M > 0\), \cite[Theorem 8.25]{jacod79} (or Novikov's condition \cite[Theorem 19.24]{Kallenberg}) yields that \(Z^{P_x}_{\cdot \wedge \rho_M}\) is a uniformly integrable \(P_x\)-martingale. 
We define a probability measure \(Q^M_x\) on \((\Omega, \cF)\) via the Radon--Nikodym derivative \(d Q^M_x / d P_x = Z^x_{\rho_M}\). 
Thanks to Girsanov's theorem (cf. \cite[Proposition~I.0.6]{roeckner15}), the process
\[
B^{M} \triangleq W + \int_0^{\cdot \wedge \rho_M} \mu (\g^x (s, X), X_s) ds 
\]
is a cylindrical standard \(Q_x^M\)-Brownian motion. Thus, under \(Q_x^M\), we have
\[
X_t = S_tx + \int_0^t S_{t - s} d B^{M}_s, \quad t  \leq \rho_M.
\]
Now, for every \(t \in [0, M)\), we obtain 
\begin{equation} \label{eq: OU bound}
\begin{split}
Q_x^M (\rho_M \leq t) &\leq \frac{1}{M} E^{Q_x^M} \Big[ \sup_{s \in [0, t \wedge \rho_M]} \|X_s\|_H \Big] 
\\&= \frac{1}{M} E^{Q_x^M} \Big[ \sup_{s \in [0, t \wedge \rho_M]} \Big\| S_s x + \int_0^s S_{s - r} B^{M}_r\Big\|_H \Big]
\\&\leq \frac{1}{M} E^{Q_x^M} \Big[ \sup_{s \in [0, t]} \Big\| S_sx + \int_0^s S_{s - r} B^{M}_r\Big\|_H \Big]
\\&= \frac{1}{M} E^{P_x} \Big[ \sup_{s \in [0, t]} \Big\|S_s x + \int_0^s S_{s - r} W_r\Big\|_H \Big] \to 0 \text{ as } M \to \infty.
\end{split}
\end{equation}
Hence,
\[
E^{P_x} \big[ Z^{P_x}_t \big] = \lim_{M\to \infty} E^{P_x} \big[ Z^{P_x}_t \1_{\{\rho_M > t\}} \big] = \lim_{M \to \infty} Q^M_x (\rho_M> t) = 1, 
\]
where we use the monotone convergence theorem for the first equality. This shows that \(Z^{P_x}\) is a true \(P_x\)-martingale.
Consequently, by a standard extension theorem (\cite[Lemma~19.19]{Kallenberg}), there exists a probability measure \(Q_x\) on \((\Omega, \cF)\) such that \(Q_x (G) = E^{P_x}[ Z_T^{P_x} \1_G ]\) for all \(G \in \mathcal{F}_T\) and \(T > 0\). Using again Girsanov's theorem (cf. \cite[Proposition~I.0.6]{roeckner15}), we obtain that the process
\[
B \triangleq W + \int_0^\cdot \mu (\g^x (s, X), X_s) ds 
\]
is a cylindrical standard \(Q_x\)-Brownian motion and, under \(Q_x\), 
\[
X_t = S_tx + \int_0^t S_{t - s} d B_s, \quad t \in \bR_+.
\]

Now, we are in the position to prove Theorem \ref{theo: strong USC NSPDE}. 
Let \((x^n)_{n = 0}^\infty \subset H\) be a sequence such that \(x^n \to x^0\) and fix \(\varepsilon > 0\). By virtue of \eqref{eq: OU bound}, there exists an \(M' > 0\) such that 
\begin{align} \label{eq: stopping time control 1}
\sup_{n \in \mathbb{Z}_+} Q_{x^n} (\rho_{M'} \leq T) \leq \varepsilon. 
\end{align}
Furthermore, by Lemma \ref{lem: uni moment bound}, there exists an \(M^\circ > 0\) such that 
\begin{align} \label{eq: stopping time control 2}
\sup_{n \in \mathbb{Z}_+} P_{x^n} (\rho_{M^\circ} \leq T) \leq \varepsilon. 
\end{align}
We set \(M \triangleq M' \vee M^\circ\). By the linear growth part of Condition \ref{cond: strong USC Feller NSPDE}, there exists a constant \(\C = \C_M > 0\) such that 
\[
\| \mu (g, x) \|_H \leq \C
\]
for all \(g \in G\) and \(x \in H \colon \|x\|_H \leq M\). 
Take \(\beta \in (0, T)\) and notice that, for all \(n \in \mathbb{Z}_+\),
\begin{align*}
	\big( E^{P_{x^n}} \big[ |1 - Z^{P_{x^n}}_{\beta \wedge \rho_M}| \big] \big)^2 &\leq E^{P_{x^n}} \big[ |1 - Z^{P_{x^n}}_{\beta \wedge \rho_M}|^2 \big]
	= E^{P_{x^n}} \big[ \big(Z^{P_{x^n}}_{\beta \wedge \rho_M}\big)^2 \big] - 1 
	\leq e^{\beta \C^2} - 1.
\end{align*}
Now, let \(\beta \in (0, T)\) small enough such that, for all \(n \in \mathbb{Z}_+\),
\begin{align} \label{eq: second moment Z control}
	E^{P_{x^n}} \big[ |1 - Z^{P_{x^n}}_{\beta \wedge \rho_M}| \big] \leq \varepsilon.
\end{align}
Define
\[
\Psi_\beta (x) \triangleq E^{P_{\beta, x}} \big[ \psi (X_{T}) \big], \quad x \in H.
\]
Then, for every \(n \in \mathbb{N}\), using the (strong) Markov property of \(\{P_{(s, x)} \colon (s, x) \in \bR_+ \times H\}\) and \eqref{eq: stopping time control 1}, \eqref{eq: stopping time control 2} and \eqref{eq: second moment Z control}, we obtain 
\begin{align*}
	\big| E^{P_{x^n}} \big[ \psi (X_T) \big] &- E^{P_{x^0}} \big[ \psi (X_T) \big] \big| 
	\\&= \big| E^{P_{x^n}} \big[ \Psi_\beta (X_\beta) \big] - E^{P_{x^0}} \big[ \Psi_\beta (X_\beta) \big] \big|
	\\&\leq \big| E^{P_{x^n}} \big[ \Psi_\beta (X_\beta) \1_{\{\rho_M > \beta\}} \big] - E^{P_{x^0}} \big[ \Psi_\beta (X_\beta) \1_{\{\rho_M> \beta\}}\big] \big| + 2 \varepsilon
	\\&\leq \big| E^{P_{x^n}} \big[ Z^{P_{x^n}}_\beta \Psi_\beta (X_\beta) \1_{\{\rho_M > \beta\}} \big] - E^{P_{x^0}} \big[ Z^{P_{x^0}}_\beta\Psi_\beta (X_\beta) \1_{\{\rho_M> \beta \}}\big] \big| 
	\\&\hspace{3.5cm} + E^{P_{x^n}} \big[ |1 - Z^{P_{x^n}}_{\beta \wedge \rho_M}| \big] + E^{P_{x^0}} \big[ |1 - Z^{P_{x^0}}_{\beta \wedge \rho_M}| \big]  + 2 \varepsilon
	\\&\leq \big| E^{Q_{x^n}} \big[ \Psi_\beta (X_\beta) \1_{\{\rho_M > \beta\}} \big] - E^{Q_{x^0}} \big[\Psi_\beta (X_\beta) \1_{\{\rho_M> \beta\}}\big] \big| 
	+ 4 \varepsilon
	\\&\leq \big| E^{Q_{x^n}} \big[ \Psi_\beta (X_\beta)\big] - E^{Q_{x^0}} \big[\Psi_\beta (X_\beta)\big] \big| 
	+ 6 \varepsilon.
\end{align*}
Thanks to \cite[Theorem 9.32]{DaPrato}, we have
\[
\big| E^{Q_{x^n}} \big[ \Psi_\beta (X_\beta) \big] - E^{Q_{x^0}} \big[ \Psi_\beta (X_\beta) \big] \big| \to 0
\]
as \(n \to \infty\). 
Hence, there exists an \(N \in \mathbb{N}\) such that, for all \(n \geq N\), 
\begin{align*}
	\big| E^{P_{x^n}} \big[ \psi (X_T) \big] - E^{P_{x^0}} \big[ \psi (X_T) \big] \big| \leq 7 \varepsilon.
\end{align*}
We conclude that the map \(x \mapsto E^{P_x} [ \psi (X_T) ]\) is continuous. This finishes the proof of Theorem \ref{theo: strong USC NSPDE}.
\qed

\section*{Acknowledgment}
The author thanks Lars Niemann for many interesting discussions. Moreover, helpful comments of an anonymous referee are gratefully acknowledged. 

\appendix

\section{Some facts on the space of L\'evy measures}
As in Section \ref{sec: nonlinear levy}, let \(\mathcal{L}\) be the space of all L\'evy measure on \(\bR^d\), i.e., the space of all measures \(\K\) on \((\bR^d, \mathcal{B}(\bR^d))\) such that 
\[K(\{0\} = 0 \quad \text{ and } \quad \int (1 \wedge \|x\|^2) \K (dx) < \infty.\] 
We endow \(\mathcal{L}\) with the weakest topology under which all maps \[
\K \mapsto \int f(x) (1 \wedge \|x\|^2) \K (dx), \qquad  f \in C_b (\bR^d; \bR),\] are continuous. The following lemma shows that this topology has a convenient structure.  It extends \cite[Lemma 2.3]{neufeld2014measurability}, where \(\mathcal{L}\) is shown to be a separable metrizable space.
\begin{lemma} \label{lem: Levy polish}
	The space \(\mathcal{L}\) is a Polish space. 
\end{lemma}
\begin{proof}
	The proof is split into two steps. 
	
	\smallskip
	{\em Step 1.}
	Let \(\mathcal{M}^f (\bR^d)\) be the set of all finite measures on \((\bR^d, \mathcal{B}(\bR^d))\) endowed with the weak topology. This space is Polish by \cite[Theorem 1.11]{prok}. We now prove that
	\[
	\mathcal{M}^f_0 (\bR^d) \triangleq \big\{ K \in \mathcal{M}^f (\bR^d) \colon K (\{0\}) = 0 \big\}
	\]
	is a \(G_\delta\) subset of \(\mathcal{M}^f (\bR^d)\) and consequently, a Polish subspace. 
	Let \((g_n)_{n = 1}^\infty\) be a sequence of continuous functions from \(\bR^d\) into \([0, 1]\) such that \(g_n \searrow \1_{\{0\}}\). It is well-known that such a sequence exists. Now, set 
	\[
	H \triangleq \bigcap_{n = 1}^\infty \bigcup_{m = 1}^\infty \Big\{ K \in \mathcal{M}^f (\bR^d) \colon \int g_m d K < 1/n \Big\}.
	\]
	The map \(K \mapsto \int g_m d K\) is continuous by the definition of the weak topology. Thus, \(\{K \in \mathcal{M}^f (\bR^d) \colon \int g_n d K < 1/m\}\) is open and consequently, \(H\) is a \(G_\delta\) set. We now explain that 
	\(
	 \mathcal{M}^f_0 (\bR^d) = H.
	\)
	If \(K \in \mathcal{M}^f_0 (\bR^d)\), then 
	\[
	\lim_{m \to \infty} \int g_md K = K (\{0\}) = 0,
	\]
	by the dominated convergence theorem. Hence, \(K \in H\). 
	Conversely, take \(K \in H\). For every \(n \in \mathbb{N}\), there exists an \(m \in \mathbb{N}\) such that 
	\(
	\int g_m d K < 1/n.
	\)
	As \(k \mapsto g_k\) decreases, we get 
	\(
	\int g_k d K < 1/n
	\)
	for all \(k \geq m\). This proves that \(\lim_{m \to \infty} \int g_m d K = 0\) and it follows from the dominated convergence theorem that \(K (\{0\}) = 0\). In summary, \(K \in  \mathcal{M}^f_0 (\bR^d)\), which completes the proof of \(\mathcal{M}^f_0 (\bR^d) = H\).
	
	\smallskip
	{\em Step 2.}
	Let \(\mathsf{d}\) be a metric that induces the weak topology on \(\mathcal{M}^f_0 (\bR^d)\). Then, 
	\[
	(K_1, K_2) \mapsto \mathsf{d}_{\mathcal{L}} (K_1, K_2) \triangleq \mathsf{d} ( g d K_1, g d K_2), \quad g (x) \triangleq 1 \wedge \|x\|^2, 
	\]
	induces the topology of \(\mathcal{L}\). 
	In other words, the function
	\(
	K \mapsto g d K
	\)
	is an isometry between \(\mathcal{L}\) and \(\mathcal{M}^f_0 (\bR^d)\). Further, it is a bijection whose inverse is given by 
	\(
	K \mapsto \1_{\bR^d \backslash \{0\}} dK / g.
	\)
	Hence, \(\mathcal{L}\) and \(\mathcal{M}^f_0 (\bR^d)\) are isometric, which implies that \(\mathcal{L}\) is a Polish space.
\end{proof}

We also identify a class of continuous functions on \(\mathcal{L}\).
\begin{lemma} \label{lem: continuity in L}
	Let \(g \colon \mathbb{R}^d \to \mathbb{R}\) be a continuous function such that \(|g (x)| \leq C (1 \wedge \|x\|^2)\) for some constant \(C > 0\). Then, the map 
	\[
	K \mapsto \int g (x) K (dx)
	\]
	is continuous from \(\mathcal{L}\) into \(\mathbb{R}\). 
\end{lemma}
\begin{proof}
	
	Suppose that \(\K^n \to \K^0\) in \(\mathcal{L}\). Then, by definition of the topology on \(\mathcal{L}\), \(G^n \triangleq (1 \wedge \|x\|^2) K^n (dx) \to (1 \wedge \|x\|^2) K^0 (dx) \triangleq G^0\) in \(\mathcal{M}^f (\bR^d)\). As \(G^0 (\{0\}) = 0\), the function 
	\[
	f (x) \triangleq \begin{cases} g (x) / (1 \wedge \|x\|^2), & x \not = 0, \\ 0, & x = 0, \end{cases}
	\]
	is bounded and \(G^0\)-a.e. continuous. Hence, the continuous mapping theorem for \(\mathcal{M}^f (\bR^d)\) (see \cite[Theorem~1.8] {prok}) yields that 
	\(
	\int f \, d G^n \to \int f \, d G^0.
	\)
	By definition of \(f\), this yields the claim. 
\end{proof}

\end{document}